\documentclass[a4paper]{amsart}
\usepackage[a4paper,outer=3.5cm,inner=3.5cm,total={14cm,21.4cm}]{geometry}

\usepackage[english]{babel}
\usepackage[utf8]{inputenc}
\usepackage{amsmath}
\usepackage{enumerate}
\usepackage{graphicx}
\usepackage{amssymb,amsxtra}
\usepackage{courier}
\usepackage{listings}
\usepackage{mathtools}
\RequirePackage{amsthm}
\RequirePackage{tikz-cd}
\RequirePackage{mathrsfs}
\RequirePackage[colorinlistoftodos]{todonotes}
\usepackage{amsthm}
\usepackage{amsfonts}
\usepackage{amssymb}
\usepackage{amscd}
\usepackage{epsf}
\usepackage{verbatim}
\usepackage{latexsym}
\usepackage{stackengine}
\usepackage{adjustbox}
\usepackage{thmtools}
\usepackage{thm-restate}
\stackMath
\usepackage{array}
\usepackage{lscape}
\usepackage[colorlinks=true]{hyperref}
\hypersetup{colorlinks, citecolor=blue, filecolor=black, linkcolor=black, urlcolor=black}
\usepackage{epstopdf}
\usepackage{tikz}
\usetikzlibrary{calc}
\usetikzlibrary{matrix,arrows,decorations.pathmorphing}
\usepackage{tikz-cd}
\usepackage{color}
\usepackage{geometry}
\usepackage{multirow}
\usepackage{enumitem}

\usepackage{framed}
\usepackage{stmaryrd} 

\usepackage{calligra} 

\newtheorem{thm}{Theorem}[section]
\newtheorem{Theorem}[thm]{Theorem}
\newtheorem{Lemma}[thm]{Lemma}

\newtheorem{Proposition}[thm]{Proposition}

\newtheorem{Corollary}[thm]{Corollary}

\theoremstyle{definition}

\newtheorem{Definition}[thm]{Definition}
\newtheorem{question}[thm]{Question}

\newtheorem{Example}[thm]{Example}

\newtheorem{Remark}[thm]{Remark}
\newtheorem{notation}[thm]{Notation}

\newcommand{\smallmat}[4]{\left(\begin{smallmatrix} #1 & #2 \\ #3 & #4\end{smallmatrix}\right)}

\newcommand*\isomarrow{%
	\xrightarrow{\raisebox{-0.35em}{\smash{\ensuremath{\sim}}}}
}

\newcommand\restr[2]{{
		\left.\kern-\nulldelimiterspace 
		#1 
		\right|_{#2} 
}}

\newcommand{\R}{\mathbb{R}}
\newcommand{\N}{\mathbb{N}}
\newcommand{\Z}{\mathbb{Z}}
\newcommand{\Q}{\mathbb{Q}}
\newcommand{\A}{\mathbb{A}}

\newcommand{\F}{\mathbb{F}}

\newcommand{\Tate}{\mathrm{T}}

\newcommand{\m}{\mathfrak{m}}

\renewcommand{\P}{\mathbb{P}}
\renewcommand{\O}{\mathcal{O}}

\newcommand{\Hom}{\operatorname{Hom}}
\newcommand{\Map}{\operatorname{Map}}
\newcommand{\Mapc}{\operatorname{Map}_{\cts}}

\newcommand{\GL}{\operatorname{GL}}

\newcommand{\Spec}{\operatorname{Spec}}

\newcommand{\Spa}{\operatorname{Spa}}
\newcommand{\Spf}{\operatorname{Spf}}

\newcommand{\cts}{{\operatorname{cts}}}

\newcommand{\an}{\mathrm{an}}

\newcommand{\rel}{\mathrm{rel}}

\newcommand{\perf}{{\operatorname{perf}}}

\newcommand{\hotimes}{\hat{\otimes}}

\newcommand{\HT}{\operatorname{HT}}

\newcommand{\X}{\mathcal X}

\newcommand{\Y}{\mathcal Y}
\newcommand{\E}{\mathcal E}
\newcommand{\Xa}{\mathcal X^{\ast}}

\newcommand{\XpaIne}{\mathcal X'^{\ast}_{\Ig(p^n)}(\epsilon)}
\newcommand{\XpaIie}{\mathcal X'^{\ast}_{\Ig(p^\infty)}(\epsilon)}

\newcommand{\mX}{\mathfrak X}

\newcommand{\mXa}{\mathfrak X^{\ast}}

\newcommand{\Xpa}{\mathcal X'^{\ast}}

\newcommand{\Xpae}{\mathcal X'^{\ast}(\epsilon)}
\newcommand{\Xpaep}{\mathcal X'^{\ast}(\epsilon)^\perf}
\newcommand{\Xpep}{\mathcal X'(\epsilon)^\perf}

\newcommand{\mXpa}{\mathfrak X'^{\ast}}

\newcommand{\mXp}{\mathfrak X'}

\newcommand{\XaGzoea}{\mathcal X^{\ast}_{\Gamma_0(p)}(\epsilon)_a}

\newcommand{\XaGea}[2]{\mathcal X^{\ast}_{\Gamma_{#1}(p^{#2})}(\epsilon)_a}

\newcommand{\XGea}[2]{\X_{\Gamma_{#1}(p^{#2})}(\epsilon)_a}

\newcommand{\XaGoiea}{\mathcal X^{\ast}_{\Gamma_1(p^\infty)}(\epsilon)_a}

\newcommand{\XpaInep}{\mathcal X'^{\ast}_{\Ig(p^n)}(\epsilon)^{\perf}}
\newcommand{\XpaIiep}{\mathcal X'^{\ast}_{\Ig(p^\infty)}(\epsilon)^{\perf}}

\newcommand{\Ig}{\operatorname{Ig}}

\newcommand{\Xp}{\mathcal X'}

\newcommand{\Ha}{\mathrm{Ha}}
\newcommand{\ord}{\mathrm{ord}}

\newcommand{\ad}{\mathrm{ad}}
\newcommand{\tr}{\operatorname{tr}}

\renewcommand{\O}{\mathcal O}

\newcommand{\smallvector}[2]{\left(\begin{smallmatrix}#1\\#2\end{smallmatrix}\right)}

\tikzset{
	labelrotate/.style={anchor=south, rotate=90, inner sep=.5mm}} 
\tikzset{
	labelrotatep/.style={anchor=north, rotate=90, inner sep=.75mm}}

\usepackage{relsize}
\usepackage[bbgreekl]{mathbbol}
\usepackage{amsfonts}
\DeclareSymbolFontAlphabet{\mathbb}{AMSb} 
\DeclareSymbolFontAlphabet{\mathbbl}{bbold}

\newcommand{\gd}{\mathrm{gd}}
\newcommand{\red}{\mathrm{red}}
\newcommand{\lauc}[1]{\langle \!\langle #1 \rangle\! \rangle}
\newcommand{\bb}[1]{\llbracket #1  \rrbracket}
\newcommand{\cc}[1]{(\!( #1  )\!)}
\newcommand{\D}{\mathcal D}

\title{Cusps and \texorpdfstring{\MakeLowercase{$q$}}{\MakeLowercase{q}}-expansion principles for modular curves at infinite level}
\author{Ben Heuer}
\date{}
\begin{document}
\maketitle
\begin{abstract}
	We develop an analytic theory of cusps for Scholze's $p$-adic modular curves at infinite level  in terms of perfectoid parameter spaces for Tate curves.  As an application, we describe a canonical tilting isomorphism between an anticanonical overconvergent neighbourhood of the ordinary locus of the modular curve at level $\Gamma_1(p^\infty)$ and the analogous locus of an infinite level perfected Igusa variety. We also prove various $q$-expansion principles for functions on modular curves at infinite level, namely that the properties of extending to the cusps, vanishing, coming from finite level, and being bounded, can all be detected on $q$-expansions.
\end{abstract}

\setcounter{tocdepth}{2}
\section{Introduction}
\subsection{Cusps of modular curves at infinite level}
Let $p$ be a prime and let $K$ be a perfectoid field extension of $\Q_p$. Throughout we shall assume that $K$ contains all $p^n$-th unit roots for all $n\in \N$. Let $N\geq 3$ be coprime to $p$ and let $\Xa$ be the compactified modular curve over $K$ of some rigidifying tame level $\Gamma^p$ such that $\Gamma(N)\subseteq \Gamma^p\subseteq \GL_2(\Z/N\Z)$. Here we consider $\mathcal X^{\ast}$ as an analytic adic space.  

The first goal of this article is to give a detailed analytic description of the geometry at the cusps in the inverse system of modular curves with varying level structures at $p$, as well as for the  modular curves at infinite level introduced by Scholze in \cite{torsion}. In doing so, we aim to complement results on the boundary of infinite level Siegel varieties for $\mathrm{GSp}_{2g}$ for $g\geq 2$ from \cite[\S III.2.5]{torsion}, proved there using machinery like Hartog's extension principle and a perfectoid version of Riemann's Hebbarkeitssatz: Due to assumptions on the codimension of the boundary to be $\geq 2$, these tools do not apply in the elliptic case. Instead, in this case one can get a much more explicit description with more elementary means.

The way we study the boundary in the elliptic case is in terms of adic analytic parameter spaces for Tate curves. Fix a cusp $x$ of $\X^\ast$, this is in general a point defined over a finite field extension $K\subseteq L_x\subseteq K[\zeta_N]$ depending on $x$, and it corresponds to a $\Gamma^p$-level structure on the Tate curve $\Tate(q^{e_x})$ over $\O_{L_x}\cc{q}$ for some $1\leq e_x\leq N$. The analytic Tate curve parameter space in this situation is simply the adic open unit disc $\D_x$ over $L_x$, and there is a canonical open immersion $\D_x\hookrightarrow \mathcal X^{\ast}$ that sends the origin to $x$.

In order to state our main result, let us recall the tower of anticanonical moduli spaces from \cite[\S3 III]{torsion}: Away from the cusps, the modular curve $\X^{\ast}$ is the moduli space of elliptic curves $E$ together with a $\Gamma^p$-level structure. 
Let $\X^{\ast}_{\Gamma_0(p)}\to \X^{\ast}$ be the finite flat cover that relatively represents (away from the cusps) the data of a cyclic subgroup scheme of rank $p$ of $E[p]$.
By the theory of the canonical subgroup, for small enough $\epsilon>0$, the $\epsilon$-overconvergent neighbourhood of the ordinary locus $\X^{\ast}_{\Gamma_0(p)}(\epsilon)\subseteq \X^{\ast}_{\Gamma_0(p)}$ decomposes into two disjoint open components: the canonical locus $\X^{\ast}_{\Gamma_0(p)}(\epsilon)_c$ and the anticanonical locus $\X^{\ast}_{\Gamma_0(p)}(\epsilon)_a$. In order to understand the cusps of the perfectoid modular curve at infinite level, we first study the tower of compactified  modular curves of higher level at $p$ relatively over $\X^{\ast}_{\Gamma_0(p)}(\epsilon)_a$. Specifically, for any $n\in \N$, the pullback of $\mathcal X^{\ast}_{\Gamma_0(p)}(\epsilon)_a\subseteq \mathcal X^{\ast}_{\Gamma_0(p)}$ defines a tower
\[\mathcal X^{\ast}_{\Gamma(p^n)}(\epsilon)_a \to \mathcal X^{\ast}_{\Gamma_1(p^n)}(\epsilon)_a  \to \mathcal X^{\ast}_{\Gamma_0(p^n)}(\epsilon)_a \to\mathcal X^{\ast}(\epsilon)\]
 of anticanonical loci. Here the rightmost map is finite flat and totally ramified at the cusps, whereas the map $\mathcal X^{\ast}_{\Gamma(p^n)}(\epsilon)_a \to \mathcal X^{\ast}_{\Gamma_0(p^n)}(\epsilon)_a$ is finite Galois with group
\[ \Gamma_0(p^n,\Z/p^n\Z):=\{\smallmat{\ast}{\ast}{0}{\ast}\in \GL_2(\Z/p^n\Z)\}.\]

The cusps of these moduli spaces of higher finite level can be described using analogous parameter spaces for Tate curves: As we shall discuss in detail in \S2, it is essentially an adic analytic version of the classical calculus of cusps of Katz--Mazur \cite{KatzMazur} that for any cusp $x$ of $\mathcal X^{\ast}$, there are Cartesian diagrams of adic spaces of topologically finite type over $K$
\begin{equation*}
	\begin{tikzcd}
		\underline{\Gamma_0(p^n,\Z/p^n\Z)}\times \D_{n,x} \arrow[d, hook] \arrow[r] & \underline{(\Z/p^n\Z)^{\times}}\times \D_{n,x} \arrow[d, hook] \arrow[r] & \mathbb \D_{n,x} \arrow[r] \arrow[d, hook] & \D_x \arrow[d, hook] \\
		\mathcal X^{\ast}_{\Gamma(p^n)}(\epsilon)_a \arrow[r] & \mathcal X^{\ast}_{\Gamma_1(p^n)}(\epsilon)_a \arrow[r] & \mathcal X^{\ast}_{\Gamma_0(p^n)}(\epsilon)_a \arrow[r] & \mathcal X^{\ast}(\epsilon),
	\end{tikzcd}
\end{equation*}
where the top left map is $\smallmat{a}{b}{0}{d}\mapsto d$ in the first component and the identity in the second, and where $\D_{n,x}$ is the open unit disc in the variable $q^{1/p^n}$ over $L_x$,
\[\D_{n,x}= \text{ open subspace of }\Spa(L_x\langle q^{1/p^n}\rangle,\O_{L_x}\langle q^{1/p^n}\rangle) \text{ defined by }|q|<1.\]
In the limit $n\to \infty$, these open discs become parameter spaces for Tate curves with infinite $\Gamma_0$-level structure at $p$, given by perfectoid open unit discs 
\[\D_{\infty,x}= \text{ open subspace of }\Spa(L_x\langle q^{1/p^\infty}\rangle,\O_{L_x}\langle q^{1/p^\infty}\rangle) \text{ defined by }|q|<1.\]
We then get the following description of the cusps at infinite level: Let 
\[\Gamma_0(p^\infty):=\{\smallmat{\ast}{\ast}{0}{\ast}\}\subseteq \GL_2(\Z_p)\]
and let $\underline{\Gamma_0(p^\infty)}$ be the associated profinite perfectoid space.
\begin{Theorem}\label{t:cusps of X_Gamma(p^infty)}
	Let $x$ be any cusp of $\X^{\ast}$.
	\begin{enumerate}
	\item There is a Cartesian diagram of perfectoid spaces over $K$
\begin{equation*}
	\begin{tikzcd}
		\underline{\Gamma_0(p^\infty)}\times \D_{\infty,x} \arrow[d,hook] \arrow[r] & \underline{\Z_p^\times}\times \D_{\infty,x} \arrow[d,hook] \arrow[r] &  \D_{\infty,x} \arrow[d,hook] \arrow[r] & \D_x\arrow[d,hook]\\
		\mathcal X^{\ast}_{\Gamma(p^\infty)}(\epsilon)_a \arrow[r] &  \X^{\ast}_{\Gamma_1(p^\infty)}(\epsilon)_a \arrow[r] & \mathcal X^{\ast}_{\Gamma_0(p^\infty)}(\epsilon)_a \arrow[r]&\mathcal X^{\ast}(\epsilon).
	\end{tikzcd}
\end{equation*}
		\item Define a right action of $\Z_p$ on $\underline{\GL_2(\Z_p)}\times \D_{\infty,x}$ by $(\gamma,q^{1/p^n})\cdot h\mapsto (\gamma\smallmat{1}{0}{h}{1},q^{1/p^n}\zeta^{h/e_x}_{p^n})$. Let $\GL_2(\Z_p)$ act on the left via the first factor. Then there is a Cartesian diagram
		\begin{equation*}
			\begin{tikzcd}
				(\underline{\GL_2(\Z_p)}\times \D_{\infty,x})/\Z_p \arrow[d,hook] \arrow[r] & \D_x\arrow[d,hook]\\
				\mathcal X^{\ast}_{\Gamma(p^\infty)} \arrow[r] &\mathcal X^{\ast}
			\end{tikzcd}
		\end{equation*}
		for which the left map is a $\GL_2(\Z_p)$-equivariant open immersion.
		\item The Hodge--Tate period map $\pi_{\HT}\colon \mathcal X^{\ast}_{\Gamma(p^\infty)} \rightarrow \P^1$ restricts to the locally constant map
		\[
				(\underline{\GL_2(\Z_p)}\times \D_{\infty,x})/\Z_p  \to\underline{\P^1(\Z_p)},\quad \big(\!\smallmat{a}{b}{c}{d},q\big)\mapsto (b:d).
		\]
	\end{enumerate}
\end{Theorem}
We refer to Thm.~\ref{Theorem: Tate parameter spaces at level Gamma(p^infty)} and Thm.~\ref{t:Main-Theorem-parts-2-3} for slightly more precise statements.
In other words, the part of the boundary of $\mathcal X^{\ast}_{\Gamma(p^\infty)}$ lying over $x$ is given by the closed profinite subspace \[\underline{\GL_2(\Z_p)/\smallmat{1}{0}{\Z_p}{1}}\hookrightarrow \mathcal X^{\ast}_{\Gamma(p^\infty)}\]
and has an open neighbourhood given by a Tate curve  parameter space $(\underline{\GL_2(\Z_p)}\times \D_{\infty,x})/\Z_p$.

We will prove part (1) of the theorem step by step in \S\ref{s:infinite-level-Gamma_0}-\ref{s:infinite-level-Gamma}.
We then deduce (2) from (1) via $\GL_2(\Z_p)$-translations. For this one needs to describe the action of the larger group 
\[\Gamma_0(p):=\{\smallmat{\ast}{\ast}{c}{\ast}|c\in p\Z_p\}\subseteq \GL_2(\Z_p)\]
on the Tate curve parameter space $\underline{\Gamma_0(p^\infty)}\times \D_{\infty,x}\hookrightarrow \mathcal X^{\ast}_{\Gamma(p^\infty)}(\epsilon)_a $, which also takes into account isomorphisms of Tate curves of the form $q\mapsto \zeta^h_{p^n}q$ for $h\in\Z_p$. We will do so in \S\ref{s:action-of-Gamma_0}.

\begin{Remark}
	We note that  Pilloni--Stroh \cite{pilloni2016cohomologie} in their construction of perfectoid tilde-limits of toroidal compactifications of Siegel moduli varieties also describe the boundary of $\mathcal X^{\ast}_{\Gamma(p^\infty)}$: 
	More precisely, the second part of Thm.~\ref{t:cusps of X_Gamma(p^infty)} also follows from \cite[Prop. A.14]{pilloni2016cohomologie}. While their proposition is much more general, the above description is arguably slightly more explicit. We will also identify the canonical and anticanonical subspaces.
\end{Remark}

On the way, we discuss in \S2 some aspects of modular curves as analytic adic space that are not visible in the rigid setting as treated in \cite{conrad2006modular}. For example, in the adic setting there is also a larger analytic Tate curve parameter space $\overline{\D}_x\to \X^{\ast}$ of the form 
\[\overline{\D}_x:=\Spa(\O_{L_x}\bb q[\tfrac{1}{p}],\O_{L_x}\bb q)\]
 where $\O_{L_x}\bb q$ is endowed with the $p$-adic topology (rather than the $(p,q)$-adic one). This gives rise at infinite level to a map $\overline{\D}_{\infty,x}\to \X_{\Gamma_0(p^\infty)}^{\ast}(\epsilon)_a$ where 
\[\overline{\D}_{\infty,x}:=\Spa(\O_{L_x}\bb {q^{1/p^\infty}}_p[\tfrac{1}{p}],\O_{L_x}\bb{q^{1/p^\infty}}_p)\sim \textstyle\varprojlim_{q\mapsto q^p} \overline{\D}_x.\]
Here $\O_{L_x}\bb{q^{1/p^\infty}}_p$ is the $p$-adic completion of $\varinjlim_n \O_{L_x}\bb{q^{1/p^n}}$. 

While these are no longer open immersions, they are sometimes useful, for example because in contrast to $\D_x$, the spaces $\overline{\D}_x$ for all $x$ together with the good reduction locus $\X_{\gd}$ cover the adic space $\X^{\ast}$. More precisely, we have $\overline{\D}\setminus \D= \Spa(\O_{L_x}\lauc{q}[\tfrac{1}{p}],\O_{L_x}\lauc{q}^+)$ where $\O_{L_x}\lauc{q}$ is the $p$-adic completion of $\O_{L_x}\bb{q}[q^{-1}]$, a local ring, and $\O_L\lauc{q}^+$ is a certain valuation subring of rank 2. The image of this rank $2$ point in $\mathcal X^{\ast}$ is a closed point that is neither contained in $\X_{\gd}$ nor in $\D_x$. We discuss this situation in more detail in \S\ref{s:points-of-Xast}.

\subsection{cusps of perfectoid modular curves in characteristic $p$}
There are natural analogues of the above descriptions for modular curves in characteristic $p$, which we treat in \S\ref{s:cusps-in-char-p}: We shall work over the perfectoid field $K^{\flat}$ and choose $\varpi^{\flat}\in K^{\flat}$ with $|\varpi^{\flat}|=|p|$. Let $\X'^{\ast}$ be the compactified modular curve of level $\Gamma^p$ over $K^{\flat}$, considered as an analytic adic space. Then, again, for every cusp $x$ of $\mathcal X'^{\ast}$, there is a Tate parameter space $\D'_x\hookrightarrow \X'^{\ast}$ where  now $\D'_x$ is the adic open unit disc over a finite cyclotomic extension $L_x^{\flat}\subseteq K^{\flat}[\zeta_N]$ (here the notation as the tilt of an extension of $K$ corresponding to $x$ will be justified later).

Recall that over any overconvergent neighbourhood $\X^{\ast}(\epsilon)$ of the locus of ordinary reduction, there is a finite \'etale Igusa curve $\X'^{\ast}_{\Ig(p^n)}(\epsilon)\to \X'^{\ast}(\epsilon)$. In the limit over the relative Frobenius morphism, and over $n\to \infty$, these give rise to a pro-\'etale morphism of perfectoid spaces over  $\X'^{\ast}(\epsilon)$
\[\X'^{\ast}_{\Ig(p^\infty)}(\epsilon)^{\perf}\to \X'^{\ast}(\epsilon)^{\perf}.\]
Let now $\D_{\infty,x}'$ denote the perfectoid open unit disc over $L_x^{\flat}$, which is the perfection of $\D_x'$. Then we have the following analogue of Thm.~\ref{t:cusps of X_Gamma(p^infty)} in characteristic $p$:
\begin{Theorem}
	For every cusp $x$ of $\X'^{\ast}$, there are Cartesian diagrams
	\begin{center}
	\begin{tikzcd}
	\underline{\Z_p^\times}\times \D_{\infty,x}' \arrow[d,hook] \arrow[r] & \D'_{\infty,x} \arrow[d,hook]\arrow[r]&\D'_{x} \arrow[d,hook] \\
	\XpaIiep \arrow[r] & \Xpaep \arrow[r] & \Xpae.
	\end{tikzcd}
	\end{center}
\end{Theorem}
We then compare this diagram to the situation in characteristic $0$ via tilting:
\subsection{A tilting isomorphism at level $\Gamma_1(p^\infty)$}
In \cite[Cor.~III.2.19]{torsion}, Scholze identifies the tilt of $\mathcal X^{\ast}_{\Gamma_0(p^\infty)}(\epsilon)_a$ by proving that there is a canonical isomorphism
\begin{equation}\label{eq:tilting the Gamma_0 tower}
 \mathcal X^{\ast}_{\Gamma_0(p^\infty)}(\epsilon)_a^{\flat} \isomarrow\mathcal X'^{\ast}(\epsilon)^{\perf}
 \end{equation}
of perfectoid spaces  over $K^\flat$. In the case of Siegel spaces parametrising abelian varieties of dimension $g\geq 2$, he then extends this tilting isomorphism to level $\Gamma_1(p^\infty)$. 

Using Tate curve parameter spaces, we complement this result in \S\ref{s:tilting-isomorphism} by treating the case $g=1$ of elliptic curves. Moreover, we work out the precise situation at the cusps: It follows from \eqref{eq:tilting the Gamma_0 tower} that the cusps of $\X^{\ast}$ (which can be identified with those of $\mathcal X^{\ast}_{\Gamma_0(p^\infty)}(\epsilon)_a$) and the cusps of $\mathcal X'^\ast$ (which can be identified with those of $\mathcal X'^{\ast}(\epsilon)^{\perf}$) can be identified via tilting, and the same is true for the field extensions $L_x$ and $L^{\flat}_x$. Using these identifications, we have:

\begin{Theorem}\label{t: tilting the Gamma_1(p^infty)-tower}
	\begin{enumerate}
		\item There is a canonical isomorphism 
\[\mathcal X^{\ast}_{\Gamma_1(p^\infty)}(\epsilon)_a^{\flat} \isomarrow \mathcal X'^{\ast}_{\Ig(p^\infty)}(\epsilon)^{\perf}\]
which is $\Z_p^\times$-equivariant and makes the following diagram commute:
\begin{equation*}
\begin{tikzcd}[row sep = 0.15cm]
\mathcal X^{\ast}_{\Gamma_1(p^\infty)}(\epsilon)_a^{\flat} \arrow[d,"\sim"labelrotate,equal] \arrow[r] & \mathcal X^{\ast}_{\Gamma_0(p^\infty)}(\epsilon)_a^{\flat}\arrow[d,"\sim"labelrotate,equal] \\ \mathcal X'^{\ast}_{\Ig(p^\infty)}(\epsilon)^{\perf}  \arrow[r] & \mathcal X'^{\ast}(\epsilon)^{\perf}.
\end{tikzcd}
\end{equation*}
\item For any cusp $x$ of $\X^{\ast}$ with corresponding cusp $x^{\flat}$ of $\X'^{\ast}$, the diagram
\begin{equation*}
\begin{tikzcd}[row sep = 0.15cm]
\underline{\Z_p^{\times}}\times \D_{\infty,x}^{\flat} \arrow[d,"\sim"labelrotate,equal] \arrow[r,hook] & 
\XaGoiea^{\flat} \arrow[d,"\sim"labelrotate,equal]\\
\underline{\Z_p^{\times}}\times \D'_{\infty,x^{\flat}}\arrow[r,hook]& \mathcal X'^{\ast}_{\Ig(p^\infty)}(\epsilon)^{\perf}
\end{tikzcd}
\end{equation*}
 commutes, where the left map is given by the canonical identification $\D_{\infty,x}^{\flat}\cong  \D'_{\infty,x^{\flat}}$. 
	\end{enumerate}
\end{Theorem}
We are interested in this result because of an application to $p$-adic modular forms: In \cite{heuer-thesis}, we use Thm~\ref{t: tilting the Gamma_1(p^infty)-tower} to give a perfectoid perspective on the $t$-adic modular forms at the boundary of weight space of Andreatta--Iovita--Pilloni.

\subsection{\texorpdfstring{$q$}{q}-expansion principles}
As a second application, Tate curve parameter spaces give a way to talk about $q$-expansions of functions on modular curves:
For any $f\in \O(\X^{\ast}_{\Gamma(p^\infty)}(\epsilon)_a)$, we may define the $q$-expansion  of $f$ at a cusp $x\in \X^{\ast}_{\Gamma(p^\infty)}(\epsilon)_a$ to be its restriction to the associated open subspace $\D_{\infty,x}\hookrightarrow \X^{\ast}_{\Gamma(p^\infty)}(\epsilon)_a$, i.e.\ the image under
\[\O(\X^{\ast}_{\Gamma(p^\infty)}(\epsilon)_a)\to \O(\D_{\infty,x})\]
which will automatically lie in $\O_K\bb{q^{1/p^\infty}}[\tfrac{1}{p}]\subseteq \O(\D_{\infty,x})$. One has analogous definitions for other infinite level modular curves, or open subspaces thereof, as well as for profinite families of cusps.
Such $q$-expansions can be useful when working with modular curves at infinite level, as they often allow one to extend constructions which are a priori defined only away from the cusps (or even just on the good reduction locus),  for instance maps defined using moduli functors, to the compactifications.  For example:

\begin{Lemma}
A function $f$ on the uncompactified modular curve $\mathcal X_{\Gamma(p^\infty)}(\epsilon)_a$ extends to a function on  $\mathcal X^{\ast}_{\Gamma(p^\infty)}(\epsilon)_a$ if and only if at every cusp $x$ of  $\mathcal X^{\ast}_{\Gamma(p^\infty)}(\epsilon)_a$, the $q$-expansion of $f$ is already contained in $\O_{L_x}\llbracket q^{1/p^\infty}\rrbracket[\frac{1}{p}]\subseteq \O_{L_x}\lauc {q^{1/p^\infty}}[\frac{1}{p}]$. Any such extension is unique.
\end{Lemma}
This is what we mean when we say that $q$-expansions can be used as a replacement for Hartog's extension principle in the elliptic case of $g=1$.
An instance of such an application of Tate parameter spaces in the context of $p$-adic modular forms can be found in \cite{ChrisChris-HilbertCHJ}.

Finally in this article, we show in \S\ref{s:q-expansion-principles} that in the spirit of Katz' $q$-expansion principle for modular forms \cite[Thm.~1.6.1]{p-adicMSMF}, one can use Tate parameter spaces to prove various $q$-expansion principles for functions on perfectoid modular curves. 

\begin{Proposition}[$q$-expansion principle I]\label{p: q-expansion principle I}
	Let $\mathcal C$ be a collection of cusps of $\mathcal X^{\ast}$ such that each connected component of $\mathcal X^{\ast}$ contains at least one $x\in \mathcal C$. Then restriction of functions to the Tate curve parameter spaces associated to $\mathcal C$ defines injective maps
	\begin{alignat*}{4}
	&\O(\XaGea{0}{\infty})&\hookrightarrow& \textstyle\prod_{x \in \mathcal C}\O_{L_x}\llbracket q^{1/p^\infty}\rrbracket[\tfrac{1}{p}],\\
	&\O(\XaGea{1}{\infty})&\hookrightarrow& \textstyle\prod_{x \in \mathcal C}\Mapc(\Z_p^\times,\O_{L_x}\llbracket q^{1/p^\infty}\rrbracket)[\tfrac{1}{p}],\\
	&\O(\XaGea{}{\infty})&\hookrightarrow& \textstyle\prod_{x \in \mathcal C}\Mapc(\Gamma_0(p^\infty),\O_{L_x}\llbracket q^{1/p^\infty}\rrbracket)[\tfrac{1}{p}],\\
	&\O(\Xpaep)&\hookrightarrow& \textstyle\prod_{x \in \mathcal C}\O_{L_x^{\flat}}\llbracket q^{1/p^\infty}\rrbracket[\tfrac{1}{\varpi^{\flat}}],\\
	&\O(\XpaIiep)&\hookrightarrow& \textstyle\prod_{x \in \mathcal C}\Mapc(\Z_p^\times,\O_{L_x^{\flat}}\llbracket q^{1/p^\infty}\rrbracket)[\tfrac{1}{\varpi^{\flat}}].
	\end{alignat*}
\end{Proposition}
We note that $p$-adic modular forms can be described as functions on $\mathcal X^{\ast}_{\Gamma_0(p^\infty)}(\epsilon)_a$ satisfying a certain equivariance property, see \cite{CHJ}, \cite{ChrisChris-HilbertCHJ}. Prop.~\ref{p: q-expansion principle I} may thus be seen as a generalisation of its classical version from modular forms to more general functions.

Similarly, one can  detect on $q$-expansions whether a function comes from finite level:

\begin{Proposition}[$q$-expansion principle II]\label{p: q-expansion principle II}
	Let  $f\in \O(\XaGea{0}{\infty})$. Then for any $n\in \Z_{\geq 0}$, the following are equivalent:
	\begin{enumerate}
		\item $f$ comes via pullback from $\XaGea{0}{n}$, i.e.\  $f\in\O(\XaGea{0}{n})\subseteq \O(\XaGea{0}{\infty})$.
		\item The $q$-expansion of $f$ at every cusp $x$ is contained in $\O_{L_x}\llbracket q^{1/p^n}\rrbracket[\frac{1}{p}]\subseteq \O_{L_x}\llbracket q^{1/p^\infty}\rrbracket[\frac{1}{p}]$.
		\item On each connected component of $\XaGea{0}{n}$, there is at least one cusp $x$ at which the $q$-expansion of $f$ is already in ${\O_{L_x}}\llbracket q^{1/p^n}\rrbracket[\frac{1}{p}]\subseteq \O_{L_x}\llbracket q^{1/p^\infty}\rrbracket[\frac{1}{p}]$.
	\end{enumerate}
	The analogous statements for $\Xpaep$ are also true.
\end{Proposition}
For $\epsilon=0$, i.e.\ on the ordinary locus, one can see on $q$-expansions whether a function is integral, i.e.\ bounded by 1. We note, however, that this fails for $\epsilon>0$.
\begin{Proposition}[$q$-expansion principle III]\label{p: q-expansion principle III}
	For $f\in \O(\mathcal X^{\ast}_{\Gamma_0(p^\infty)}(0)_a)$ are equivalent:
	\begin{enumerate}
		\item	$f$ is integral, i.e.\ it is contained in $\O^+(\mathcal X^{\ast}_{\Gamma_0(p^\infty)}(0)_a)$. 	
		\item The $q$-expansion of $f$ at every cusp $x$ is already in $\O_{L_x}\llbracket q^{1/p^\infty}\rrbracket$.
		\item On each connected component of $\XaGea{0}{n}$, there is at least one cusp $x$ at which the $q$-expansion of $f$ is in $\O_{L_x}\llbracket q^{1/p^\infty}\rrbracket$.
	\end{enumerate}
	The analogous statements for $\mathcal X^{\ast}_{\Gamma_1(p^\infty)}(0)_a$, $\mathcal X^{\ast}_{\Gamma(p^\infty)}(0)_a$, $\mathcal X'^{\ast}(0)^{\perf}$ and $\mathcal X'^{\ast}_{\Ig(p^\infty)}(0)^\perf$ are also true when we replace $\O_{L_x}\llbracket q^{1/p^\infty}\rrbracket$ by the respective algebra in Prop.~\ref{p: q-expansion principle I}.
\end{Proposition}
Finally, there is also a version of $q$-expansions for the good reduction locus: 
\begin{Proposition}[$q$-expansion principle IV]\label{p: q-expansion principle IV}
	Let $\mathcal C$ be a collection of cusps of $\Xa$ such that each connected component contains at least one $x\in \mathcal C$. Then a function on the good reduction locus $\X_{\gd}(\epsilon)$ extends to all of $\X^{\ast}(\epsilon)$ if and only if its $q$-expansion with respect to $\overline{\D}(|q|\geq 1)\to \X_{\gd}(\epsilon)$ at each $x\in \mathcal C$ is already in $\O_{L_x}\llbracket q\rrbracket[\tfrac{1}{p}]\subseteq \O_{L_x}\lauc{q}[\tfrac{1}{p}]$. In this case, the extension is unique. The analogous statements for $\mathcal X^{\ast}_{\Gamma_0(p^\infty)}(\epsilon)_a$ $\mathcal X^{\ast}_{\Gamma_1(p^\infty)}(\epsilon)_a$, $\mathcal X^{\ast}_{\Gamma(p^\infty)}(\epsilon)_a$, $\X'^{\ast}(\epsilon)$, $\mathcal X'^{\ast}(\epsilon)^{\perf}$ and $\mathcal X'^{\ast}_{\Ig(p^\infty)}(\epsilon)^\perf$ are also true.
\end{Proposition}

\section*{Acknowledgements}
I would like to thank Johannes Ansch\"utz, Kevin Buzzard, Ana Caraiani, Vincent Pilloni and Peter Scholze for helpful discussions.
This work was carried out while the author was supported by the Engineering and Physical Sciences Research Council [EP/L015234/1], the EPSRC Centre for Doctoral Training in Geometry and Number Theory (The London School of Geometry and Number Theory), University College London.

\section{Adic analytic theory of cusps at finite level}\label{s:adic-cusps-finite}
\subsection{Recollections on the classical theory of cusps}\label{s:classical-recollections-on-cusps}
We start by recalling from \cite[\S8.6-8.11]{KatzMazur} some basic facts about cusps of modular curves  that we will use freely throughout, and fix some notation and conventions: 

Let $N\geq 3$ and let $R$ be an excellent Noetherian regular $\Z[1/N]$-algebra, for instance $R=\Z[1/N]$.  Let $X_R$ be the modular curve of some rigidifying level $\Gamma(N)\subseteq \Gamma\subseteq \GL_2(\Z/N\Z)$ over $R$.
By definition, the compactification $X_{R}^{\ast}$ of $X_{R}$ is then the normalisation in $\P_R^1$ of the finite flat $j$-invariant $j\colon X_{R}\to \A_R^1$ . The divisor of cusps is defined as the closed subscheme 
\[\partial X_{R}^{\ast} := (X_{R}^{\ast}\backslash X_{R})^{\red}=j^{-1}(\infty)^{\red}\hookrightarrow X_{R}^{\ast},\] which is finite \'etale over $\Spec(R)$. When we refer to ``a cusp'' we shall mean by this a (not necessarily geometrically) connected component of $\partial X_{R}^{\ast}$.

We recall from \cite[\S8.11]{KatzMazur} that the divisor of cusps can be computed explicitly using the Tate curve $\Tate(q)$: This is an elliptic curve over $\Z\cc{q}$ of $j$-invariant $1/q+744+\dots$, which we may base--change to $\Tate(q)_{R\cc{q}}\to \Spec(R\cc{q})$. Then we have:
\begin{Proposition}[{\cite[{Thm. 8.11.10}]{KatzMazur}}]\label{p:completion-at-cusp}
	The completion $\hat{\partial} X_{R}^{\ast}$ of $X_{R}^{\ast}$ along $\partial X_{R}^{\ast}$ is the normalisation of $R\llbracket q\rrbracket $ in the finite flat scheme over $R\cc{q}$ that represents $\Gamma$-level structures of $\Tate(q)_{R\cc{q}}$. Via the $j$-invariant, $\hat{\partial} X_{R}^{\ast}$ is finite over the completion $R\llbracket q\rrbracket $ of $\P_R^1$ at $\infty$.
\end{Proposition}
To say more concretely what $\partial X_{R}^{\ast}$ looks like, recall that $\Tate(q)[N]$ is an extension
\[0\to \mu_{N}\to \Tate(q)[N]\to \underline{\Z/N\Z}\to 0\]
over $\Z\cc{q}$ which becomes split over $\Z\cc{q^{1/N}}$. Consequently, the $\Gamma$-level structures of $\Tate(q)_{R\cc{q}}$ are defined over various subrings of $R[\zeta_N]\cc{q^{1/N}}$. In particular, each component of $\hat{\partial} X_{R}^{\ast}$ will be of the form $\Spf(R[\zeta_d]\llbracket q^{1/e}\rrbracket )$ for some $d|N$ and $e|N$.
\begin{notation}\label{n:e}
In order to lighten notation, we wish to reduce the amount of $N$-th roots of $q$ throughout. Therefore, we shall by convention renormalise this ring of definition at each cusp to be a subring of the form $R[\zeta_d]\cc{q}$ for some $d|N$, by passing from $\Tate(q)$ to $\Tate(q^{e})$.
\end{notation}
This means that depending on the cusp, the completion of the $j$-invariant at the cusp is now given by a map of the form $\Spf(R[\zeta_d]\llbracket q\rrbracket )\to \Spf(R\llbracket q\rrbracket )$ that sends $q\mapsto q^e$.
If $R'$ is another Noetherian excellent regular $\Z[1/N]$-algebra, then $\partial X_{R'}^{\ast}$, $\hat{\partial} X_{R'}^{\ast}$, etc.\ agree with the base-changes via $R\to R'$ \cite[Prop.\ 8.6.6]{KatzMazur}. Therefore, more generally, for any $\Z[1/N]$-algebra $S$ we may simply define $X_S$, $X^{\ast}_S$, $\partial X_S^{\ast}$, $\hat{\partial} X_S^{\ast}$, etc.\ by base-change. 
\subsection{The analytic setup}
We now switch to a $p$-adic analytic situation and recall the setup from \cite{torsion}.

Let $p$ be a prime and let $K$ be a perfectoid field extension of $\Q_p$ like in the introduction. We denote by $\mathfrak m$ the maximal ideal of the ring of integers $\O_K$ and by $k$ the residue field. We fix a complete algebraically closed extension $C$ of $K$ and assume that $K$ contains all $p$-power unit roots in $C$. We moreover fix a pseudo-uniformiser $\varpi\in K$ with $|\varpi|=|p|$ such that $\varpi$ contains arbitrary $p$-power roots in $K$. This is possible since $K$ is perfectoid.

Let $N$ be coprime to $p$ and let $\Gamma(N)\subseteq \Gamma^p\subseteq \GL_2(\Z/N\Z)$ be a rigidifying tame level. Let $X:=X_K$ be the modular curve of level $\Gamma^p$ over $K$ and let $X^{\ast}:=X^{\ast}_K$ be its compactification.
We denote by $\mathfrak X$ and $\mathfrak X^{\ast}$ the respective $p$-adic completions of $X_{\O_K}$ and $X^{\ast}_{\O_K}$. Let $\mathcal X$ and $\mathcal X^{\ast}$ be the respective adic analytifications of $X$ and $X^{\ast}$. This is the only way in which we deviate from the notation in \cite{torsion}, where $\X$ denotes the good reduction locus, which we shall instead denote by $\X_{\gd}\subseteq \X\subseteq \X^\ast$.

For any of the classical level structures $\Gamma=\Gamma_0(p^n),\Gamma_1(p^n),\Gamma(p^n)$, $n\in \N$, we denote by $X_\Gamma\to X$ the representing moduli scheme. These all have compactifications $X^{\ast}_{\Gamma}\to X^{\ast}$, and we have associated adic spaces $\mathcal X_{\Gamma}\to \mathcal X$ and  $\mathcal X^{\ast}_{\Gamma}\to \mathcal X^{\ast}$.
The uncompactified spaces have a natural moduli interpretation in the category $\mathbf{Adic}$ of adic spaces:
\begin{Lemma}\label{Lemma: adic moduli interpretation of X_Gamma^an}
	Let $S$ be an honest adic space over $\Spa(K,\mathcal O_K)$. Then
	\[\Hom_{\mathbf{Adic}}(S,\X_{\Gamma})=X(\mathcal O_S(S)).\]
	In particular, the $S$-points of $\X_{\Gamma}$ are in functorial correspondence with isomorphism classes of elliptic curves over $\mathcal O_S(S)$ with tame level structure $\Gamma^p$ and level structure $\Gamma$ at $p$.
\end{Lemma}
\begin{proof}
	The scheme $X_{\Gamma}$ over $K$ is an affine curve \cite[Cor.\ 4.7.2]{KatzMazur}. Let $\mathbf{LRS}$ be the category of locally ringed spaces, then by the universal property of the analytification $\X_\Gamma = X_{\Gamma}^{\an}$: \[\Hom_{\mathbf{Adic}}(S,\X_{\Gamma})=\Hom_{\mathbf{LRS}}(S,X_{\Gamma})=X_{\Gamma}(\O_S(S))\]
	where the last step is the adjunction of Spec and global sections for locally ringed spaces.
\end{proof}

Let $0\leq \epsilon<\tfrac{1}{2}$ be such that $|p|^{\epsilon}\in |K|$. Using local trivialisations of the Hodge bundle and lifts $\Ha$ of the Hasse invariant one defines an open subspace $\mathcal X^{\ast}(\epsilon)\subseteq \Xa$ cut out by the condition that $|\Ha|\geq |p|^{\epsilon}$.
This has a canonical integral model $\mathfrak X^{\ast}(\epsilon)\to \mathfrak X^{\ast}$, for example by \cite[Lemma III.2.13]{torsion}.
In general, for any morphism $S\to \mathcal X^{\ast}$ we shall write
\[S(\epsilon):=S\times_{\mathcal X^{\ast}}\mathcal X^{\ast}(\epsilon).\]
In particular, for any of the classical level structures $\Gamma=\Gamma_0(p^n),\Gamma_1(p^n),\Gamma(p^n)$ the modular curve $\mathcal X^{\ast}_{\Gamma}\to \X^{\ast}$ restricts to a morphism $\mathcal X^{\ast}_{\Gamma}(\epsilon)\to \X^{\ast}(\epsilon)$. 
We note that the open subspace $\X^{\ast}(0)$ is precisely the ordinary locus of $\Xa$. 

\begin{Definition}
	We shall say that an elliptic curve $E$ is $\epsilon$-nearly ordinary if $|\Ha(E)|\geq |p|^{\epsilon}$.
\end{Definition}
By the theory of the canonical subgroup, the forgetful morphism $\mathcal X^{\ast}_{\Gamma_0(p)}(\epsilon)\to \mathcal X^{\ast}(\epsilon)$ has a canonical section. We denote by $\mathcal X^{\ast}_{\Gamma_0(p)}(\epsilon)_c$ the image of this section, that is the component of $\mathcal X^{\ast}_{\Gamma_0(p)}(\epsilon)$ that parametrises the $\Gamma_0(p)$-structure given by the canonical subgroup. This is called the canonical locus. We denote its complement by $\mathcal X^{\ast}_{\Gamma_0(p)}(\epsilon)_a$ and call it the anticanonical locus. For any adic space $S\to \mathcal X^{\ast}_{\Gamma_0(p)}$ we denote by
\[S(\epsilon)_a:=S\times_{\Xa_{\Gamma_0(p)}} \XaGzoea\]
the open subspace that lies over the anticanonical locus.
\begin{Definition}
For any adic space $S\to \X(\epsilon)$ corresponding to an $\epsilon$-nearly ordinary elliptic curve $E$ over $\O_S(S)$, we shall call a $\Gamma$-level structure anticanonical if it corresponds to a point of $\X_{\Gamma}(\epsilon)_a\subseteq \X_{\Gamma}$. For instance, a $\Gamma_0(p^n)$-level structure is a locally free subgroup scheme $G_n\subseteq E[p^n]$, \'etale locally cyclic of rank $p^n$, and it is anticanonical if $G_n\cap C_1=0$. Similarly, a $\Gamma(p^n)$-level structure, given by an isomorphism of group schemes $\alpha\colon(\Z/p^n\Z)^2 \to E[p^n]$ is anticanonical if the subgroup scheme generated by $\alpha(1,0)$ is anticanonical.
\end{Definition}
For any $n\in\N$, the transformation of moduli functors that sends an elliptic curve $E$ together with an anticanonical $\Gamma_0(p^n)$-structure $G_n$ to $E/G_n$ induces an isomorphism
\begin{equation}\label{eq:Atkin--Lehner}
\X_{\Gamma_0(p^n)}(\epsilon)_a \xrightarrow{\sim}  \X(p^{-n}\epsilon)
\end{equation}
that is called the Atkin--Lehner isomorphism. The inverse is given by sending $E$ with its canonical subgroup $C_n$ of rank $p^n$ to the data of $E/C_n$ with $\Gamma_0(p^n)_a$-structure $E[p^n]/C_n$. The Atkin--Lehner isomorphism uniquely extends to the cusps for all $n$, and for varying $n$ the resulting isomorphisms fit into a commutative diagram of towers
\begin{equation*}
\begin{tikzcd}
	\cdots \arrow[r] & \Xa_{\Gamma_0(p^2)}(\epsilon)_a \arrow[d, no head,"\sim" labelrotate] \arrow[r] & \Xa_{\Gamma_0(p)}(\epsilon)_a \arrow[d, no head,"\sim" labelrotate] \arrow[r] & \Xa(\epsilon) \arrow[d, no head,equal] \\
	\cdots \arrow[r] &  \Xa(p^{-2}\epsilon) \arrow[r,"\phi"] &  \Xa(p^{-1}\epsilon) \arrow[r,"\phi"] &  \Xa(\epsilon)
\end{tikzcd}
\end{equation*}
where in the bottom row, the morphism $\phi$ is the ``Frobenius lift'' defined in terms of moduli by sending $E$ to $E/C_1$. The resulting tower is called the ``anticanonical tower''.

It is a crucial intermediate result in \cite{torsion} that the anticanonical tower ``becomes perfectoid'' in the inverse limit. More precisely:

\begin{Theorem}[{\cite[Cor.~III.2.19]{torsion}}]
There is an affinoid perfectoid tilde-limit
\[\XaGea0\infty \sim \textstyle\varprojlim_{n\in\N}\XaGea0n.\]
\end{Theorem}

Since the forgetful morphisms $\XaGea 1 n\to \XaGea 0 n$ are finite \'etale $(\Z/p^n\Z)^{\times}$-torsors, even over the cusps, one immediately deduces that in the inverse limit these give rise to an affinoid perfectoid space $\XaGea 1 \infty\sim \varprojlim_n\XaGea 1 n $ together with a forgetful map \[\XaGea 1 \infty  \rightarrow \XaGea 0 \infty\]
that is a pro-\'etale $\Z_p^\times$-torsor. Similarly, for full level $\Gamma(p^n)$, one obtains an affinoid perfectoid space $\XaGea{}{\infty}$ together with a forgetful map that is  a pro-\'etale $\Gamma_0(p^\infty)$-torsor $\XaGea{}{\infty}\rightarrow \XaGea{0}{\infty}$, where we set:
\begin{Definition}\label{df: Gamma_0(p^m)}
	For any $m\in \Z_{\geq 0}\cup\{\infty\}$, let $\Gamma_0(p^m)=\{\smallmat{\ast}{\ast}{c}{\ast}\in \GL_2(\Z_p)\mid c\equiv 0 \bmod p^m\}$.
\end{Definition}

All in all, we have a tower of morphisms
\begin{equation*}
	\begin{tikzcd}
		\mathcal X^{\ast}_{\Gamma(p^\infty)}(\epsilon)_a \arrow[r] & \mathcal X^{\ast}_{\Gamma_1(p^\infty)}(\epsilon)_a \arrow[r] & \mathcal X^{\ast}_{\Gamma_0(p^\infty)}(\epsilon)_a \arrow[r] & \mathcal X^{\ast}_{\Gamma_0(p)}(\epsilon)_a
	\end{tikzcd}
\end{equation*}
which is a pro-\'etale $\Gamma_0(p)$-torsor away from the boundary, but not globally: One reason is that there is ramification over the cusps in $\mathcal X^{\ast}_{\Gamma_0(p^\infty)}(\epsilon)_a \to \mathcal X^{\ast}_{\Gamma_0(p)}(\epsilon)_a$, but we note that the tower is still no torsor on the quasi-pro-\'etale site, as we will see on $q$-expansions.

\subsection{Analytic Tate curve parameter spaces at tame level}\label{subsection: adic cusps and Tate curves}
In this subsection, we recall the universal analytic Tate curves at the cusps, as developed by \cite{conrad2006modular}.  The main technical difference is that we work with analytic adic spaces instead of rigid spaces. In particular, instead of the generalisation of Berthelot's functor constructed in \S3 of \textit{loc. cit.} we may use the adic generic fibre functor.

For now, we shall focus on the adic analytic modular curve $\X^{\ast}$ over $K$. We remark that everything in this section and the next will also work for $\X_{\Gamma_0(p^n)}^{\ast}$, except for the additional phenomenon of ramification at the cusps. In order to separate the discussion of these two topics, and to simplify the exposition, we shall therefore focus on $\X^{\ast}$ for now.

As discussed in \S\ref{s:classical-recollections-on-cusps}, the subscheme $\partial X^{\ast}\subseteq X^{\ast}$ decomposes into a union of points of the form $x\colon\Spec(L)\rightarrow X^{\ast}$ where $K\subseteq L\subseteq K[\zeta_N]$ is a subfield depending on $x$. For example, if $\Gamma^p=\Gamma(N)$, we have $L=K[\zeta_N]$ at every cusp. We now switch to an analytic setup: 
\begin{Definition}\label{d:field-of-definition-of-Tate-curve}
By a cusp of $\X^{\ast}$ we shall mean a  connected (but not necessarily geometrically connected) component of $(\X^{\ast}\backslash \X)^{\red}$.
Given a fixed cusp $x$, we shall denote by $L=L_x\subseteq K[\zeta_N]$ the coefficient field of definition of the corresponding Tate curve. We have $L=K[\zeta_d]$ for some $d|N$. Let $\mathfrak m_L$ be the maximal ideal of $\O_L$ and let $k_L$ be the residue field.
\end{Definition}
From now on for the rest of this section, let us fix a cusp $x\in \X^{\ast}$. To simplify notation, we will write $L=L_x$ and $e=e_x$.  We note that the cusp $x$ will decompose into $[L:K]$ disjoint $L$-points after base-changing $\mathcal X^{\ast}$ from $K$ to $L$.
 
By Prop.~\ref{p:completion-at-cusp}, the completion along $x$ is canonically of the form
\begin{equation}\label{eq:completion-at-schematic-cusp}
\Spf(\O_{L}\llbracket q\rrbracket)\rightarrow X_{\O_K}^{\ast}
\end{equation}
where $\O_{L}\llbracket q\rrbracket$ carries the $q$-adic topology. Here we recall from Notation~\ref{n:e} that we have renormalised parameters from $q^{1/e}$ to $q$, so that the universal Tate curve is $\Tate(q^{e})$. Upon $\pi$-adic completion, this becomes a morphism
\[\Spf(\O_{L}\llbracket q\rrbracket)\rightarrow \mathfrak X^{\ast}\]
where now $\O_{L}\llbracket q\rrbracket$ is endowed with the $(p,q)$-adic topology.
We note that this morphism restricts to $\Spf(\O_{L}\llbracket q\rrbracket)\rightarrow \mathfrak X^{\ast}(0)$ since the supersingular locus is disjoint from the cusps.

On the $p$-adic generic fibre, we obtain a morphism of analytic adic spaces over $K$
\[\D:=\Spf(\mathcal O_{L}\llbracket q\rrbracket)^{\ad}_\eta \rightarrow \mathcal X^{\ast}.\]
Here $\D$ is the open unit disc over $K$, a topologically finite type but non-quasicompact, non-affinoid analytic adic space. We emphasize that in general, this depends on $x$, as $L=L_x$ does. If the cusp $x$ is not clear from the context, we shall therefore denote this space by $\D_x$.

The global functions on $\D$ are given by $\mathcal O^+(\D) =\O_{L}\llbracket q\rrbracket$ and
\[\mathcal O(\D)=\Big \{ \textstyle\sum_{n\geq 0}a_nq^n \in L\llbracket q\rrbracket \text{ such that } |a_n|z^n\to 0\text{ for all }0\leq z<1 \Big\}.\]

More classically, if we regarded $\D$ as a rigid space, it would be associated to the formal scheme $\Spf(\mathcal O_{L}\llbracket q\rrbracket)$ via Conrad's non-Noetherian generalisation of Berthelot's rigid generic fibre construction, \cite[Thm.\ 3.1.5]{conrad2006modular}.

\begin{Lemma}\label{l: Conrad's theorem on generic fibres of completions around the cusp}
	The map $\D \hookrightarrow \mathcal X^{\ast}$ is an open immersion that sends the origin to the cusp.
\end{Lemma}
\begin{Remark}
This is part of \cite[Thm.~3.2.8]{conrad2006modular} for $\Gamma^p = \Gamma_1(N)$, and in general follows from \cite[Thm.~3.2.6]{conrad2006modular}. These moreover give a moduli interpretation in terms of  analytic generalised elliptic curves, as well as a universal  analytic generalised Tate curve over $\D$.
\end{Remark}

\begin{proof}
	In an affine open formal neighbourhood $\Spf(A)\subseteq \mathfrak X^{\ast}(0)$ of the cusp, $x$ is cut out by a principal ideal $(f)$ for some non-zero-divisor $f$. The completion along the cusp is then  $A\llbracket T\rrbracket/(T-f)$. The adic generic fibre is thus the union over the spaces $\Spa(A\langle f^n/p\rangle[1/p])$, and is thus the open subspace of $\Spa(A[1/p])\subseteq \mathcal X^{\ast}$ defined by the condition $|f|<1$.
\end{proof}
\begin{Lemma}\label{lemma for comparing schematic cusp to adic cusp}
	The morphism of locally ringed spaces $\D\rightarrow \Spec (\mathbb Z_p\llbracket q\rrbracket\otimes_{\Z_p} { \O_{L}})$
	induced by the inclusion $\mathbb Z_p\llbracket q\rrbracket \hookrightarrow \mathcal O^+(\D)$ fits into a commutative diagram of ringed spaces:
	
	\begin{equation*}
		\begin{tikzcd}
			\Spec(\mathbb Z_p\llbracket q\rrbracket\otimes_{\Z_p} { \O_L}) \arrow[r] & X^{\ast}_{\O_K} \\
			\D \arrow[r,hook] \arrow[u] & \mathcal X^{\ast} \arrow[u].
		\end{tikzcd}
	\end{equation*}
\end{Lemma}
\begin{proof}
Let $R:=\mathbb Z_p[\zeta_d]$ and $R_0:=\mathbb Z_p[\zeta_{d_0}]$ where $d_0$ is the largest divisor of $d$ such that $K$ contains a primitive $d_0$-th unit root. 
The adification of the $p$-adic completion of
$f\colon\Spf(R\llbracket q\rrbracket)\rightarrow \Spec(R\llbracket q\rrbracket)\rightarrow X^{\ast}_{R_0}$  fits into a commutative diagram of ringed spaces
\begin{equation*}
	\begin{tikzcd}
		{\Spec(R\llbracket q\rrbracket)} \arrow[r,"f"] & X_{R_0}^{\ast}\arrow[r]& \Spec(R_0) \\
		{\Spa(R\llbracket q\rrbracket,R\llbracket q\rrbracket)} \arrow[u] \arrow[r,"\hat{f}"] & {\mathfrak X_{R_0}^{\ast \ad}} \arrow[u]\arrow[r]&\Spa(R_0,R_0).\arrow[u]
	\end{tikzcd}
\end{equation*}
The lemma follows upon taking the fibre over $\Spa(K,\O_K)\to \Spec(\O_K)$.
\end{proof}

We thus have the following moduli interpretation of $\mathring{\D}:=\D(q\neq 0)\subseteq \D$:
\begin{Lemma}\label{l:moduli interpretation of adic Tate parameter space}
	Let $S$ be an honest adic space over $K$ and let $\varphi\colon S\rightarrow \X$ be a morphism corresponding to an elliptic curve $E$ over $\mathcal O_S(S)$ with $\Gamma^p$-level structure $\alpha_N$.
	Then $\varphi$ factors through the punctured open unit disc $\mathring{\D}\rightarrow \X$ at the cusp $x$ if and only if $E\cong \Tate(q_E)$ is a Tate curve for some $q_E\in S$, the $\Gamma^p$-structure $\alpha_N$ corresponds to $x$ under Prop.~\ref{p:completion-at-cusp}, and $q_E$ is locally topologically nilpotent on $S$, i.e.\ $v_z(q_E)$ is cofinal in the value group for all $z\in S$.
\end{Lemma}
\begin{proof}
	If $\varphi\colon S\rightarrow \X^{\ast}\to X^{\ast}$ factors through $\mathring{\D}\hookrightarrow \X$, then by Lemma~\ref{lemma for comparing schematic cusp to adic cusp} it factors through the map $\Spec(\mathcal O_{L}\cc q)\to X^{\ast}$. Consequently, $E$ is a Tate curve and we obtain a parameter $q_E\in \mathcal O_S(S)$ as the image of $q\in \mathcal O_{L}\cc q\subseteq \O(\D)$ on global sections. This is locally topologically nilpotent because $q\in \mathcal O(\D)$ is locally topologically nilpotent. 
	
	Conversely, assume that $E$ is a Tate curve such that $q_E\in \mathcal O(S)$ is topologically nilpotent, with $\Gamma^p$-level structure associated to $x$. The latter condition implies that $\O(S)$ is naturally an $L$-algebra.
	It therefore suffices to consider the case that $S=\Spa(B,B^{+})$ is an affinoid adic space over $\Spa(L,\mathcal O_L)$. The condition that $q_E$ is topologically nilpotent implies that for any $x$ there is $n$ such that $|q_E(x)|^n\leq |\varpi|$. Since $S$ is affinoid and thus quasicompact, we can find $n$ that works for all $x\in S$. Similarly, since $E$ is a Tate curve, $q_E\in B$ is a unit and we thus have $0<|q_E(x)|$ for all $x\in S$. Again by compactness, we can find $m$ such that $|\varpi|^m\leq |q_E|$. 
	But then $q_E^n/\varpi,\varpi^m/q_E\in B^{+}$ and there is a natural morphism of affinoids
	\[(L\langle q,q^n/\varpi,\varpi^m/q\rangle,\mathcal O_{L}\langle q,q^n/\varpi,\varpi^m/q\rangle)\,\rightarrow\, (B,B^{+}),\quad q\mapsto q_E\]
	through which the map $\mathcal O_{L}\cc{q}\rightarrow B$ defining the Tate curve structure factors. Since the algebra on the left defines an affinoid open of $\mathring{\D}$, this gives the desired factorisation.
\end{proof}

\begin{Definition}
	Consider $\O_L\cc q$ with the $p$-adic topology, that is we suppress the topology coming from $q$.
	Let $\O_L\lauc{q}:=\O_L\cc{q}^{\wedge}$ be the $p$-adic completion.	Explicitly,
	\[\O_L\lauc{q}=\big\{\textstyle\sum_{n\in\Z} a_nq^n\in \O_L\llbracket q^{\pm 1}\rrbracket\mid a_n\to 0 \text{ for }n\to-\infty\big\}.\]
	This is a discrete valuation ring with maximal ideal $(p)$ and residue field $k_L\cc q$.
\end{Definition}

\begin{Remark}\label{remark: Tate curve of good reduction}
	The adic space $S=\Spa(\O_L\lauc{q}[\tfrac{1}{p}],\O_L\lauc{q})$ consists of a single point. As we have suppressed the $q$-adic topology, it is clear that $q$ is not topologically nilpotent on $S$. We conclude that the map 
	\[S\rightarrow \Spec \mathcal O_{L}\cc{q}\to X^{\ast}\]
	does not factor through $\mathring{\D}\hookrightarrow \X^{\ast}$ even though it corresponds to a Tate curve. The point is that this Tate curve has good reduction: Concretely, this means that it is already an elliptic curve over $\mathcal O_{L}\cc{q}$. Its reduction mod $\mathfrak m_L$ is simply the Tate curve $\Tate(q)$ over $k_L\cc{q}$. It therefore gives rise to a point in the good reduction locus $\mathcal X_{\gd}\subseteq \X\subseteq \X^{\ast}$.
\end{Remark}
We can enlarge the Tate parameter space $\mathcal D$ so that it includes the above example:

\begin{Definition}\label{d: OK bb q^1/p^infty _p}
	Let $\overline{\D}=\Spa(\O_L\llbracket q\rrbracket[\tfrac{1}{p}],\O_L\llbracket q\rrbracket)$ where, contrary to our usual convention, $\O_L\llbracket q\rrbracket$ is endowed with the $p$-adic topology. We let $\O_L\llbracket q^{1/p^\infty}\rrbracket _p:=(\varinjlim_m \O_L\llbracket q^{1/p^m}\rrbracket )^{\wedge}$ where, crucially, the completion is the $p$-adic one. This is a perfectoid $\O_L$-algebra. If we use the $(p,q)$-adic topology instead, we obtain a $(p,q)$-adically complete ring $\O_L\llbracket q^{1/p^\infty}\rrbracket $. There is a natural inclusion $\O_L\llbracket q^{1/p^\infty}\rrbracket _p\hookrightarrow\O_L\llbracket q^{1/p^\infty}\rrbracket $, but this is not an isomorphism: For example, $\sum_{n\in\N} q^{n+\frac{1}{p^n}}$ defines an element contained in the codomain but not in the image.
	
	As before, we emphasize that $\overline{\D}$ depends on our chosen cusp $x$. If this cusp is not clear from the context, we shall also write $\overline{\D}_x$ for the parameter space associated to $x$.
\end{Definition}
\begin{Lemma}\label{l:overline{D} is sousperfectoid}
	\begin{enumerate}
	\item The Huber pair $(\O_L \llbracket q\rrbracket[\tfrac{1}{p}],\O_L \llbracket q \rrbracket)$, where $\O_L \llbracket q\rrbracket$ is endowed with the $p$-adic topology, is sous-perfectoid in the sense of \cite[\S6.3]{ScholzeBerkeleyLectureNotes}, and thus sheafy.
	\item We have an open immersion $
	\D=\cup_m\overline{\D}(|q|\leq |p|^{1/p^m})\hookrightarrow \overline{\D}$.
	\item We have  an open immersion
	$\Spa(\O_L\lauc{q}[\tfrac{1}{p}],\O_L\lauc{q})= \overline{\D}(|q|\geq 1)\hookrightarrow \overline{\D}$.
	\end{enumerate}
\end{Lemma}
Parts (1)-(2) of the lemma say that we may think of $\overline{\D}$ as a compactification of $\D$, albeit a non-standard one. As discussed in Example~\ref{ex: rank-2-point} below, $\mathcal D$ and $\overline{\D}(|q|\geq 1)$ form a disjoint open cover of $\overline{\D}$ up to one additional point of rank $2$ which lies between them.
\begin{proof}
	 The traces $\O_L\llbracket q^{1/p^n}\rrbracket\to \O_L\llbracket q\rrbracket$, $q^{i/p^n}\mapsto 0$ for $0<i<p^n$ in the $p$-adically completed limit give rise to an $\O_L\bb{q}$-linear section $\O_L \llbracket q^{1/p^\infty}\rrbracket_p\to \O_L \llbracket q\rrbracket$. This shows the first part.
	 
	 Part (3) is clear as $\overline{\D}(|q|\geq 1)$ is formed by adjoining $q^{-1}$ and completing $p$-adically.
	 
	 Part (2) follows from $\D=\cup_m \D(|q|^m\leq |p|)$ and \[\O(\overline{\D}(|q|^m\leq |p|))=\O_L\llbracket q\rrbracket\langle \tfrac{q^m}{p}\rangle[\tfrac{1}{p}]=\O_L\langle \tfrac{q^m}{p}\rangle[\tfrac{1}{p}]=\O(\D(|q|^m\leq |p|)).\qedhere\]
\end{proof}

\begin{Lemma}\label{l:full-Tate-curve-parameter-space}
	For every cusp of $\X^{\ast}$ there is a natural morphism $\overline{\D}\to \X^{\ast}(\epsilon)$, extending the  morphism $\D\hookrightarrow \X^{\ast}(\epsilon)$. The fibre of the good reduction locus is precisely $\overline{\D}(|q|\geq 1)$.
\end{Lemma}
\begin{proof} For the first part, we take the morphism $\Spec(\O_L\llbracket q\rrbracket )\to X_{\O_K}^{\ast}$ from Lemma~\ref{lemma for comparing schematic cusp to adic cusp}, complete $p$-adically and pass to the generic fibre.
	
	To see the second part, we note that morphisms $\Spa(R,R^+)\to \mathcal X_{\gd}$ correspond to elliptic curves over $R^+$. The locus of $\overline{\D}$ where the Tate curve is defined over $\O^+$ is precisely that where $q$ is in $\O^{+,\times}$. Equivalently, as we have $|q|\leq 1$ on $\overline{\D}$,  this means that $|q|\geq 1$.
\end{proof}
\begin{Remark}
	\begin{enumerate}
		\item The map $\overline{\D}\to \X^{\ast}(\epsilon)$ is no longer an open immersion. Indeed, $\overline{\D}$ is not even of locally topologically finite type over $K$. We therefore do not have a rigid analogue of $\overline{\D}$, and can only capture this map in the setting of adic spaces.
		\item Here and in the following, by base-change we could  replace $\O_{L}\llbracket q\rrbracket$ by the smaller ring $\Z_p\llbracket q\rrbracket\hotimes_{\Z_p}\O_{L}$ where the completion is $p$-adic. But we will not need this.
		\item  If we form the union over all cusps $\mathcal C$, we get two maps \[(i)\: \bigsqcup_{x\in \mathcal C} {\D}_x\times \mathcal X_{\gd}\to \mathcal X ,\quad\quad (ii) \: \bigsqcup_{x\in \mathcal C} \overline{\D}_x\times \mathcal X_{\gd}\to \mathcal X\]
		of which (ii) is a cover (for example in the $v$-topology), in contrast to (i) whose image misses some points of rank 2. This is what we shall discuss next.
	\end{enumerate}
\end{Remark}

\subsection{The points of the adic space \texorpdfstring{$\mathcal X^{\ast}$}{X*}}\label{s:points-of-Xast}
Our next goal is to see how much of the adic space $\mathcal X^{\ast}$ is captured by the Tate curve parameter spaces $\D$ in conjunction with the good reduction locus. The answer is that they give everything except for a finite set of higher rank points whose moduli interpretation in terms of Tate curves we shall describe. 

Indeed, since the loci $\mathcal X_{\gd}$ and $\D$ in $\X^{\ast}$ are the preimages of a cover by an open and a closed set of the special fibre $X_{k}$ under the specialisation map $ |\mathcal X^{\ast}|\rightarrow |X_{k}|$, it follows from the adaptation of \cite[Lemma~7.2.5]{de1995crystalline},  to the setting of \cite{conrad2006modular} that if we worked with rigid spaces, then $\D$ and  $\mathcal X$ would cover $\mathcal X^{\ast}$ set-theoretically, but not admissibly so \cite[\S 7.5.1]{de1995crystalline}.

\begin{Proposition}\label{p: classification of points of X^ast}
	Let $z\in \mathcal X^{\ast}$ be any point, then we are in either of the following cases:
	\begin{enumerate}[label=(\alph*)]
		\item $z\in \mathcal X^{\ast}$ is contained in the good reduction locus $\mathcal X_{\gd}$,
		\item $z\in \D_x\hookrightarrow \mathcal X^{\ast}$ is contained in a Tate parameter space around a cusp $x$ of $\X^{\ast}$,
		\item $z\in \mathcal X^{\ast}\backslash\mathcal X_{\gd}$ is of rank $>1$ and its unique height $1$ vertical generisation $z'$ is in $\mathcal X_{\gd}$.	\end{enumerate}
	When we denote by $j$ the global function on $\X$ induced by the $j$-invariant $j\colon \X\rightarrow \A^{1,\an}$, then the above are respectively equivalent to
	\begin{enumerate}[label=(\alph*')]
		\item $|j(z)|\leq 1$,
		\item $|j(z)|> 1$ and its inverse is cofinal in the value group,
		\item $|j(z)|> 1$ and its unique rank 1 generisation $z'$ satisfies $|j(z')|=1$. 
	\end{enumerate}
\end{Proposition}
We note that the analogous description also holds after adding level at $p$.
\begin{proof}
	The space $\X^{\ast}$ is analytic, hence the valuation $v_z$ is always microbial. This implies that $z$ has a unique generisation $z'$ of height 1, so statements ($c$) and ($c$') make sense.
	
	The case of the cusps is clear, so let us without loss of generality assume that $z\in \X$.
	
	We start by proving that ($a$) and ($a'$) are equivalent. Recall that $X_{\O_K}$ is the preimage of $\A_{\O_K}^{1}$ under the morphism of $\mathcal O_K$-schemes $j\colon X_{\O_K}^{\ast}\rightarrow \P_{\O_K}^{1}$. Upon formal completion and passing to the adic generic fibre, $j$ becomes $j^{\an}$ while $\A^{1}_{\O_K}\subseteq \P^{1}_{\O_K}$ is sent to the open disc $B_{1}(0)\subseteq \A_K^{1,\an}\subseteq \P_K^{1,\an}$ defined by $|j(z)|\leq 1$. Since the adic generic fibre of the completion of $X_{\O_K}\subseteq X_{\O_K}^{\ast}$ is $\X_{\gd}\subseteq \mathcal X^{\ast}$, this shows that $\X_{\gd}$ is precisely the preimage of $B_{1}(0)$ under $j^{\an}\colon \X=X^{\an}\rightarrow \A^{1,\an}$.
	
	Next, let us prove that ($b$) and ($b$') are equivalent. We can always find a morphism
	\[r_z\colon \Spa(C,C^{+})\rightarrow \X\]
	where $(C,C^{+})$ is a complete algebraically closed non-archimedean field, such that $z$ is in the image of $r_z$. It thus suffices to show that $r_z$ factors through some $\D\hookrightarrow \X^{\ast}$. By Lemma~\ref{l:moduli interpretation of adic Tate parameter space} it suffices to show that ($b$') holds if and only if the elliptic curve $E$ over $C$ that $r_z$ represents is a Tate curve with nilpotent parameter $q_E\in C$. 
	
	The image of $j$ in $C$ is precisely the $j$-invariant $j_E$ of $E$. Since in a non-archimedean field the elements with cofinal valuation are precisely the topologically nilpotent ones, condition  ($b$') is equivalent to $j_E\ne 0$ and $j_E^{-1}$ being topologically nilpotent. 
	We can now argue like in the classical case of $p$-adic fields to see that this is equivalent to $E$ being a Tate curve with $q_E$ topologically nilpotent: 
	Assume the latter, then $j_E=q_E^{-1}+744+196884q+\dots \ne 0$ has valuation $|j_E|=|1/q_E|$ in $C$ and thus $j_E$ satisfies ($b$').
	To see the converse, recall that in the formal Laurent series ring $\Z\cc{q}$ the formula $j(q)=q^{-1}+744+\dots$ reverses to 
	\[q(j^{-1}) = j^{-1}+744j^{-2}+750420j^{-3}+\dots \in \Z\llbracket j^{-1}\rrbracket. \]
	If now $j_E^{-1}$ is topologically nilpotent, the above series converges in $C$ and we obtain a topologically nilpotent element $q_E\in C^{\times}$ with $j_E=1/q_E+744+\dots=j(q_E)$. The Tate curve $\Tate(q_E)_C$ over $C$ thus has the same $j$-invariant as $E$, and since $C$ is algebraically closed we conclude that $E\cong \Tate(q_E)_C$. This shows that ($b$) and ($b$') are equivalent.
	
	Next, let us show that ($c$') holds if and only if ($a$') and ($b$') don't hold. Recall that we always have a unique height 1 vertical generisation $z'$. Clearly $|j(z)|\neq 0$ if and only if $|j(z')|\neq 0$, and if in this case $|j(z)|^{-1}$ is cofinal then $|j(z')|^{-1}$ is cofinal. This implies that (b') and (c') can't hold at the same time. On the other hand, if $|j(z)|>1$, then either $|j(z')|=1$, or $|j(z')|>1$ in which case $|j(z')|^{-1}<1$ is cofinal because $v_{z'}$ has height 1. This shows that if $|j(z)|>1$ then we are either in case ($b$') or in ($c$').
	
	Since clearly ($c$) implies that ($a$) and ($b$) do not hold, it remains to prove ($c$') implies ($c$). But this follows from applying the equivalence of ($a$) and ($a'$) first to $z$ and then to $z'$.  
\end{proof}

\begin{Example}\label{ex: rank-2-point}
	Let us work out an example for an elliptic curve corresponding to a point of type $(c)$: Let $x$ be a cusp and $L=L_x$.	
	Let $\R_{>0}\times \gamma^{\Z}$ be the totally ordered group for which $\gamma$ is such that $z<\gamma^n<1$ for all $n\in \Z_{\geq 1}$ and all $z\in \R_{<1}$. Equip the field $F=\mathcal O_{L}\lauc{q}[\tfrac{1}{p}]$ with the valuation 
	\[v_{1^{-}}\colon F\rightarrow (\R_{> 0}\times \gamma^{\Z}) \cup \{0\},\quad \sum a_nq^n\mapsto \max_{n\in\Z} |a_n|\gamma^n. \]
	
	Recall that $\m_{L}\subseteq \mathcal O_{L}$ is the maximal ideal. The valuation subring of $F$ defined by $v_{1^{-}}$ is
	\[ F^{+}=\left\{\sum_{n\gg -\infty}^{\infty}a_nq^n \in \mathcal O_{L}\lauc{q} \middle | a_n\in \mathfrak m_L\text{ for all }n<0\right\}. \]
	Indeed, we have $v_{1^{-}}(\sum_{n\gg -\infty}^{\infty}a_nq^n)\leq 1$ if and only if $|a_n|\gamma^n\leq 1$ for all $n$. For $n\geq 0$ we have $|a_n|\gamma^n\leq 1$ if and only if $|a_n|\leq 1$, that is $a_n\in \mathcal O_{L}$. For $n<0$, on the other hand, $|\gamma|^n$ is ``infinitesimally'' bigger than $1$, so that $|a_n|\gamma^n\leq 1$ if and only if $|a_n|<1$, that is $a_n\in \m_L$.
	
	The Tate curve $\Tate(q^{e})$ over $F$ for $e=e_x$ equipped with the $\Gamma^p$-structure corresponding to $x$, gives rise to a map 
	\[ \varphi\colon \Spa(F, F^{+})\rightarrow \mathcal X^{\ast}.\]
	We claim that $\varphi$ sends $v_{1^{-}}$ neither into $\mathcal X_{\gd}$ nor into any of the Tate parameter spaces $\D\subseteq \mathcal X^{\ast}$. Indeed, the $j$-invariant of $\Tate(q^{e})$ is 
	\begin{equation}\label{equation: j-invariant of Tate curve over rank 2 point}
	j=q^{-e}+744+q^e(\dots) \not \in F^{+}
	\end{equation}
	which is not contained in $F^+$ by the above description. This shows that $\Tate(q^e)$ does not extend to an elliptic curve over $F^+$. On the other hand,  $q$ is not locally topologically nilpotent in $L$. Thus, by Lemma~\ref{l:moduli interpretation of adic Tate parameter space}, $v_{1^{-}}$ does not land in the open subspaces $\D_x\hookrightarrow \X^{\ast}$. 
	
	Prop.~\ref{p: classification of points of X^ast} explains this as follows: The unique rank 1 vertical generisation of $v_{1^{-}}$ is
	\[ v\colon  F \rightarrow \R_{\geq 0},\quad \sum a_nq^n\mapsto \max_{n\in\Z} |a_n|\]
	with larger valuation ring $\mathcal O_{L}\lauc{q}\supseteq F^{+}$.
	We see from equation~\eqref{equation: j-invariant of Tate curve over rank 2 point} that 
	\[|j(v_{1^{-}})|=\gamma^{-e}<1, \text{ while } |j(v)|=1.\]
	This shows that $\varphi$ sends $v_{1^{-}}$ to one of the points of type (c) in Prop.~\ref{p: classification of points of X^ast}, while its generisation $v$ goes to the point of $\mathcal X_{\gd}$ defined in Remark~\ref{remark: Tate curve of good reduction}.
\end{Example}
\subsection{Tate curve parameter spaces at level $\Gamma_0(p^n)$}
Next, we discuss the behaviour of the Tate curve parameter spaces in the anticanonical tower. For this we first recall the situation at the cusps on the level of schemes:

Consider the forgetful morphism $f\colon X^{\ast}_{\Gamma_0(p)}\to X^{\ast}$. Over each cusp of $X^{\ast}$ there are precisely two cusps of $X^{\ast}_{\Gamma_0(p)}$: One is called the \'etale cusp, it corresponds to the $\Gamma_0(p)$-level structure $\mu_{p}\subseteq \Tate(q)[p]$ on the Tate curve. The other is the ramified cusp, it corresponds to the level structure $\langle q^{1/p}\rangle \subseteq \Tate(q)[p]$. In particular, this latter level structure is only defined over $\Z\cc{q^{1/p}}$, and by Prop.~\ref{p:completion-at-cusp} the completion at this cusp is given by 
\[\Spf(\O_L\llbracket q^{1/p}\rrbracket  )\to X_{\Gamma_0(p)}^{\ast}.\]
 The names reflect that the map $X^{\ast}_{\Gamma_0(p)}\to X^{\ast}$ is \'etale at the one sort of cusps, but is ramified at the other: Over the \'etale cusp the map induced on completions is the identity
\[\O_L\llbracket q\rrbracket\to \O_L\llbracket q\rrbracket\]
whereas over the ramified cusp it is the inclusion
\[\O_L\llbracket q\rrbracket\to \O_L\llbracket q^{1/p}\rrbracket.\]

For higher level structures $\Gamma_0(p^n)$, the curve $X^{\ast}_{\Gamma_0(p^n)}\to X^{\ast}$ has more cusps of different degrees of ramification $p^i$ with $i\in\{0,\dots,n\}$, and corresponding completions of the form $\Spf(\O_L\llbracket q^{1/p^i}\rrbracket  )\to X^{\ast}_{\Gamma_0(p^n)}$. There is, however, exactly one \'etale cusp, corresponding to the $\Gamma_0(p^n)$-level structure $\mu_{p^n}\subseteq\Tate(q)[p^n]$, and exactly one purely ramified one, corresponding to $\langle q^{1/p^n}\rangle$. Relatively with respect to $X^{\ast}_{\Gamma_0(p^n)}\to X^{\ast}_{\Gamma_0(p)}$,  the purely ramified cusps lie over the ramified cusps of $X^{\ast}_{\Gamma_0(p)}$, while all other cusps of $X^{\ast}_{\Gamma_0(p^n)}$ lie over the \'etale cusps of $X^{\ast}_{\Gamma_0(p)}$.

All constructions from the last two sections now go through with the same proofs for $\X^{\ast}$ replaced by $\X^{\ast}_{\Gamma_0(p^n)}$, and $\O_L\llbracket q\rrbracket  $ replaced by $\O_L\llbracket q^{1/p^i}\rrbracket  $ where $i$ depends on the cusp:
\begin{Definition}
	For any $n\in \Z_{\geq 0}$, we write $\mathcal D_i:=\Spa(L\langle q^{1/p^i}\rangle,\O_L\langle q^{1/p^i}\rangle)(|q|>1)$ for the open unit disc over $L$ in the parameter $q^{1/p^i}$. We write $\mathring{\D}_i$ for the open subspace obtained by removing the origin, i.e.\ the point defined by $q^{1/p^i}\mapsto 0$. When we want to emphasize the dependence on field $L=L_x$ determined by a cusp $x\in \X^{\ast}$, we write these as $\mathcal D_{i,x}$ and $\mathring{\D}_{i,x}$.
\end{Definition}
Then like in Lemma~\ref{l: Conrad's theorem on generic fibres of completions around the cusp}, we have for any cusp $x$ of $\X^{\ast}$ and any cusp of $\X^{\ast}_{\Gamma_0(p^n)}$ over $x$ of ramification degree $p^i$ a canonical open immersion
\[\mathcal D_{i,x}\hookrightarrow  \X^{\ast}_{\Gamma_0(p^n)}.\]

From now on until the next chapter, we shall focus exclusively on the anticanonical locus $\X^{\ast}_{\Gamma_0(p^n)}(\epsilon)_a$. Here the ramification is very easy to describe, by the following proposition:

\begin{Proposition}\label{Proposition: Tate parameter spaces in the Gamma_0-tower}
	Fix a cusp $x\in \mathcal X^{\ast}$.
	\begin{enumerate}
		\item The cusps of  $\X^{\ast}_{\Gamma_0(p^n)}(\epsilon)_a$ are precisely the purely ramified cusps of $\X^{\ast}_{\Gamma_0(p^n)}$. In particular, there is a canonical open immersion $\mathcal D_{n,x}\hookrightarrow \X^{\ast}_{\Gamma_0(p^n)}(\epsilon)_a$ over $x$.
		\item The forgetful map $\mathcal X^{\ast}_{\Gamma_0(p^n)}(\epsilon)_a\to \mathcal X^{\ast}_{\Gamma_0(p^{n-1})}(\epsilon)_a$ gives a bijection between the respective cusps.
		The associated Tate curve parameter spaces fit into Cartesian diagrams
		\begin{equation*}
			\begin{tikzcd}
				\D_{n,x} \arrow[r] \arrow[d,hook] & \D_{n-1,x} \arrow[d,hook]  \\
				\mathcal X^{\ast}_{\Gamma_0(p^n)}(\epsilon)_a \arrow[r]
				& \mathcal X^{\ast}_{\Gamma_0(p^{n-1})}(\epsilon)_a,
			\end{tikzcd}
		\end{equation*}
		where $\D_n\to \D_{n-1}$ is the canonical finite flat map of degree $p$ which sends $q\mapsto q$.
		\item For any adic space $S$ over $(K,\mathcal O_K)$, the $S$-points of $\mathring{\D}_{n,x}\hookrightarrow \X_{\Gamma_0(p^n)}(\epsilon)_a$ correspond functorially to Tate curves $\Tate(q)$ over $\mathcal O(S)$ with topologically nilpotent parameter $q\in\mathcal O(S)$, a $\Gamma^p$-structure corresponding to $x$ and a choice of $p^n$-th root $q^{1/p^n}$ of $q$ defining a $\Gamma_0(p^n)$-structure $\langle q^{1/p^n}\rangle \subseteq \Tate(q)$. 
	\end{enumerate}
\end{Proposition}
\begin{proof}
	Since the canonical subgroup of the Tate curve is given by $\mu_{p}\subseteq \Tate(q)[p]$, the cusps of $\X^{\ast}_{\Gamma_0(p^n)}$ contained in $\X^{\ast}_{\Gamma_0(p^n)}(\epsilon)_a$ are precisely the ones over the ramified cusps in $X^{\ast}_{\Gamma_0(p)}$. But the cusps of $ X^{\ast}_{\Gamma_0(p^n)}$ over the ramified cusps of $ X^{\ast}_{\Gamma_0(p)}$ are precisely the purely ramified ones. This proves (1). Part (3) follows immediately.
	
	The diagram in (2) commutes because by construction it is the generic fibre of a commutative diagram of formal schemes. Since the morphisms are open immersions, it suffices to check that it is Cartesian on the level of points, where it follows from (3) and Lemma~\ref{l:moduli interpretation of adic Tate parameter space}.
\end{proof}

\subsection{Tate parameter spaces of \texorpdfstring{$\mathcal X^{\ast}_{\Gamma_1(p^n)}(\epsilon)_a$}{X*_G_1(p^n)(e)_a}}
We now pass to higher level at $p$ and describe Tate curve parameter spaces of the form $\D_n\hookrightarrow \mathcal X^{\ast}_{\Gamma_1(p^n)}(\epsilon)_a$.
We note that the integral theory of cusps for $\Gamma_1(p^n)$ is slightly complicated in general, see \S4.2 of \cite{conrad2007arithmetic} for a thorough discussion. However, over the anticanonical locus, the story is very simple:

\begin{Lemma}\label{l:Tate parameter spaces of X^{ast}_{Gamma_1(p^n)}}
	Let $x$ be a cusp of $\mathcal X^{\ast}$. Then there are $\Z_p^\times$-equivariant Cartesian squares
		\begin{equation*}
			\begin{tikzcd}
				\underline{(\Z/p^{n+1}\Z)^{\times}} \times \D_{n+1,x} \arrow[d] \arrow[r] &\underline{(\Z/p^n\Z)^{\times}}\times \D_{n,x}\arrow[d, hook] \arrow[r] & \D_{n,x} \arrow[d, hook] \\
				\mathcal X^{\ast}_{\Gamma_1(p^{n+1})}(\epsilon)_a \arrow[r]&\mathcal X^{\ast}_{\Gamma_1(p^n)}(\epsilon)_a \arrow[r] & \mathcal X^{\ast}_{\Gamma_0(p^n)}(\epsilon)_a,
			\end{tikzcd}
		\end{equation*}
		in which the morphism on the top left is given by the reduction $(a,q)\mapsto (\overline{a},q)$.
\end{Lemma}
\begin{proof}
	As the morphisms in the bottom row are finite \'etale Galois torsors for the groups $\Z/p^n\Z$ and $\Z/p^{n+1}\Z$, it suffices to produce a section $\D_{n,x}\rightarrow
	\X^{\ast}_{\Gamma_1(p^n)}$.
	
	By Prop.~\ref{Proposition: Tate parameter spaces in the Gamma_0-tower}, the cusp of $\X^{\ast}_{\Gamma_0(p^n)}(\epsilon)_a$ over $x$ corresponds to the choice of $\langle q^{1/p^n}\rangle\subseteq \Tate(q)$ as a $\Gamma_0(p^n)$-structure. This can be lifted canonically to the $\Gamma_1(p^n)$-structure given by the $p^n$-th root $q^{1/p^n}$ in $\O(\D_{n,x})$. Upon normalisation, we thus get a canonical lift
	\begin{equation*}
		\begin{tikzcd}
			& \Spec(\Z_p\llbracket q^{1/p^n}\rrbracket\otimes L_x) \arrow[d] \arrow[ld, dashed] \\
			X^{\ast}_{\Gamma_1(p^n)} \arrow[r] & X^{\ast}_{\Gamma_0(p^n)}.
		\end{tikzcd}
	\end{equation*}
	The factorisation $\D_{n,x}\rightarrow \Spec(\Z_p\llbracket q^{1/p^n}\rrbracket\otimes L_x)\rightarrow X^{\ast}_{\Gamma_0(p^n)}$ from the analogue for level $\Gamma_0(p^n)$ of Lemma~\ref{lemma for comparing schematic cusp to adic cusp} together with the universal property of the analytification now give rise to the desired section
	\begin{equation*}
		\begin{tikzcd}
			& \D_{n,x} \arrow[d] \arrow[ld, dashed] \\
			\mathcal X^{\ast}_{\Gamma_1(p^n)} \arrow[r] & \mathcal X^{\ast}_{\Gamma_0(p^n)}.
		\end{tikzcd}
	\end{equation*}
	This shows that the right square is Cartesian. 
	That the outer square is Cartesian follows in combination with Prop.~\ref{Proposition: Tate parameter spaces in the Gamma_0-tower}. Consequently, the left square is also Cartesian.
\end{proof}

\subsection{Tate parameter spaces of \texorpdfstring{$\mathcal X^{\ast}_{\Gamma(p^n)}(\epsilon)_a$}{X*_G(p^n)(e)_a}}
Next, we look at what happens with the cusps in the transition $\mathcal X^{\ast}_{\Gamma(p^n)}(\epsilon)_a\rightarrow \mathcal X^{\ast}_{\Gamma_1(p^n)}(\epsilon)_a$.

Let us fix notation for the {left} action of $\GL_2(\Z/p^n\Z)$ on $\mathcal X^{\ast}_{\Gamma(p^n)}$ in terms of moduli: For any $\gamma\in\GL_2(\Z/p^n\Z)$ it is given  by sending a trivialisation $(\Z/p^n\Z)^{2}\xrightarrow{\sim} E[p^n] $ to \[(\Z/p^n\Z)^{2}\xrightarrow{\gamma^\vee} (\Z/p^n\Z)^{2}\xrightarrow{\sim} E[p^n]\]
where $\gamma^{\vee}=\det(g)\gamma^{-1}$. Here the inverse is necessary to indeed obtain a left action, and the twist by $\det(g)$ ensures that the action on the fibres of the map $\mathcal X^{\ast}_{\Gamma(p^n)}(\epsilon)_a\rightarrow \mathcal X^{\ast}_{\Gamma_1(p^n)}(\epsilon)_a$ is given by matrices of the form $\smallmat{\ast}{\ast}{0}{1}$ rather than $\smallmat{1}{\ast}{0}{\ast}$.

\begin{Definition}
	For $0\leq m\leq n\in\N$, we denote by $\Gamma_0(p^m,\Z/p^n\Z)\subseteq \operatorname{GL}_2(p^m,\Z/p^n\Z)$ the subgroup of matrices of the form $\smallmat{\ast}{\ast}{c}{\ast}$ with $c\equiv 0 \bmod p^m$.
\end{Definition}
The forgetful map $X^{\ast}_{\Gamma(p^n)}\to X^{\ast}_{\Gamma_0(p)}$ is given by reducing $(\Z/p^n\Z)^{2}\xrightarrow{\sim} E[p^n]$ mod $p$  to $(\Z/p\Z)^{2}\xrightarrow{\sim} E[p]$ and sending it to the subgroup generated by $(1,0)$. Consequently, the action of $\Gamma_0(p,\Z/p^n\Z)$ leaves the forgetful morphism $X^{\ast}_{\Gamma(p^n)}\to X^{\ast}_{\Gamma_0(p)}$ invariant. We see from this that the action of $\Gamma_0(p,\Z/p^n\Z)$ restricts to an action on $\mathcal X^{\ast}_{\Gamma(p^n)}(\epsilon)_a\subseteq \X^{\ast}_{\Gamma(p^n)}$.
\begin{Lemma}\label{l:Tate parameter spaces of X_Gamma(p^n)(epsilon)_a}
	Let $x$ be a cusp of $\mathcal X^{\ast}$.
	\leavevmode
	\begin{enumerate}
		\item Depending on our chosen primitive root $\zeta_{p^n}$, there is a canonical Cartesian diagram
		\begin{equation*}
			\begin{tikzcd}
				\underline{\Gamma_0(p^n,\Z/p^n\Z)}\times \D_{n,x} \arrow[d,hook] \arrow[r] & \D_{n,x} \arrow[d,hook] \\
				\mathcal X^{\ast}_{\Gamma(p^n)}(\epsilon)_a \arrow[r] & \mathcal X^{\ast}_{\Gamma_0(p^n)}(\epsilon)_a.
			\end{tikzcd}
		\end{equation*}
		where the left map is $\Gamma_0(p^n,\Z/p^n\Z)$-equivariant for the action via the first factor.
		\item Let $x_{\gamma}$ be the cusp of $\mathcal X^{\ast}_{\Gamma(p^n)}(0)_a$ over $x$ determined by $\gamma=\smallmat{a}{b}{0}{d}\in \Gamma_0(p^n,\Z/p^n\Z)$. 
		Then for any honest adic space $S$ over $K$, the $S$-points of $x_{\gamma}\colon \mathring{\D}_{n,x}\hookrightarrow \X_{\Gamma(p^n)}(\epsilon)_a$ correspond functorially to Tate curves $E=\Tate(q)$ with topologically nilpotent parameter $q_E\in\mathcal O(S)$, a $\Gamma^p$-structure corresponding to $x$, and the basis $(q_E^{d/p^n},q_E^{-b/p^n}\zeta_{p^n}^{a})$ of $E[p^n]$, where $q_E^{1/p^n}$ is the $p^n$-th root of $q_E$ determined by $x$.
		\item The reduction map $\Gamma_0(p^{n+1},\Z/p^{n+1}\Z)\rightarrow \Gamma_0(p^{n},\Z/p^{n}\Z)$ gives a Cartesian diagram
		\begin{equation*}
			\begin{tikzcd}[column sep = 2cm]
				\underline{\Gamma_0(p^{n+1},\Z/p^{n+1}\Z)} \times \D_{n+1,x} \arrow[d,hook] \arrow[r,"{(\gamma,q)\mapsto (\overline{\gamma},q)}"] & \underline{\Gamma_0(p^n,\Z/p^n\Z)}\times \D_{n,x} \arrow[d,hook] \\
				\mathcal X^{\ast}_{\Gamma(p^{n+1})}(\epsilon)_a \arrow[r] & \mathcal X^{\ast}_{\Gamma(p^n)}(\epsilon)_a.
			\end{tikzcd}
		\end{equation*}
	\end{enumerate}
\end{Lemma}

\begin{proof}
	\begin{enumerate}
		\item 	
	Arguing as in Lemma~\ref{l:Tate parameter spaces of X^{ast}_{Gamma_1(p^n)}}, it suffices to produce a splitting
		\begin{equation*}\label{diagram: lifting Tate curve from Gamma_0 to Gamma}
			\begin{tikzcd}
				& \Spec(\Z_p\llbracket q^{1/p^n}\rrbracket\otimes L_x) \arrow[d] \arrow[ld, dashed] \\
				X^{\ast}_{\Gamma(p^n)} \arrow[r] & X^{\ast}_{\Gamma_1(p^n)}
			\end{tikzcd}
		\end{equation*}
		which we construct as follows: Consider the Tate curve $T=\Tate(q^e)_R$ over $R=\Z_p\cc{q^{1/p^n}}\otimes \O_L$ with its Weil pairing $e_{p^n}\colon T[p^n]\times T[p^n]\to \mu_{p^n}$. Specialising at $q^{e/p^n}\in T[p^n]$, we obtain an isomorphism \[e(q^{e/p^n},-)\colon C_n\to \mu_{p^n}\]
		where $C_n$ is the canonical subgroup of $T$. The preimage of $\zeta_{p^n}$ gives the second basis vector of $T[p^n]$ of an anticanonical $\Gamma(p^n)$-structure on $T[p^n]$, defining the desired lift.
	By analytification we then obtain the map $g\colon \D_{n,x}\rightarrow \mathcal X^{\ast}_{\Gamma(p^n)}(\epsilon)_a$ which gives the Cartesian diagram as $\X^{\ast}_{\Gamma(p^n)}(\epsilon)_a\to\X^{\ast}_{\Gamma_0(p^n)}(\epsilon)_a$ is Galois with group $\underline{\Gamma_0(p^n,\Z/p^n\Z)}$.

	\item We have just seen that the cusp label $1\in \Gamma_0(p^n,\Z/p^n\Z)$ corresponds via $q^e\mapsto q_E$ to the isomorphism $\varphi\colon (\Z/p^n\Z)^2\to E[p^n]$ defined by the ordered basis $(q_E^{1/p^n},\zeta_{p^n})$. For the general case, we use that the action of $\gamma=\smallmat{a}{b}{0}{d}$ is given by
	\[
	\begin{tikzcd}[row sep = 0.2cm]
		{\smallvector{1}{0}} \arrow[r, "\gamma^{\vee}", maps to] & {\smallvector{d}{0}} \arrow[r, "\varphi", maps to] & q_E^{d/p^n} \\
		{\smallvector{0}{1}} \arrow[r, "\gamma^{\vee}", maps to]                 & {\smallvector{-b}{a}} \arrow[r, "\varphi", maps to]&  q_E^{-b/p^n}\zeta_{p^n}^a 
	\end{tikzcd}
	\]
	\item This follows from (1)\ and Lemma~\ref{Proposition: Tate parameter spaces in the Gamma_0-tower} as $X^{\ast}_{\Gamma(p^n)}\to X^{\ast}_{\Gamma(p^{n-1})}$ is equivariant for $\pi$.\qedhere
	\end{enumerate}
\end{proof}

All in all, we get the following description of Tate curve parameter spaces at finite level:

\begin{Proposition}\label{p: structure of Tate parameter spaces from tame level to Gamma(p^n)}
	Let $x$ be any cusp of $\mathcal X^{\ast}$
	Then depending on our choice of $\zeta_{p^n}\in K$, there is a canonical tower of Cartesian squares
		\begin{equation*}
			\begin{tikzcd}[column sep = 1.4cm]
				\underline{\Gamma_0(p^n,\Z/p^n\Z)}\times \D_{n,x} \arrow[d, hook,"\varphi"] \arrow[r,"\smallmat{a}{b}{0}{d}\mapsto d"] & \underline{(\Z/p^n\Z)^{\times}}\times \D_{n,x} \arrow[d, hook] \arrow[r] & \D_{n,x} \arrow[r] \arrow[d, hook] & \D_x \arrow[d, hook] \\
				\mathcal X^{\ast}_{\Gamma(p^n)}(\epsilon)_a \arrow[r] & \mathcal X^{\ast}_{\Gamma_1(p^n)}(\epsilon)_a \arrow[r] & \mathcal X^{\ast}_{\Gamma_0(p^n)}(\epsilon)_a \arrow[r] & \mathcal X^{\ast}(\epsilon).
			\end{tikzcd}
		\end{equation*}
\end{Proposition}
\begin{proof}
	The square on the right is Prop.~\ref{Proposition: Tate parameter spaces in the Gamma_0-tower}.(2).
	The square in the middle is Lemma~\ref{l:Tate parameter spaces of X^{ast}_{Gamma_1(p^n)}}.
	The Cartesian diagram on the left exists as a consequence of Lemma~\ref{l:Tate parameter spaces of X_Gamma(p^n)(epsilon)_a}.1 together with the fact that  $X^{\ast}_{\Gamma(p^n)}\to X^{\ast}_{\Gamma_1(p^{n})}$ is equivariant with respect to the map $ \smallmat{a}{b}{0}{d}\mapsto d$.
\end{proof}

Lemma~\ref{l:Tate parameter spaces of X_Gamma(p^n)(epsilon)_a} describes the $\Gamma_0(p^n,\Z/p^n\Z)$-action at the cusps of $\mathcal X^{\ast}_{\Gamma(p^n)}(\epsilon)_a$, but there is also an action of the larger group $\Gamma_0(p,\Z/p^n\Z)$. While the action of $\Gamma_0(p^n,\Z/p^n\Z)$ just permutes the different copies of $\D_{n,x}$ over $x$, the action of $\Gamma_0(p,\Z/p^n\Z)$ in general has a non-trivial effect on each of these Tate curve parameter space, because it also takes into account isomorphisms of Tate curves of the form $q^{1/p^n}\mapsto  \zeta_{p^n}q^{1/p^n}$, as we shall now discuss.

\begin{Proposition}\label{Proposition: the action of Gamma_0(p,Z/p^nZ) on cusps of Gamma(p^n)}
	Over any cusp  $x$ of $\mathcal X^{\ast}$, the $\Gamma_0(p,\Z/p^n\Z)$-action on  $\mathcal X^{\ast}_{\Gamma(p^n)}(\epsilon)_a$ restricts to an action on $\varphi\colon \underline{\Gamma_0(p^n,\Z/p^n\Z)}\times \D_{n,x}\hookrightarrow \mathcal X^{\ast}_{\Gamma(p^n)}(\epsilon)_a$  that can be described as follows: Equip $\underline{\Gamma_0(p,\Z/p^n\Z)}\times \D_{n,x}$ with a right action by $p\Z/p^n\Z$ given by $(\gamma,q)\mapsto (\gamma\smallmat{1}{0}{h}{1},\zeta^{h/e_x}_{p^n}q)$, then
	\[(\underline{\Gamma_0(p,\Z/p^n\Z)}\times \D_{n,x})/(p\Z/p^n\Z)=\underline{\Gamma_0(p^n,\Z/p^n\Z)}\times \D_{n,x}\]
	and the left action of $\Gamma_0(p,\Z/p^n\Z)$ is the natural left action via the first factor.
	
	Explicitly, for any $\gamma_1\in \Gamma_0(p,\Z/p^n\Z)$, the action is given by
	\begin{alignat*}{3}
	\gamma_1\colon &&\underline{\Gamma_0(p^n,\Z/p^n\Z)}\times \D_{n,x}&\;\xrightarrow{\sim}\;&& \underline{\Gamma_0(p^n,\Z/p^n\Z)}\times \D_{n,x}\\
	&& \gamma_2,q^{1/p^n} &\; \;\mapsto\; && \smallmat{\det( \gamma_3)/d_3}{b_3}{0}{d_3},\zeta_{p^n}^{-\frac{c_3}{d_3e_x}}q^{1/p^n}
	\end{alignat*}
	where $\gamma_3=\smallmat{a_3}{b_3}{c_3}{d_3}:=\gamma_1\cdot \gamma_2$.
\end{Proposition}
Here we recall that $e_x$ was introduced in Notation~\ref{n:e}.
\begin{proof}
	Recall that the reason for the pullback of $\D_{n,x}\hookrightarrow \mathcal X^{\ast}_{\Gamma_0(p)}$ to $\mathcal X^{\ast}_{\Gamma(p^n)}$ to be of the form $\underline{\Gamma_0(p^n,\Z/p^n\Z)}\times \D_{n,x}$ even though $\X_{\Gamma(p^n)}\rightarrow \X_{\Gamma_0(p)}$ has larger Galois group $\Gamma_0(p,\Z/p^n\Z)$ is that in the step from $\mathcal X^{\ast}$ to $\mathcal X^{\ast}_{\Gamma_0(p^n)}(\epsilon)_a$ the isomorphism 
	\begin{equation*}\label{eq:automorphisms of D_{n,x}}
	\psi_h\colon \D_{n,x}\rightarrow \D_{n,x}, \quad q^{1/p^n}\mapsto \zeta^{h}_{p^n}q^{1/p^n}
	\end{equation*}
	for any $h\in \Z/p^n\Z$ induces an isomorphism of moduli for the universal Tate curve $T=\Tate(q^{e_x})$ over $\D_{x}$ that sends the anti-canonical $\Gamma_0(p^n)$-level structure $\langle q^{{e_x}/p^n} \rangle$ to $\langle \zeta^{he_x}_{p^n}q^{{e_x}/p^n} \rangle$.
	For the action of $\Gamma_0(p,\Z/p^n\Z)$, this means the following: 
	
	Consider the Tate parameter space $\varphi(1)\colon \D_{n,x}\hookrightarrow \mathcal X^{\ast}_{\Gamma(p^n)}(\epsilon)_a$ at $1\in \Gamma_0(p,\Z/p^n\Z)$, associated to the isomorphism $\alpha\colon (\Z/p^n\Z)^2\to T[p^n]$ defined by the ordered basis $(q^{e_x/p^n},\zeta_{p^n})$. Then the action of $\gamma_1=\smallmat{1}{0}{h}{1}$ sends this to the isomorphism $\alpha\circ \gamma^\vee$ defined by $(1,0)\mapsto \zeta_{p^n}^{-h}q^{e_x/p^n}$ and $(0,1)\mapsto \zeta_{p^n}$. The isomorphism $\psi_{-h/e_x}$ identifies this with the basis $(q^{1/p^n},\zeta_{p^n})$:
	\begin{equation*}
		\begin{tikzcd}
			q^{1/p^n}\arrow[d,mapsto]&\langle q^{e_x/p^n}\rangle  \arrow[d] \arrow[r] & \mathring \D_{n,x}\arrow[d,"\psi_{-h/e_x}"'] \arrow[r, hook,"\varphi(1)"] & \X_{\Gamma(p^n)}(\epsilon)_a \arrow[d, "\gamma_1"] \\
			\zeta_{p^n}^{-h/e_x}q^{1/p^n}&\langle \zeta^{-h}_{p^n}q^{e_x/p^n}\rangle \arrow[r] & \mathring{\D_{n,x}} \arrow[r, hook,"\varphi(1)"] & \X_{\Gamma(p^n)}(\epsilon)_a.
		\end{tikzcd}
	\end{equation*}
	The action of $\smallmat{1}{0}{h}{1}$ on the component $\{1\}\times \D_{n,x}$ of $\underline{\Gamma_0(p^n,\Z/p^n\Z)}\times \D_{n,x}$ defined by $1\in\Gamma_0(p^n,\Z/p^n\Z)$ is thus given by $\psi_{-h/e_x}$.

	In general, in order to describe the action of $\gamma_1$ on the component $\{\gamma_2 \}\times \D_{n,x}$, it suffices to compute the action of $\gamma_3=\gamma_1\cdot \gamma_2$ on $\{1\}\times \D_{n,x}$, since we have a commutative diagram
	\[\begin{tikzcd}
		{\{1 \}\times \D_{n,x}} \arrow[d, "\gamma_2"] \arrow[r, dotted] \arrow[rd, "\gamma_3"] & {\{1 \}\times \D_{n,x}} \arrow[d, dotted] \\
		{\{\gamma_2 \}\times \D_{n,x}} \arrow[r, "\gamma_1"]                                   & {\{\gamma_3 \}\times \D_{n,x}}.
	\end{tikzcd}\]
	One can now decompose $\gamma_3$ into the actions which we have already computed, using
	\begin{equation}\label{equation: decomposition of Gamma_0(p) into Gamma_0(p^n){1}{0}{ast}{1}}
	\smallmat{a}{b}{c}{d}=\smallmat{\det(\gamma)/d}{b}{0}{d}\smallmat{1}{0}{c/d}{1}.
	\end{equation}
	Applying this to $\gamma_3$, we get the desired action by equivariance of $\varphi$ under $\Gamma_0(p,\Z/p^n\Z)$.
\end{proof}
\section{Adic theory of cusps at infinite level}\label{s:adic-cusps-infinite}
We now pass to infinite level, starting with $\Xa_{\Gamma_0(p^\infty)}(\epsilon)_a\sim\varprojlim \Xa_{\Gamma_0(p^n)}(\epsilon)_a$. We first note:

\begin{Lemma}\label{l: moduli interpretation of modular curve at infinite level}
	Let $(R,R^+)$ be a perfectoid $K$-algebra. Then the set $\X_{\Gamma_0(p^\infty)}(\epsilon)_a(R,R^+)$ is in functorial bijection with isomorphism classes of triples $(E,\alpha_N,(G_n)_{n\in \mathbb N})$ of elliptic curves $E$ over $R$ that are $\epsilon$-nearly ordinary, together with a $\Gamma^p$-structure $\alpha_N$ and a collection of anticanonical cyclic subgroups $G_n\subseteq E[p^n]$ of order $p^n$ for all $n$ that are compatible in the sense that $G_n = G_{n+1}[p^n]$. Equivalently, one could view $G=(G_n)_{n\in \mathbb N}$ as a p-divisible subgroup of $E[p^\infty]$ of height 1 such that $D_1$ is anticanonical.
\end{Lemma}
\begin{proof}
	Since $(R,R^+)$ is perfectoid, one has 
	\[\mathcal X_{\Gamma_0(p^\infty)}(\epsilon)_a(R,R^+)= \varprojlim_{n\in \N} \mathcal X_{\Gamma_0(p^n)}(\epsilon)_a(R,R^+)\]
	by \cite[Prop~2.4.5]{ScholzeWeinstein}. The statement thus follows from Lemma~\ref{Lemma: adic moduli interpretation of X_Gamma^an}.
\end{proof}
\begin{Definition}
	 We shall call the $p$-divisible group $D$ an {anticanonical $\Gamma_0(p^\infty)$-structure}.
\end{Definition}
We wish to study the cusps at infinite level, by which we mean the following:

\begin{Definition}
	By the cusps of $\XaGea{0}{\infty}$, $\XaGea{1}{\infty}$ etc.\ we mean the preimage of the divisor of cusps of $\X^{\ast}$ endowed with its induced reduced structure. 
\end{Definition}
We note that at infinite level, the cusps are in general a profinite Zariski-closed subspace of the respective infinite level modular curve.
However, at level $\Gamma_0(p^\infty)$, we will see that the map $\mathcal X^{\ast}_{\Gamma_0(p^\infty)}(\epsilon)_a\rightarrow \mathcal X^{\ast}$ isomorphically identifies the cusps of $\mathcal X^{\ast}_{\Gamma_0(p^\infty)}(\epsilon)_a$ and those of $\X^{\ast}$.
\subsection{The perfectoid Tate parameter space at level \texorpdfstring{$\Gamma_0(p^\infty)$}{G_0(p^infty)}}\label{s:infinite-level-Gamma_0}

We start our discussion by having a closer look at the cusps in the anticanonical tower.

Like before, we fix a cusp $x$ of $\X^{\ast}$ and let $L$ be the field of definition of the associated Tate curve, like in Def.~\ref{d:field-of-definition-of-Tate-curve}. In particular, if $K$ contains a primitive $N$-th unit root, we simply have $L=K$.
By Lemma~\ref{Proposition: Tate parameter spaces in the Gamma_0-tower}, there is for $x$ a tower of Cartesian squares
\begin{equation}\label{diagram: Tate parameter spaces in anticanonical tower, forgetful morphism version}
\begin{tikzcd}
\dots \arrow[r] & \D_2 \arrow[r]\arrow[d,hook] & \D_1 \arrow[d,hook] \arrow[r] & \D \arrow[d,hook] \\
\dots \arrow[r] &\mathcal X^{\ast}_{\Gamma_0(p^2)}(\epsilon)_a \arrow[r] & \mathcal X^{\ast}_{\Gamma_0(p)}(\epsilon)_a \arrow[r] & \mathcal X^{\ast}(\epsilon).
\end{tikzcd}
\end{equation}

\begin{Lemma}\label{the perfectoid parameter space of the Tate curve at infinite level}
	\begin{enumerate}
		\item The  perfectoid open disc $\D_\infty:=\Spa(L\langle q^{1/p^\infty}\rangle,\O_L\langle q^{1/p^\infty}\rangle)(|q|< 1)$ is a tilde-limit 
		$\D_\infty \sim \varprojlim_{n\in\N} \D_n$.
		\item Denote by $\mathcal O_{L}\llbracket q^{1/p^\infty}\rrbracket$ the $(p,q)$-adic completion of $\varinjlim_{n \in \N} \mathcal O_{L}\llbracket q^{1/p^n}\rrbracket$. Then $\D_\infty$ is the adic generic fibre of the formal scheme $\Spf(\mathcal O_{L}\llbracket q^{1/p^\infty}\rrbracket)$.
		\item The global sections of $\D_\infty$ are given by $\mathcal O^+(\D_\infty)= \O_L\llbracket q^{1/p^\infty}\rrbracket$ and
		\[\mathcal O(\D_\infty) = \Bigg \{ \sum_{n\in \Z[\frac{1}{p}]_{\geq 0}}a_nq^n \in L\llbracket q^{1/p^\infty}\rrbracket \Bigg |\quad\stackanchor{ |a_n|q^n\to 0\text{ for all }0\leq q<1, }{|a_n|\to 0 \text{ on bounded intervals}} \Bigg\}\]
		where the second condition means that for any $\delta> 0$ and for any bounded interval $I\subseteq \Z[\frac{1}{p}]_{\geq 0}$ there are only finitely many $n\in I$ such that $|a_n|> \delta$.
	\end{enumerate}
\end{Lemma}
\begin{proof}
	It is clear on global sections that 
	\[\Spa(L\langle q^{1/p^\infty}\rangle,\O_L\langle q^{1/p^\infty}\rangle)\sim\varprojlim_{n\in\N} \Spa(L\langle q^{p^n}\rangle,\O_L\langle q^{1/p^n}\rangle).\] 
	The first part follows from \cite[Prop.~2.4.3]{ScholzeWeinstein} since $\D$ is the restriction to the open subspace defined by $|q|<1$.
	More explicitly, this means that $\D_\infty$ is given by glueing the affinoid perfectoid unit discs of increasing radii $<1$ given by
	\[\D_\infty(|q|^{p^m}\leq|\varpi|) = \Spa(L\langle (q/\varpi^{1/p^m})^{1/p^\infty} \rangle,\mathcal O_L\langle (q/\varpi^{1/p^m})^{1/p^\infty} \rangle)\]
	for all $m\in \N$.  These can be obtained by rescaling the perfectoid unit disc. Computing the intersection of their respective global functions gives  $\O(\D_\infty)$ and $\O^+(\D_\infty)$.
	
	Part (2) is not just a formal consequence as tilde-limits do not necessarily commute with taking generic fibres. But it follows from the construction: Let 
	\[S=\Spa(\mathcal O_{L}\llbracket q^{1/p^\infty}\rrbracket,\mathcal O_{L}\llbracket q^{1/p^\infty}\rrbracket)\]
	and consider the subspaces $S(|q|^{p^m}\leq |\varpi|\neq 0)$ which are rational because $(q^n,\varpi)$ is open. As usual, one shows that since $\mathcal O_{L}\llbracket q^{1/p^\infty}\rrbracket$ has ideal of definition $(q,\varpi)$, the element $|q(x)|$ must be cofinal in the value group for any $x \in S$. This shows that
	\[S^{\ad}_\eta =S(|\varpi|\neq 0) =\bigcup S(|q|^{p^m}\leq |\varpi|\neq 0).\]
	Let $B_m^{+}=\O^+(S(|q|^{p^m}\leq |\varpi|\neq 0))$, then as $q^{p^m}/\varpi \in B_m^{+}$, we have $(q,\varpi)^{p^m}\subseteq(\varpi)$ and the ring $B_m^{+}$ thus has the $\varpi$-adic topology. From this one deduces that $B_m^{+} = \mathcal O_L\langle q^{1/p^\infty}/\varpi^{1/mp^\infty} \rangle$ and thus the spaces $S(|q|^{p^m}\leq |\varpi|\neq 0)$ and $\D_\infty(|q|^{p^m}\leq |\varpi|\neq 0)$ coincide.
\end{proof}
\begin{Remark}\label{Remark: D_infty is not affinoid}
	Note that $\D_\infty$ is not affinoid, even though it is the generic fibre of an affine formal scheme, as we did not equip $\O_L\llbracket q^{1/p^\infty}\rrbracket$ with the $p$-adic topology.
\end{Remark}
\begin{Definition}
	 The origin in $\D_\infty$ is the closed point $x\colon \Spa(L,\mathcal O_L)\rightarrow \D_\infty$ where $q=0$. By removing this point, we obtain a space $\mathring{\D}_\infty := \D_\infty\backslash \{x\}$ that satisfies $\mathring{\D}_\infty \sim \varprojlim \mathring{\D}_n$.
\end{Definition}
\begin{Definition}
	Let $\overline{\D}_\infty:=\Spa(\O_L\llbracket q^{1/p^\infty}\rrbracket_p[\tfrac{1}{p}],\O_L\llbracket q^{1/p^\infty}\rrbracket_p)$ with the $p$-adic topology on $\O_L\llbracket q^{1/p^\infty}\rrbracket_p$ (see  Def.~\ref{d: OK bb q^1/p^infty _p}). Then it is clear from the definition that $\overline{\D}_\infty\sim \varprojlim_{q\mapsto q^p} \overline{\D}$.
\end{Definition}
We are now ready to discuss cusps at infinite level and the corresponding Tate curves.

\begin{Proposition}\label{p: the Tate parameter space at infinite level Gamma_0(p^infty)}
	Fix a cusp $x$ of $\mathcal X^{\ast}$, with corresponding cusps in the anticanonical tower.
	\begin{enumerate}
		\item  The open immersions $\D_n\hookrightarrow \mathcal X^{\ast}_{\Gamma_0(p^n)}(\epsilon)_a$ over $x$ in the limit $n\to \infty$ give rise to an open immersion
		$\D_\infty\hookrightarrow \mathcal X^{\ast}_{\Gamma_0(p^\infty)}(\epsilon)_a$ that fits into a Cartesian diagram
		\begin{equation*}
			\begin{tikzcd}
				\D_\infty \arrow[d] \arrow[r, hook] &	\overline{\D}_\infty \arrow[d] \arrow[r] & \mathcal X^{\ast}_{\Gamma_0(p^\infty)}(\epsilon)_a \arrow[d] \\
				\D \arrow[r, hook] &\overline{\D} \arrow[r, hook] & \mathcal X^{\ast}(\epsilon).
			\end{tikzcd}
		\end{equation*}
		\item Consider the restriction $\mathring{\D}_\infty\hookrightarrow \X_{\Gamma_0(p^\infty)}(\epsilon)_a$. For any perfectoid $K$-algebra $(R,R^+)$, \[\mathring{\D}_\infty(R,R^{+})\subseteq  \X_{\Gamma_0(p^\infty)}(\epsilon)_a(R,R^{+})\] 
		is in functorial bijection with isomorphism classes of triples $(E,\alpha_N,(q_E^{1/p^n})_{n\in \N})$ where $E\cong\Tate(q_E)$ is a Tate curve over $R$ for some topologically nilpotent unit $q_E\in R$, where $\alpha_N$ is a $\Gamma^p$-level structure, and $(q_E^{1/p^n})_{n\in \N}$ is a compatible system of $p^n$-th roots of $q_E$, defining an anticanonical $\Gamma_0(p^\infty)$-structure on $E$.
	\end{enumerate}
\end{Proposition}
Exactly as before, we shall also write $\mathcal D_{\infty,x}\hookrightarrow \mathcal X^{\ast}_{\Gamma_0(p^\infty)}(\epsilon)_a$ for the open immersion in the proposition if we wish to emphasize the cusp $x$ we are working over.
\begin{proof}
	The map $\D_\infty\rightarrow \mathcal X^{\ast}_{\Gamma_0(p^\infty)}(\epsilon)_a$  is induced from Prop.~\ref{Proposition: Tate parameter spaces in the Gamma_0-tower} and Prop.~\ref{the perfectoid parameter space of the Tate curve at infinite level} by the universal property of the perfectoid tilde-limit. The outer square in part (1) is now Cartesian using \cite[Prop~2.4.3]{ScholzeWeinstein},  and the fact that the squares in diagram~\eqref{diagram: Tate parameter spaces in anticanonical tower, forgetful morphism version} are Cartesian. 
	
	It is clear that the left square is Cartesian.
	To show that the right square is as well,
	it now suffices to prove this away from the cusps, where it follows from the relative moduli interpretation: Giving an anticanonical $\Gamma_0(p^\infty)$-level structure on $\Tate(q)$ over some $\O_L\llbracket q\rrbracket$-algebra where $q$ is invertible corresponds to giving a system of $p^n$-th roots of $q$.
	
	The moduli interpretation of $\mathring{\D}_\infty$ follows from diagram~\eqref{diagram: Tate parameter spaces in anticanonical tower, forgetful morphism version} and Cor.~\ref{Proposition: Tate parameter spaces in the Gamma_0-tower}.
\end{proof}
\subsection{Tate curve parameter spaces of $\XaGea{1}{\infty}$}\label{s:infinite-level-Gamma_1}
Next, we discuss the Tate parameter spaces in the pro-\'etale map $\XaGea{1}{\infty}\to \XaGea{0}{\infty}$. This is just a matter of pulling back the descriptions from finite level: Let
\[\mathcal X^{\ast}_{\Gamma_1(p^n)\cap \Gamma_0(p^\infty)}(\epsilon)_a:=\XaGea{1}{n}\times_{\XaGea{0}{n}}\XaGea{0}{\infty}.\]

\begin{Lemma}\label{l: Tate parameter spaces of X^ast_Gamma_1(p^m)cap Gamma_0(p^infty)}
	Let $x$ be any cusp of $\mathcal X^{\ast}$. Let $n\in \Z_{\geq 0}$.
	\begin{enumerate}
		\item The map $\mathcal X^{\ast}_{\Gamma_1(p^n)\cap\Gamma_0(p^\infty)}(\epsilon)_a\to \XaGea{1}{n}$ restricts to an isomorphism on the cusps.
		\item There are canonical Cartesian cubes
		\begin{equation*}
			\begin{tikzcd}[row sep ={1cm,between origins}, column sep ={2.1cm,between origins}]
			& \underline{(\Z/p^{n+1}\Z)^{\times}}\times \D_{n+1}  \arrow[rr]\arrow[dd,hook] &  & \underline{(\Z/p^n\Z)^{\times}}\times \D_{n,x} \arrow[dd, hook] \arrow[rr] &  & \D_{n,x} \arrow[dd, hook] \\
			\underline{(\Z/p^{n+1}\Z)^{\times}}\times \D_{\infty,x} \arrow[dd, hook] \arrow[rr,crossing over] \arrow[ru] &  & \underline{(\Z/p^n\Z)^{\times}}\times \D_{\infty,x}  \arrow[rr,crossing over] \arrow[ru] &  & {\D_{\infty,x}} \arrow[dd, hook] \arrow[ru] &  \\
			& \mathcal X^{\ast}_{\Gamma_1(p^{n+1})}(\epsilon)_a  \arrow[rr]&  & \mathcal X^{\ast}_{\Gamma_1(p^{n})}(\epsilon)_a \arrow[rr] &  & \mathcal X^{\ast}_{\Gamma_0(p^n)}(\epsilon)_a \\
			\mathcal X^{\ast}_{\Gamma_1(p^{n+1})\cap \Gamma_0(p^\infty)}(\epsilon)_a  \arrow[rr] \arrow[ru] &  & \mathcal X^{\ast}_{\Gamma_1(p^n)\cap \Gamma_0(p^\infty)}(\epsilon)_a \arrow[rr]\arrow[from =uu, hook,crossing over] \arrow[ru] &  & \mathcal X^{\ast}_{\Gamma_0(p^{\infty})}(\epsilon)_a. \arrow[from =uu, hook,crossing over]\arrow[ru] & 
			\end{tikzcd}
		\end{equation*}
	\end{enumerate}
\end{Lemma}
\begin{proof}
	In part (2), the bottom faces are Cartesian by definition, the back faces are Cartesian by Lemma~\ref{l:Tate parameter spaces of X^{ast}_{Gamma_1(p^n)}}, the rightmost square is Cartesian by Prop.~\ref{p: the Tate parameter space at infinite level Gamma_0(p^infty)}, and the top faces are clearly also Cartesian. Thus all other faces are Cartesian.
	Part (1) follows immediately.
\end{proof}

We now take the limit $n\to \infty$ to get Tate parameter spaces for $\XaGea{1}{\infty}$: In doing so, we need to account for the fact that in the inverse limit, the divisor of cusps becomes a profinite set of points, rather than just a disjoint union of closed points. 

\begin{Definition}
	Let $S$ be a profinite set, and let $(S_i)_{i\in I}$ be a system of finite sets  with $S=\varprojlim_{i\in I} S_i$. Then we define $\underline{S}$ to be the unique perfectoid tilde-limit $\underline{S}\sim \varprojlim_{i\in I} \underline{S_i}$. This is independent of the choice of $S_i$:
	Explicitly, $\underline{S}$ is the affinoid perfectoid space 
	\[\underline{S} = \Spa(\Mapc(S,K),\Mapc(S,\O_K)).\]
\end{Definition}
\begin{Proposition}\label{Proposition: Tate parameter spaces at level Gamma_1(p^infty)}
	Let $x$ be a cusp of $\mathcal X^{\ast}$ with Tate parameter space $\D_{\infty,x}\hookrightarrow \mathcal X^{\ast}_{\Gamma_0(p^\infty)}(\epsilon)_a$.
	Then in the limit, the open immersions $\underline{(\Z/p^{n}\Z)^{\times}}\times \D_{\infty,x}\hookrightarrow \mathcal X^{\ast}_{\Gamma_1(p^n)\cap \Gamma_0(p^\infty)}(\epsilon)_a$ give a $\Z_p^\times$-equivariant open immersion $\underline{\Z_p^\times}\times \D_{\infty,x}\hookrightarrow \mathcal X^{\ast}_{\Gamma_1(p^\infty)}(\epsilon)_a$ that fits into a Cartesian diagram 
		\begin{equation*}
			\begin{tikzcd}
				\underline{\Z_p^\times}\times \D_{\infty,x} \arrow[d,hook] \arrow[r] & \D_{\infty,x} \arrow[d,hook]\\
				\mathcal X^{\ast}_{\Gamma_1(p^\infty)}(\epsilon)_a \arrow[r] & \mathcal X^{\ast}_{\Gamma_0(p^\infty)}(\epsilon)_a.
			\end{tikzcd}
		\end{equation*}
\end{Proposition}
\begin{Remark}
	Upon specialisation to the origin $\Spa(L,\O_L)\to \D_{\infty,x}$, this shows that the subspace of cusps of $\mathcal X^{\ast}_{\Gamma_1(p^\infty)}(\epsilon)_a$ over $x$ can be identified with $(\underline{\Z^\times_p})_{L}$, the base-change of the profinite adic space $\underline{\Z_p^\times}$ to $L$. In particular, for any $a \in \Z_p^\times$, specialisation at $a$ gives rise to a locally closed immersion $\D_{\infty,x}\hookrightarrow \mathcal X^{\ast}_{\Gamma_1(p^\infty)}(\epsilon)_a$ but in contrast to the case of $\Gamma_0(p^\infty)$ this is not going to be an open immersion due to the non-trivial topology on the cusps.
\end{Remark}
\begin{proof}
	This follows in the inverse limit over the front of the cubes in Lemma~\ref{l: Tate parameter spaces of X^ast_Gamma_1(p^m)cap Gamma_0(p^infty)}.2, since
	\begin{equation}\label{eq:tilde-limit-distributes-over-product}
	\underline{\Z_p^\times}\times \D_{\infty,x}\sim \varprojlim \underline{(\Z/p^n\Z)^\times}\times \D_{\infty,x}.
	\end{equation}
	That this holds is easy to verify, see for example \cite[Lemma~12.2]{perfectoid-covers-Arizona}.
\end{proof}

\subsection{Tate curve parameter spaces of $\XaGea{}{\infty}$}\label{s:infinite-level-Gamma}
Next, we look at the Tate parameter spaces in the pro-\'etale map $\XaGea{}{\infty}\to \XaGea{1}{\infty}$. As before, we do so by looking at the limit of the finite level morphisms
\[\Xa_{\Gamma(p^n)\cap \Gamma_0(p^\infty)}:=\XaGea{}{n}\times_{\XaGea{0}{n}}\XaGea{0}{\infty}\to\XaGea{0}{\infty}. \]
Combining the moduli descriptions of $\X_{\Gamma(p^n)}$ and Lemma~\ref{l: moduli interpretation of modular curve at infinite level}, we see:
\begin{Lemma}\label{l:moduli interpretation of modular curve of level Gamma(p^n) cap Gamma_0(p^infty)}
	Let $(R,R^+)$ be a perfectoid $K$-algebra. Then $\X_{\Gamma(p^n)\cap\Gamma_0(p^\infty)}(\epsilon)_a(R,R^{+})$ is in functorial bijection with isomorphism classes of tuples $(E,\alpha_N,G,\beta_n)$ with $E,\alpha_N,G$ as in Lemma~\ref{l: moduli interpretation of modular curve at infinite level} and  $\beta_n$  an isomorphism $(\Z/p^n\Z)^2\rightarrow E[p^n]$ such that $\alpha(1,0)$ generates $G_n$. 
\end{Lemma}
We have the following description of the cusps of $\mathcal X^{\ast}_{\Gamma(p^m)\cap\Gamma_0(p^\infty)}(\epsilon)_a$:
\begin{Lemma}\label{l: Tate parameter spaces of X^ast_Gamma(p^m)cap Gamma_0(p^infty)}
	Let $x$ be a cusp of $\mathcal X^{\ast}$.
	\begin{enumerate}
		\item The map $\mathcal X^{\ast}_{\Gamma(p^m)\cap\Gamma_0(p^\infty)}(\epsilon)_a\rightarrow \mathcal X^{\ast}_{\Gamma(p^m)}(\epsilon)_a$ induces an isomorphism on cusps. The cusps of $\mathcal X^{\ast}_{\Gamma(p^m)\cap\Gamma_0(p^\infty)}(\epsilon)_a$ over $x$ are thus parametrised by $\Gamma_0(p^n,\Z/p^n\Z)$ where  $\smallmat{a}{b}{0}{d}$ corresponds to the ordered basis $(q^{d/p^n},\zeta_{p^n}^{a}q^{-b/p^n})$ of $\Tate(q)[p^n]$.
		\item There is the following Cartesian diagram, 
		where the top map is given by $\smallmat{a}{b}{0}{d}\mapsto d$:
		\begin{equation*}
			\begin{tikzcd}
				\underline{\Gamma_0(p^n,\Z/p^n\Z)}\times \D_{\infty,x} \arrow[d,hook] \arrow[r] &\underline{(\Z/p^n\Z)^\times}\times \D_{\infty,x} \arrow[d,hook]\\
				\mathcal X^{\ast}_{\Gamma(p^n)\cap\Gamma_0(p^\infty)}(\epsilon)_a\arrow[r] & \mathcal X^{\ast}_{\Gamma_1(p^n)\cap\Gamma_0(p^\infty)}(\epsilon)_a.
			\end{tikzcd}
		\end{equation*}
		
	\end{enumerate}
\end{Lemma}
\begin{proof}
	Part (1) follows from Lemma~\ref{l:Tate parameter spaces of X_Gamma(p^n)(epsilon)_a} and Prop.~\ref{p: the Tate parameter space at infinite level Gamma_0(p^infty)} exactly like in Lemma~\ref{l: Tate parameter spaces of X^ast_Gamma_1(p^m)cap Gamma_0(p^infty)}.
	Part (2) follows from a similar Cartesian cube using the left square in Prop.~\ref{p: structure of Tate parameter spaces from tame level to Gamma(p^n)} and Lemma~\ref{l: Tate parameter spaces of X^ast_Gamma_1(p^m)cap Gamma_0(p^infty)}.
\end{proof}

\begin{Definition}
	Let $\Gamma_0(p^{\infty})=\smallmat{\Z_p^\times}{\Z_p}{0}{\Z_p^\times}$ be the subgroup of $\GL_2(\Z_p)$ of upper triangular matrices. This is a profinite group with $\Gamma_0(p^{\infty})=\varprojlim_n \Gamma_0(p^{n},\Z/p^n\Z)$.
\end{Definition}

\begin{Lemma}\label{l:moduli interpretation of Gamma_1(p^infty)(epsilon)_a}
	Let $(R,R^+)$ be a perfectoid $K$-algebra. Then $\X_{\Gamma(p^\infty)}(\epsilon)_a(R,R^+)$ is in functorial bijection with isomorphism classes of triples $(E,\alpha_N,\beta)$ of an $\epsilon$-nearly ordinary elliptic curve $E$ over $R$, together with a $\Gamma^p$-structure $\alpha_N$, and an isomorphism of $p$-divisible groups $\beta\colon  (\Q_p/\Z_p)^2\rightarrow E[p^\infty]$ over $R$ (or equivalently an isomorphism $\Z_p^2\rightarrow T_pE(R)$) such that the restriction of $\beta$ to the first factor is an anti-canonical $\Gamma_1(p^\infty)$-structure.
\end{Lemma}
\begin{proof}
	This is an immediate consequence of Lemma~\ref{l:moduli interpretation of modular curve of level Gamma(p^n) cap Gamma_0(p^infty)} and \cite[Prop.~2.4.5]{ScholzeWeinstein}.
\end{proof}

We are now ready to give the main result of this section, which summarises the discussion so far and moreover describes the cusps of $\mathcal X^{\ast}_{\Gamma(p^\infty)}(\epsilon)_a$. For the statement, let us recall that for any $n$, the Tate curve $\Tate(q)$ over $\D_\infty$ has an anticanonical ordered basis for $\Tate(q)[p^n]$ given by $(q^{1/p^n},\zeta_{p^n})$. In particular, an anticanonical ordered basis of the Tate module $T_p\Tate(q)$ is given by the compatible system $(q^{1/p^n})_{n\in\N}$, that we denote by $q^{1/p^\infty}$, and the compatible system $(\zeta_{p^n})_{n\in\N}$, that we denote by $\zeta_{p^\infty}$.
\begin{Theorem}\label{Theorem: Tate parameter spaces at level Gamma(p^infty)}
	Let $x$ be a cusp of $\mathcal X^{\ast}$ with corresponding morphism $\D_x\hookrightarrow  \mathcal X^{\ast}$.
	\begin{enumerate}
		\item We have a tower of Cartesian squares, where the top left map sends $\smallmat{a}{b}{0}{d}\mapsto d$:
		\begin{equation*}
		\begin{tikzcd}
		\underline{\Gamma_0(p^\infty)}\times \D_{\infty,x} \arrow[d,hook] \arrow[r] & \underline{\Z_p^\times}\times \D_{\infty,x} \arrow[d,hook] \arrow[r] & \D_{\infty,x} \arrow[d,hook] \arrow[r] & \D_{x}\arrow[d,hook]\\
		\mathcal X^{\ast}_{\Gamma(p^\infty)}(\epsilon)_a \arrow[r] &  X^{\ast}_{\Gamma_1(p^\infty)}(\epsilon)_a \arrow[r] & \mathcal \X^{\ast}_{\Gamma_0(p^\infty)}(\epsilon)_a \arrow[r]&\mathcal X^{\ast}(\epsilon) .
		\end{tikzcd}
		\end{equation*} 
		\item For any $\gamma=\smallmat{a}{b}{0}{d}\in \Gamma_0(p^\infty)$, the cusp of $\mathcal X^{\ast}_{\Gamma(p^\infty)}(\epsilon)_a$ obtained by specialising $\underline{\Gamma_0(p^\infty)}\times \D_{\infty,x}\hookrightarrow \mathcal X^{\ast}_{\Gamma(p^\infty)}(\epsilon)_a$ at  $\gamma$ is the one corresponding to the
		isomorphism $\Z_p^2\rightarrow T_p\Tate(q)$ defined by the ordered basis $(q^{d/p^\infty}, \zeta^{a}_{p^\infty}q^{-b/p^\infty})$ of $T_p\Tate(q)$.

	\end{enumerate}
\end{Theorem}
\begin{proof}
	\begin{enumerate}
		\item
		The only statement we have not yet proved is that the left square is Cartesian. But this follows from Lemma~\ref{l: Tate parameter spaces of X^ast_Gamma(p^m)cap Gamma_0(p^infty)}.(2) in the limit $n\to \infty$. Here we use that 
		\[ \underline{\Gamma_0(p^\infty)}\times \D_{\infty,x}\sim \varprojlim_{n\in\N}  \underline{\Gamma_0(p^n,\Z/p^n)}\times \D_{n,x}\]
		as well as the analogous statement from \eqref{eq:tilde-limit-distributes-over-product}, which hold by the same argument.
		\item This follows from Lemma~\ref{l: Tate parameter spaces of X^ast_Gamma(p^m)cap Gamma_0(p^infty)}.(1) in the limit.\qedhere
	\end{enumerate}
\end{proof}

We note the following easy consequence. The analogue of this for Siegel moduli spaces for dimension $g>1$ is shown in the proof of \cite[Lemma III.2.35]{torsion}.
\begin{Corollary}\label{c: over ordinary locus, Gamma to Gamma_1 is split}
	For any $n\in \N\cup \{\infty\}$, our choice of $\zeta_{p^n}$ induces a canonical isomorphism
	\[\mathcal X^{\ast}_{\Gamma(p^n)}(0)_a = \bigsqcup_{\Gamma(p^n)/\Gamma_1(p^n)}\mathcal X^{\ast}_{\Gamma_1(p^n)}(0)_a.\]
\end{Corollary}
\begin{proof}
	For $n=\infty$, there is away from the cusps a canonical section induced by $T_pE=T_pC\times T_pG$ and the canonical isomorphism $T_pC=T_pG^\vee$ induced by the Weil pairing. On Tate parameter spaces, one checks that this splitting is given by the map \[\underline{\Z_p^\times}\times \mathring{\D}_\infty \to \underline{\Gamma_0(p^\infty)}\times \mathring{\D}_\infty, \quad(a,q)\mapsto  \left(\smallmat{a}{0}{0}{a^{-1}},q\right).\] 
	This clearly extends over the puncture. Similarly for $n<\infty$.
\end{proof}

\subsection{The action of $\Gamma_0(p)$ on the cusps of $\mathcal X^{\ast}_{\Gamma(p^\infty)}(\epsilon)_a$}\label{s:action-of-Gamma_0}
Next, we discuss the full action of $\Gamma_0(p)$ on the Tate parameter spaces at infinite level.

Since the full $\GL_2(\Z/p^n\Z)$-action on each $\Xa_{\Gamma(p^n)}$ restricts to a $\Gamma_0(p,\Z/p^n\Z)$-action on $\XaGea{}{n}$ as discussed in Prop.~\ref{Proposition: the action of Gamma_0(p,Z/p^nZ) on cusps of Gamma(p^n)}, we see that the $\GL_2(\Z_p)$-action on  $\Xa_{\Gamma(p^\infty)}$ restricts to an action of $\Gamma_0(p)=\varprojlim_n \Gamma_0(p,\Z/p^n\Z)$ on $\XaGea{}{\infty}$.

\begin{Proposition}\label{p: description of the action of Gamma_0(p) on the Tate parameter spaces}
	Over any cusp  $x$ of $\mathcal X^{\ast}$, the $\Gamma_0(p)$-action on  $\mathcal X^{\ast}_{\Gamma(p^\infty)}(\epsilon)_a$ restricts to an action on $\underline{\Gamma_0(p^\infty)}\times \D_\infty\hookrightarrow \mathcal X^{\ast}_{\Gamma(p^\infty)}(\epsilon)_a$  where it can be described as follows: Equip $\underline{\Gamma_0(p)}\times \D_\infty$ with a right action by $p\Z_p$ via $(\gamma,q^{1/p^m})\mapsto (\gamma\smallmat{1}{0}{h}{1},\zeta^{h/e_x}_{p^m}q^{1/p^m})$ for  $h\in p\Z_p$, then
	\[(\underline{\Gamma_0(p)}\times \D_\infty)/p\Z_p=\underline{\Gamma_0(p^\infty)}\times \D_\infty\]
	as sheaves on $\mathbf{Perf}_K$ and the left action of $\Gamma_0(p)$ is the one induced by letting $\Gamma_0(p)$ act on the first factor of $\underline{\Gamma_0(p)}\times \D_\infty$.
	Explicitly, in terms of any $\gamma_1\in \Gamma_0(p)$, the action is given by
	\begin{alignat*}{3}
	\gamma_1\colon &&\underline{\Gamma_0(p^\infty)}\times \D_\infty&\;\xrightarrow{\sim}\;&& \underline{\Gamma_0(p^\infty)}\times \D_\infty\\
	&& \gamma_2,q^{1/p^m} &\; \;\mapsto\; && \smallmat{\det( \gamma_3)/d_3}{b_3}{0}{d_3},\zeta_{p^m}^{-\frac{c_3}{d_3e_x}}q^{1/p^m}.
	\end{alignat*}
	where $\gamma_3=\smallmat{a_3}{b_3}{c_3}{d_3}:=\gamma_1\cdot \gamma_2$.
\end{Proposition}
\begin{proof}
	That the action restricts to an action on $\underline{\Gamma_0(p^\infty)}\times \D_\infty$ is a consequence of Prop.~\ref{Proposition: the action of Gamma_0(p,Z/p^nZ) on cusps of Gamma(p^n)} in the limit over $n$. The same argument gives the explicit formula.
	
	It remains to verify the isomorphism of sheaves. For this we check that the following diagram commutes, where to ease notation, let $\Gamma_m:=\Gamma_0(p^m,\Z/p^m\Z)$ and $\Gamma'_m:=\Gamma_0(p,\Z/p^m\Z)$:
	\[
		\begin{tikzcd}[column sep = 0.4cm,row sep = {1.5cm,between origins},column sep = 1.1cm]
			\underline{\Gamma'_{n+1}}\times \D_{n+1} \arrow[d] \arrow[r] & \underline{\Gamma_{n+1}}\times \D_{n+1}, \arrow[d] \\
			\underline{\Gamma'_n}\times \D_n \arrow[r] & \underline{\Gamma_n}\times \D_n,
		\end{tikzcd}
		\quad
		\begin{tikzcd}[column sep = 0.4cm,row sep = {1.5cm,between origins}]
			\smallmat{a}{b}{c}{d},q^{1/p^{n+1}} \arrow[d, maps to] \arrow[r, maps to] & \smallmat{\det( \gamma)/d}{b}{0}{d},\zeta_{p^{n+1}}^{-\tfrac{c}{de_x}}q^{1/p^{n+1}} \arrow[d, maps to] \\
			\smallmat{a}{b}{c}{d},q^{1/p^n} \arrow[r, maps to] & \smallmat{\det( \gamma)/d}{b}{0}{d},\zeta_{p^{n}}^{-\tfrac{c}{de_x}}q^{1/p^n}
		\end{tikzcd}
	\]
	where $\gamma=\smallmat{a}{b}{c}{d}$ (we emphasize that on the right we describe the maps in terms of \textit{points} rather than \textit{functions}). This diagram is $p\Z_p$-equivariant when we endow the spaces on the left with the $p\Z_p$-actions from Prop.~\ref{Proposition: the action of Gamma_0(p,Z/p^nZ) on cusps of Gamma(p^n)}, and the spaces on the right with the trivial $p\Z_p$-action. The diagram is moreover equivariant for the $\Gamma_0(p)$-action on the left via the natural reduction maps.
	In the limit we therefore obtain a $p\Z_p$-invariant morphism
	\[\underline{\Gamma_0(p)}\times \D_\infty\rightarrow \underline{\Gamma_0(p^\infty)}\times \D_\infty,\]
	equivariant for the $\Gamma_0(p)$-action via the first factor on the left, and the action described in the proposition on the right. This induces a morphism of sheaves  
	\[(\underline{\Gamma_0(p)}\times \D_\infty)/p\Z_p\rightarrow \underline{\Gamma_0(p^\infty)}\times \D_\infty.\]
	On the other hand, the inclusion $\Gamma_0(p^\infty)\subseteq\Gamma_0(p)$ defines an inverse of this map. 
\end{proof}

\subsection{The Hodge--Tate period map on Tate parameter spaces}\label{s:action-of-HT}
Next, we give an explicit description of the Hodge--Tate map on Tate parameter spaces.

Recall that over the ordinary locus, the kernel of the Hodge--Tate map $T_pE\rightarrow \omega_E$ is the Tate module $T_pC$ of the canonical $p$-divisible subgroup, and thus the Hodge--Tate filtration is given by $T_pC\subseteq T_pE$. In particular, this means that $\pi_{\HT}(\mathcal X^{\ast}_{\Gamma(p^\infty)}(0))\subseteq \P^1(\Z_p)$.

When we further restrict to the anticanonical locus, the image lies in the points of the form $(a:1)\in \P^1(\Z_p)$ with $a\in \Z_p$. Denote by $B_1(0)\subseteq \P^1(\Z_p)$ the ball of radius $1$ inside the canonical chart $\A^1\subseteq \P^1$ around $(0:1)$, then the Hodge--Tate period map thus restricts to 
\[\pi_{\HT}(\mathcal X^{\ast}_{\Gamma(p^\infty)}(0)_a)\subseteq B_1(0)\subseteq \P^1(\Z_p).\]

\begin{Proposition}\label{proposition: Hodge--Tate period map on Tate parameter spaces}
	Let $x$ be a cusp of $\mathcal X^{\ast}$. Then the Hodge--Tate period map $\pi_{\HT}\colon \mathcal X^{\ast}_{\Gamma(p^\infty)} \rightarrow \P^1$ restricts on $	(\underline{\Gamma_0(p)}\times \D_{\infty,x})/p\Z_p\hookrightarrow  \mathcal X^{\ast}_{\Gamma(p^\infty)}(\epsilon)_a$ to the locally constant map
	\[
	(\underline{\Gamma_0(p)}\times \D_{\infty,x})/p\Z_p  \to\underline{\P^1(\Z_p)}\subseteq \P^1,\quad \big(\!\smallmat{a}{b}{c}{d},q\big)\mapsto (b:d).
	\]
\end{Proposition}
We deduce this from the following lemma:
\begin{Lemma}\label{l:constant-functions}
	Let $f\colon \D_\infty\rightarrow \A^1_K$ be a function such that $f$ is constant on $(C,\mathcal O_C)$-points with value $a\in L$. Then the corresponding $f\in\mathcal O(\D_\infty)$ is the constant $a\in L\subseteq \mathcal O(\D_\infty)$.
\end{Lemma}
\begin{proof}
	It suffices to prove this for the spaces $\D_\infty(|q|\leq \varpi^n)$. After rescaling, we are reduced to showing the lemma for $\D_\infty$ replaced by $\Spa(L\langle q^{1/p^\infty}\rangle,\mathcal O_L\langle q^{1/p^\infty}\rangle)$. 
	One can now argue like in the classical proof of the maximum principle: We can regard $f$ as a function
	\[f\in L\langle q^{1/p^\infty}\rangle,\quad f=\sum_{m\in \Z[\frac{1}{p}]_{\geq 0}}a_mq^{m}.\] 
	We need to prove that if $f((x^{1/p^i})_{i\in\N})=a$ for all $(x^{1/p^i})_{i\in\N}\in \mathcal \varprojlim_{x\mapsto x^p} \mathcal O_C$ then $f=a$. 
	
	After subtracting by $a=a_0$, we may assume that $f(x)=0$ for all $x\in \mathcal O_C$. 
	Suppose $f\neq 0$. The convergence condition on coefficients ensures that $\sup_{m\in \Z[\frac{1}{p}]}|a_m|>0$ is attained and after renormalising we may assume that $|f|=\max_{m\in \Z[\frac{1}{p}]}|a_m|=1$.
	Consider the reduction 
	\[r\colon \mathcal O_L\langle q^{1/p^\infty}\rangle \rightarrow k_L[q^{1/p^n}|n\in\N]\]
	modulo $\mathfrak m_L\subseteq \mathcal O_L$. After replacing $q\mapsto q^{p^m}$ we may assume that $r(f)\in k_L[q]$. As $\mathcal O_C$ is perfectoid, the projection $\varprojlim \O_C\rightarrow \O_C\rightarrow k_C$ to the residue field is surjective, and the assumption on $f$ now implies that $r(f)$ is a non-zero polynomial in $k_C[q]$ which evaluates to $0$ on all $q\in k_C$, a contradiction as $k_C$ is infinite.
\end{proof}
\begin{proof}[Proof of Prop.~\ref{proposition: Hodge--Tate period map on Tate parameter spaces}]
	By Lemma~\ref{l:constant-functions} it suffices to prove that for any $\gamma\in \Gamma_0(p^\infty)$, the map 
	\[\D_\infty\xrightarrow{q\mapsto (\gamma,q)} \underline{\Gamma_0(p^\infty)}\times \D_{\infty,x}\hookrightarrow \mathcal X^{\ast}_{\Gamma(p^\infty)}(\epsilon)_a\xrightarrow{\pi_{\HT}}\P^1\]
	is constant with image $b/d$. To see this, we use the moduli description  on $(C,\mathcal O_C)$-points: 
	
	On the ordinary locus, $\pi_{\HT}$ sends any isomorphism $\Z_p^2\rightarrow T_pE$ to the point of $\P^1(\Z_p)$ defined by the line $T_pC\subseteq T_pE$ where $C$ is the canonical $p$-divisible subgroup. By Thm.~\ref{Theorem: Tate parameter spaces at level Gamma(p^infty)}.(2), any $(C,\O_C)$-point of $\D_\infty\xrightarrow{(q\mapsto \gamma,q)} \underline{\Gamma_0(p^\infty)}\times \D_\infty$ corresponds to a Tate curve $E=\Tate(q_E)$ with an ordered basis of $T_pE$ given by $(e_1,e_2)=(q_E^{d/p^\infty}, \zeta^{a}_{p^\infty}q_E^{-b/p^\infty})$. Then (using additive notation on $T_pE$) \[be_1+de_2=q_E^{bd/p^\infty}\zeta^{ad}_{p^\infty}q_E^{-db/p^\infty}=\zeta^{ad}_{p^\infty}\]
	which spans the line $\langle \zeta_{p^\infty}\rangle = T_pC\subseteq T_pE$. Consequently, the image of $(\gamma,q)$ under $\pi_{\HT}$ is
	\[\pi_{\HT}(\gamma,q)=(b:d)=(b/d:1)\in \underline{(\Z_p^\times:1)} \subseteq \P^1(\Z_p).\]
	
	We conclude from this that the function $f\in\Map_{\cts}(\Gamma_0(p^\infty),\mathcal O(\D_\infty))$ defined by the restriction $\pi_{\HT}\colon \underline{\Gamma_0(p^\infty)}\times \D_\infty \to B(0)$ evaluates at $\gamma$ to $f(\gamma)=b/d$. Since this is true for all $\gamma\in \Gamma_0(p^\infty)$, we see that $f$ is given by a function in 
	\[\Map_{\cts}(\Gamma_0(p^\infty),\Z_p^\times)\subseteq \Map_{\cts}(\Gamma_0(p^\infty),\mathcal O(\D_\infty)).\]
	We conclude that $\pi_{\HT}$ has the desired description
	\[\underline{\Gamma_0(p^\infty)}\times \D_\infty \xrightarrow{(\gamma,q)\mapsto b/d} \underline{\Z_p^\times}=\underline{(\Z_p^\times:1)}\subseteq \P^{1}(\Z_p) \hookrightarrow \P^1, \quad  (\gamma,q)\mapsto (b/d:1).\qedhere \]
\end{proof}

\subsection{Tate parameter spaces of the modular curve at infinite level}
As an immediate consequence of the above, we can now consider the entire modular curve $\mathcal X^{\ast}_{\Gamma(p^\infty)}$. Recall that by the very construction in \cite{torsion}, this is the space $\GL_2(\Q_p)\XaGea{}{\infty}$ defined by glueing translates of $\XaGea{}{\infty}$. This allows us to deduce the parts (2)-(3) of Thm.\ \ref{t:cusps of X_Gamma(p^infty)}:
\begin{Theorem}\label{t:Main-Theorem-parts-2-3}
	Let $x\in \mathcal X^{\ast}$ be any cusp. Define a right action of $\Z_p$ on $\underline{\GL_2(\Z_p)}\times \D_{\infty,x}$ by $(\gamma,q^{1/p^n})\cdot h\mapsto (\gamma\smallmat{1}{0}{h}{1},q^{1/p^n}\zeta^{h/e_x}_{p^n})$. Then the quotient $(\underline{\GL_2(\Z_p)}\times \D_{\infty,x})/\Z_p$ exists as a perfectoid space, and there is a Cartesian diagram
\begin{equation*}
\begin{tikzcd}
(\underline{\GL_2(\Z_p)}\times \D_{\infty,x})/\Z_p \arrow[d,hook] \arrow[r] & \D_x\arrow[d,hook]\\
\mathcal X^{\ast}_{\Gamma(p^\infty)} \arrow[r] &\mathcal X^{\ast}.
\end{tikzcd}
\end{equation*}
where the top map is induced by the projection from the second factor.
The left map is $\GL_2(\Z_p)$-equivariant for the left action on $(\underline{\GL_2(\Z_p)}\times \D_{\infty,x})/\Z_p$ via the first factor.

Under this description, the fibres of the canonical and anticanonical locus are precisely
\begin{equation}\label{eq:in-thm-descr-of-Tate-in-antican-and-can}
\begin{tikzcd}[row sep = 0cm]
\Bigg(\underline{\smallmat{\Z_p}{\Z_p}{\Z_p}{\Z_p^{\times}}}\times\D_{\infty,x}\Bigg)/\Z_p\arrow[r,hook]& \mathcal X^{\ast}_{\Gamma(p^\infty)}(\epsilon)_a\\
\Bigg(\underline{\smallmat{\Z_p}{\Z_p^\times}{\Z_p^\times}{p\Z_p}}\times\D_{\infty,x}\Bigg)/\Z_p\arrow[r,hook]& \mathcal X^{\ast}_{\Gamma(p^\infty)}(\epsilon)_c.
\end{tikzcd}
\end{equation}

\end{Theorem}
\begin{proof}
	We need to translate the $\Gamma_0(p)$-equivariant open immersion from Prop~\ref{p: description of the action of Gamma_0(p) on the Tate parameter spaces}
	\[(\Gamma_0(p)\times \D_{\infty,x})/p\Z_p\hookrightarrow \mathcal X^{\ast}_{\Gamma(p^\infty)}(\epsilon)_a\]
	according to the $\GL_2(\Z_p)$-action on the right hand side.
	
	We first rewrite the left hand side: We have
	\[\Gamma_0(p)\smallmat{1}{0}{\Z_p}{1}=\smallmat{\Z_p}{\Z_p}{\Z_p}{\Z_p^\times},\]
	and by extending the $p\Z_p$-action to a $\Z_p$-action in the natural way, we get the equivalent description of anticanonical Tate parameter spaces stated in \eqref{eq:in-thm-descr-of-Tate-in-antican-and-can} in the theorem.

	Next, we note that we may without loss of generality replace $\mathcal X^{\ast}_{\Gamma(p^\infty)}$ by 
	\[\mathcal X^{\ast}_{\Gamma(p^\infty)}(0)=\mathcal X^{\ast}_{\Gamma(p^\infty)}(0)_a\sqcup \mathcal X^{\ast}_{\Gamma(p^\infty)}(0)_c.\]
	To simplify the discussion of translates,	
	we  introduce an auxiliary open subspace 
	\[\mathcal X^{\ast}_{\Gamma(p^\infty)}(0)^c_a\subseteq \mathcal X^{\ast}_{\Gamma(p^\infty)}(0)_a.\]
	parametrising isomorphisms $\alpha\colon \Z_p^2\to T_pE$ such that $\alpha(0,1)\bmod p$ generates the canonical subgroup (``first basis vector anticanonical, second canonical''). More precisely, this subspace can be constructed as follows: 
	 According to Cor.~\ref{c: over ordinary locus, Gamma to Gamma_1 is split}, there is a canonical splitting
	 \[\mathcal X_{\Gamma_1(p)}(0)_a\to \mathcal X_{\Gamma(p)}(0)_a\]
	 that identifies the image with a component of $\mathcal X_{\Gamma(p)}(0)_a$. Let $\mathcal X^{\ast}_{\Gamma(p)}(0)^c_a$ be the finite union of the $\smallmat{(\Z/p\Z)^{\times}}{0}{0}{1}$-translates of the image. Then $\mathcal X^{\ast}_{\Gamma(p^\infty)}(0)^c_a$ is defined as the pullback
	 \begin{center}
	 \begin{tikzcd}[row sep = 0.15cm]
	 	\mathcal X^{\ast}_{\Gamma(p^\infty)}(0)^c_a \arrow[d,"\cap" description] \arrow[r] &  \arrow[d,"\cap" description] \mathcal X^{\ast}_{\Gamma(p)}(0)^c_a\\
	 	\mathcal X^{\ast}_{\Gamma(p^\infty)}(0)_a \arrow[r]                  & \mathcal X^{\ast}_{\Gamma(p)}(\epsilon)_a.          
	 \end{tikzcd}
	 \end{center}
	It is clear from this definition that $	\mathcal X^{\ast}_{\Gamma(p^\infty)}(0)^c_a$ defines an open and closed subspace of 	$\mathcal X^{\ast}_{\Gamma(p^\infty)}(0)_a$. By tracing the Tate parameter spaces through the construction, we moreover see that their fibre over $\mathcal X^{\ast}_{\Gamma(p^\infty)}(0)^c_a$ is
	\begin{equation}\label{eq:Tate paramter space for antican-can}
	\Bigg(\underline{\smallmat{\Z_p^\times}{p\Z_p}{\Z_p}{\Z_p^\times}}\times \mathcal D_{\infty,x}\Bigg)/\Z_p \hookrightarrow \mathcal X^{\ast}_{\Gamma(p^\infty)}(0)^c_a
	\end{equation}
	We can now identify $\mathcal X^{\ast}_{\Gamma(p^\infty)}(0)_a$ with the finite union of translates
	\begin{equation}\label{eq:equivar-antican-case}
	\smallmat{1}{\Z_p}{0}{1}\mathcal X^{\ast}_{\Gamma(p^\infty)}(0)^c_a=\mathcal X^{\ast}_{\Gamma(p^\infty)}(0)_a.
	\end{equation}
	Indeed, away from the cusps this follows on moduli functors using Lemma~\ref{l:moduli interpretation of Gamma_1(p^infty)(epsilon)_a}, whereas over the cusps it follows from the above explicit description using
	\[\smallmat{1}{\Z_p}{0}{1} \smallmat{\Z_p^\times}{p\Z_p}{\Z_p}{\Z_p^\times}=\smallmat{\Z_p}{\Z_p}{\Z_p}{\Z_p^\times}. \]
	
	On the other hand, inside $\mathcal X^{\ast}_{\Gamma(p^\infty)}$  we have an identification
	\begin{equation}\label{eq:equivar-can-case}
	\smallmat{0}{1}{1}{0}\mathcal X^{\ast}_{\Gamma(p^\infty)}(0)^c_a=\mathcal X^{\ast}_{\Gamma(p^\infty)}(0)_c.
	\end{equation}
	Indeed, one can first check this for $\mathcal X^{\ast}_{\Gamma(p)}(0)_c$ on moduli functors, extend to compactifications, and then pull back to infinite level.
	On Tate parameter spaces, the identity  \[\smallmat{0}{1}{1}{0}\smallmat{\Z_p^\times}{p\Z_p}{\Z_p}{\Z_p^\times}=\smallmat{\Z_p}{\Z_p^\times}{\Z_p^\times}{p\Z_p}\]
	therefore gives the desired open immersion onto a neighbourhood of the cusps over $x$
	\[
	\Bigg(\underline{\smallmat{\Z_p}{\Z_p^\times}{\Z_p^\times}{p\Z_p}}\times\D_{\infty,x}\Bigg)\slash\Z_p\hookrightarrow \mathcal X^{\ast}_{\Gamma(p^\infty)}(0)_c.\]
	Taking the disjoint union of the morphisms in~\eqref{eq:in-thm-descr-of-Tate-in-antican-and-can}, we get the desired description. 
	
	To check $\GL_2(\Z_p)$-equivariance, we note that every $\gamma\in \GL_2(\Z_p)$ can be decomposed into $\gamma_1\cdot \gamma_2$ where $\gamma_2\in \smallmat{\Z_p^\times}{p\Z_p}{\Z_p}{\Z_p^\times}$ and either $\gamma_1\in \smallmat{1}{\Z_p}{0}{1}$ or $\gamma_1=\smallmat{0}{1}{1}{0}$. Since the open immersion in \eqref{eq:Tate paramter space for antican-can} is $\smallmat{\Z_p^\times}{p\Z_p}{\Z_p}{\Z_p^\times}$-equivariant, it thus suffices to check this for $\gamma_1$, for which this follows from equivariance in the anticanonical case \eqref{eq:equivar-antican-case}, and glueing in the canonical case \eqref{eq:equivar-can-case}.
\end{proof}
\begin{proof}[Proof of Thm.~\ref{t:cusps of X_Gamma(p^infty)}.(3)]
	This follows from Prop.~\ref{proposition: Hodge--Tate period map on Tate parameter spaces} by $\GL_2(\Z_p)$-equivariance of $\pi_{\HT}$.
\end{proof}	

\section{Modular curves in characteristic $p$}\label{s:cusps-in-char-p} 
We now switch to moduli spaces in characteristic $p$. We start by recalling the general setup:
Let $R$ be any $\F_p$-algebra. Recall from \S\ref{s:adic-cusps-finite} that $X_R$ denotes the modular curve over $R$ of tame level $\Gamma^p$, and $X_R^{\ast}$ denotes its compactification. We write $X_{R,\ord}\subseteq X_R$ for the affine open subscheme where the Hasse invariant $\Ha$ is invertible. Similarly, one defines $X^{\ast}_{R,\ord}\subseteq X^{\ast}_R$ which is also an affine open subspace.

In this section, we consider analytic modular curves over the perfectoid field $K^{\flat}$, the tilt of $K$. We fix a pseudo-uniformiser $\varpi^{\flat}$ such that $\varpi^{\flat\sharp}=\varpi$.
Following the notational conventions in \cite{torsion}, we shall denote modular curves over $K^{\flat}$ with a prime, e.g.\ $X':=X_{K^{\flat}}$ and $X'^{\ast}:=X^{\ast}_{K^{\flat}}$, to distinguish them from the modular curves over $K$.

Let $\mathfrak X'$ be the $\varpi^{\flat}$-adic completion of $X_{\O_{K^{\flat}}}$ and let $\X'$ be the analytification of $X'$ over $\Spa(K^{\flat},\O_{K^{\flat}})$. We analogously define $\mX'^{\ast}$ and $\X'^{\ast}$.
Like in characteristic $0$, for $0\leq \epsilon<1/2$ such that $|\varpi^{\flat}|^\epsilon\in |K^\flat|$ we denote by $\X'^{\ast}(\epsilon)$ the open subspace of $\Xpa$ where $|\Ha|\geq |\varpi^{\flat}|^{\epsilon}$.
Like before, this has a canonical formal model $\mathfrak X'^{\ast}(\epsilon)\to \mathfrak X'^{\ast}$. For any adic space $\mathcal Y\to \mathcal X'^{\ast}$ we write $\mathcal Y(\epsilon):=\mathcal Y\times_{\mathcal X'^{\ast}}\mathcal X'^{\ast}(\epsilon)$. Finally, let $\X'^{\ast}_{\ord}$ be the analytification of $X_{\ord}'^{\ast}=X_{K^{\flat},\ord}^{\ast}$.
\begin{Remark}
We recall that while the elliptic curves parametrised by $\X'(\epsilon)$ might have good supersingular \textit{reduction}, the condition on the Hasse invariant ensures that \textit{generically}, these elliptic curves are always ordinary. In other words, $\X'(\epsilon)\subseteq \X'_{\ord}$ even for $\epsilon>0$.
\end{Remark}
\subsection{Igusa curves}\label{subsection: adding Igusa structure in adic setting}\label{s:Igusa-curves}
In characteristic $p$, one has the Igusa moduli problem:

\begin{Definition}[\cite{KatzMazur}, Def.~12.3.1]
	Let $S$ be a scheme of characteristic $p$ and let $E$ be an elliptic curve over $S$. Consider the Verschiebung morphism $\ker V^n\colon E^{(p^n)}\rightarrow E$.
	An Igusa structure on $E$ is a group homomorphism $\phi\colon \Z/p^n\Z\rightarrow E^{(p^n)}(S)$ that is a Drinfeld generator of $\ker V^n$. This means that the Cartier divisor 
	\[\sum_{a\in \Z/p^n\Z}[\phi(a)]\subseteq E^{(p^n)}\]
	coincides with $\ker V^n$.
	
	The Igusa problem $[\Ig(p^n)]$ is the moduli problem defined by the functor sending $E|S$ to the set of Igusa structures on $E$. If $E|S$ is ordinary, the group scheme $\ker V^n$ is \'etale and naturally isomorphic to the Cartier dual $C_n^\vee$ of the canonical subgroup $C_n$. In particular, in this situation, an $\Ig(p^n)$-structure is the same as an isomorphism of group schemes 
	\[\underline{\Z/p^n\Z}\isomarrow C_n^\vee,\]
	or equivalently, an isomorphism of the Cartier duals $\mu_{p^n}\isomarrow C_n$.
\end{Definition}

	For any $n\geq 0$, the Igusa problem $[\Ig(p^n)]$ is relatively representable, finite and flat of degree $\varphi(p^n)$ over $\mathbf{Ell}|R$ by \cite{KatzMazur}, Thm.~12.6.1. In particular, the simultaneous moduli problem $[\Ig(p^n),\Gamma^p]$ is representable by a moduli scheme $X_{R,\Ig(p^n)}$ over $R$. The forgetful map 
	$X_{R,\Ig(p^n)}\rightarrow X_{R}$
	is finite and flat, and is an \'etale  $(\Z/p^n\Z)^{\times}$-torsor over the ordinary locus $X_{R,\ord}\subseteq X_{R}$.
	One defines by normalisation a compactification $X^{\ast}_{R,\Ig(p^n)}$. The morphism $X_{R,\Ig(p^n)}\rightarrow X_{R}$ then extends to
	\[X^{\ast}_{R,\Ig(p^n)}\rightarrow X^{\ast}_{R}\]
	which is still finite Galois with group $(\Z/p^n\Z)^{\times}$ over the ordinary locus.

	For any map $\Spec(R'\cc{q})\rightarrow X_{R}$ corresponding to a choice of $\Gamma^p$-structure on $\Tate(q^{e})$ over some cyclotomic extension $R'$ of $R$ and for some $1\leq e\leq N$, the canonical isomorphism \[\mu_{p^n}\isomarrow C_n(\Tate(q^e))\subseteq \Tate(q^e)[p^n]\]  induces a canonical lifting to a map $\Spec(R'\cc{q})\rightarrow X
	_{R,\Ig(p^n)}$.
	In particular, over any cusp $x$ of $X^{\ast}_{R}$, the subscheme of cusps of $X^{\ast}_{R,\Ig(p^n)}$ consists of $\varphi(p^n)$ disjoint copies of $x$.

\subsection{Tate parameter spaces in the Igusa tower}

Returning to our analytic setting over $K^{\flat}$, we let $X'^{\ast}_{\Ig(p^n)}:=X^{\ast}_{K^{\flat},\Ig(p^n)}$. We write $\mathfrak X^{\ast}_{\Ig(p^n)}$ for the $\varpi^\flat$-adic completion of $X^{\ast}_{\O_{K^{\flat}},\Ig(p^n)}$ and we write $\mathcal X'^{\ast}_{\Ig(p^n)}$ for the analytification of $X'^{\ast}_{\Ig(p^n)}$. We then get an open subspace $\mathcal X'^{\ast}_{\Ig(p^n)}(\epsilon)$. Since $\X'^{\ast}(\epsilon)\subseteq\mathcal X'^{\ast}_{\ord}$, the morphism $\mathcal X'^{\ast}_{\Ig(p^n)}(\epsilon)\to \mathcal X'^{\ast}(\epsilon)$ is a  finite \'etale $(\Z/p^n\Z)^\times$-torsor. 
Like in Lemma~\ref{Lemma: adic moduli interpretation of X_Gamma^an}, one can use that $X'_{\Ig(p^n)}$ is affine to show that these spaces represent the obvious adic moduli functors.

\begin{Definition}
	The inverse system of natural forgetful morphisms
	\[\dots \to\Xpa_{\Ig(p^{n+1})}(\epsilon)\to \Xpa_{\Ig(p^{n})}(\epsilon)\to \dots \to \Xpae\]
	is called the Igusa tower.
	Note that all transition maps in this inverse system are finite \'etale.
\end{Definition}
\begin{question}
	For $\epsilon=0$, one can show that this system has a sous-perfectoid (but not perfectoid) tilde limit $\Xpa_{\Ig(p^{\infty})}(0)\sim \varprojlim_{n\in\N}\Xpa_{\Ig(p^{n})}(0)$. Is this still true for $\epsilon>0$?
\end{question}
\begin{Definition}\label{d:field-of-definition-of-Tate-curve}
	As in characteristic $0$, by a cusp we shall mean a  (not necessarily geometrically) connected component of the closed subscheme $X'^{\ast}_{\Ig(p^n)}\backslash X'_{\Ig(p^n)}\subseteq X'^{\ast}_{\Ig(p^n)}$ with its induced reduced structure.
\end{Definition} 
Given a fixed cusp $x$ of $X'_{\Ig(p^n)}$, we denote by $L_{x}\subseteq K^{\flat}[\zeta_N]$ the field of definition of the associated Tate curve.
Then the completion of $X'^{\ast}_{\O_{K^{\flat}},\Ig(p^n)}$ along the integral  extension of $x$ is canonically of the form  $\Spf(\mathcal O_{L_{x}}\llbracket q\rrbracket)\rightarrow X'^{\ast}_{\Ig(p^n)}$. Upon $\varpi$-adic completion this becomes
\[\Spf(\mathcal O_{L_x}\llbracket q\rrbracket)\rightarrow \mathfrak X'^{\ast}_{\Ig(p^n)}\]
where $\mathcal O_{L_x}\llbracket q\rrbracket$ carries the $(\varpi^\flat,q)$-adic topology.
Denote by \[\D' \rightarrow \mathcal X'^{\ast}_{\Ig(p^n)}\]
 the adic generic fibre, a morphism of adic spaces over $\Spa(K^{\flat},\O_{K^{\flat}})$. Then like before, $\D'$ is the open unit disc over $L_{x}$ in the variable $q$.
 Exactly like in Lemma~\ref{l: Conrad's theorem on generic fibres of completions around the cusp} one sees:
\begin{Lemma}\label{l: Conrad's theorem on generic fibres of completions around the cusp for Ig(p^n)}
	The morphism $\D' \hookrightarrow \mathcal X'^{\ast}_{\Ig(p^n)}$ is an open immersion. 
\end{Lemma} 
If we want to indicate the dependence on the cusp $x$, we shall also call this $\mathcal D'_{x} \hookrightarrow \mathcal X'^{\ast}_{\Ig(p^n)}$. 

The following lemma explains how the above individual descriptions fit together for different cusps of $\mathcal X'^{\ast}_{\Ig(p^n)}$ lying over the same cusp of $\X'^{\ast}$.
\begin{Lemma}\label{l:Cartesian diagrams for Tate parameter spaces in the Igusa tower}
	Let $x$ be a cusp of $\X'^{\ast}$. Then there are Cartesian diagrams
	\begin{equation*}
		\begin{tikzcd}[column sep=2cm]
				\underline{\Z/p^{n+1}\Z}\times \D'_{x} \arrow[d,hook] \arrow[r]&\underline{\Z/p^n\Z}\times \D'_{x} \arrow[d,hook] \arrow[r] &
			\D'_x   \arrow[d,hook] \\
			\X'^{\ast}_{\Ig(p^{n+1})}(\epsilon) \arrow[r]&\mathcal X'^{\ast}_{\Ig(p^n)}(\epsilon) \arrow[r]& \mathcal X'^{\ast}(\epsilon).
		\end{tikzcd}
	\end{equation*}
\end{Lemma}
\begin{proof}
	Using the canonical lift described in \S\ref{s:Igusa-curves}, this can be seen exactly like in Lemma~\ref{l:Tate parameter spaces of X^{ast}_{Gamma_1(p^n)}},  based on the analogue of  Lemma~\ref{lemma for comparing schematic cusp to adic cusp} in this setting.
\end{proof}

Like in the $p$-adic case, there is also a larger, quasi-compact Tate curve parameter space:

\begin{Definition}\label{d:overline Dp}
	Let $\overline{\D}'=\overline{\D}'_x=\Spa(\O_{L_x}\llbracket q\rrbracket[\tfrac{1}{\varpi^{\flat}}],\O_{L_x}\llbracket q\rrbracket)$ where $\O_{L_x}\llbracket q\rrbracket$ is endowed with the $\varpi^{\flat}$-adic topology. Like in Lemma~\ref{l:overline{D} is sousperfectoid}, one sees that this is a sousperfectoid adic space with an open immersion $\D'=\cup_n\overline{\D}'(|q|^n\leq |\varpi^{\flat}|)\hookrightarrow \overline{\D}'$.
\end{Definition}
\begin{Lemma}
	For any cusp of $\Xpa_{\Ig(p^n)}(\epsilon)$, the map $\D'\hookrightarrow \Xpa_{\Ig(p^n)}(\epsilon)$ extends uniquely to a natural map $\overline{\D}'\to \Xpa_{\Ig(p^n)}(\epsilon)$. The fibre of the good reduction locus is $\overline{\D}'(|q|\geq 1)$.
\end{Lemma}
\begin{proof}
	Like in Lemma~\ref{l:full-Tate-curve-parameter-space}.
\end{proof}
\begin{Lemma}\label{Lemma: Frobenius gives Cartesian diagram on Tate parameter spaces of Igusa curves}
Let $n\in\Z_{\geq 0}$. For any cusp $x$ of $\mathcal X'^{\ast}_{\Ig(p^{n})}$, the following squares are Cartesian:
\begin{equation*}
	(1)
	\begin{tikzcd}
		\D' \arrow[d,"q\mapsto q^p"] \arrow[r, hook] &\overline{\D}' \arrow[d,"q\mapsto q^p"] \arrow[r] & \mathcal X'^{\ast}_{\Ig(p^{n})}(p^{-1}\epsilon) \arrow[d,"F_{\rel}"] \\
		\D' \arrow[r, hook] & 	\overline{\D}' \arrow[r] & \XpaIne.
	\end{tikzcd}
\quad (2) \begin{tikzcd}
			\mathcal X_{\Ig(p^{n+1})}'^{\ast}(p^{-1}\epsilon)\arrow[d,"F_\rel"] \arrow[r] & \mathcal X_{\Ig(p^n)}'^{\ast}(p^{-1}\epsilon) \arrow[d,"F_\rel"] \\
			\mathcal X_{\Ig(p^{n+1})}'^{\ast}(\epsilon) \arrow[r] & \mathcal X_{\Ig(p^n)}'^{\ast}(\epsilon).
		\end{tikzcd}
	\end{equation*}
\end{Lemma}
\begin{proof}
	It is clear that the diagrams commute by functoriality of the relative Frobenius morphism. The second diagram is Cartesian because the bottom map is \'etale. 
	
	To see that the first diagram is Cartesian, we first consider the outer square. For this it suffices to check this on $(C^{\flat},C^{\flat+})$-points because the horizontal compositions are open immersions. It is clear that the cusps correspond under $F_{\rel}$, and $q\mapsto q^p$ sends the origin to the origin. Away from the cusps, we can check on moduli interpretations that the square is Cartesian: The desired statement follows as $F_{\rel}$ sends $\Tate(q)$ to $\Tate(q^{p})$, and the $\Ig(p^n)$-structure $\langle q\rangle\subseteq \Tate(q)^{(p^n)}=\Tate(q^{p^n})$ to $\langle q^p\rangle\subseteq \Tate(q^p)^{(p^n)}=\Tate(q^{p^{n+1}})$. 
		
		This argument extends to $\overline{\D}'$, which (away from $x$) we may regard as the moduli space of Tate curves $\Tate(q)$ with level structure associated to $x$ over adic spaces $S$ over $\O_{L_x}\llbracket q\rrbracket[\tfrac{1}{p}]$. Lifts of maps $S\to \overline{\D}'$ along the right vertical morphism correspond to Tate curves $\Tate(q)^{(p^{-1})}=\Tate(q^{1/p})$  over $S$ whose base change along $F_{\rel}$ is $\Tate(q)$, and thus to $p$-th roots of $q\in \O^+(S)$. 
\end{proof}

\subsection{Perfections of Igusa curves}
In this section, we discuss perfectoid Igusa curves and their Tate curve parameter spaces. We first recall the perfection functor in characteristic $p$:

\begin{Definition}[\cite{torsion}, Def.~III.2.18]
	Let $\mathcal Y$ be an analytic adic space over $(K^\flat,\mathcal O_K^\flat)$. Then there is a perfectoid space $\mathcal Y^{\perf}$ over $(K^\flat,\mathcal O_K^\flat)$ such that $\mathcal Y^{\perf} \sim \varprojlim_{F_{\rel}} \mathcal Y$
	where we identify $\Y^{(p)}$ with $\Y$ using that $K^\flat$ is perfect. We call $\mathcal Y^{\perf}$ the perfection of $\mathcal Y$. The formation $\mathcal Y\mapsto \mathcal Y^{\perf}$ is functorial and defines a left-adjoint to the forgetful functor from perfectoid spaces over $K^{\flat}$ to analytic adic spaces over $K^{\flat}$.
\end{Definition}

In the case of $\mathcal Y=\X'^{\ast}(\epsilon)$, we can first take the inverse limit 
${\mathfrak X'^{\ast}}(\epsilon)^{\perf}:= \varprojlim_{F_\rel} \mathfrak X'^{\ast}(p^{-n}\epsilon)$
in the category of formal schemes. Its generic fibre is then the tilde limit
\[\mathcal {X'^{\ast}(\epsilon)}^{\perf} = \mathfrak {X'^{\ast}(\epsilon)}^{\perf}_{\eta} \sim \textstyle\varprojlim_{F_{\rel}} (\mathfrak X'^{\ast}(p^{-n}\epsilon))^{\ad}_{\eta}\]
by \cite[Prop. 2.4.2]{ScholzeWeinstein} and it is clear on any affine open formal subscheme of $\mathfrak X'^{\ast}(\epsilon)$ that this space is perfectoid. The analogous construction also works for $\mathcal X'^{\ast}_{\Ig(p^n)}(\epsilon)^{\perf}$.

\begin{Lemma}\label{Lemma moduli interpretation of Ig(p^n)^perf}
	Let $(R,R^+)$ be a perfectoid $K^{\flat}$-algebra. Then $\X'_{\Ig(p^n)}(\epsilon)^{\perf}(R,R^+)$ is in functorial bijection with isomorphism classes of triples $(E,\alpha_N,\beta_n)$ of $\epsilon$-nearly ordinary elliptic curves $E$ over $R$ with a $\Gamma^p$-structure $\alpha_N$ and an isomorphism of group schemes \[\beta_n\colon \underline{\Z/p^n\Z}\to \ker (V:E^{(p^n)}\to E)\cong  \ker (V:E\to E^{(p^{-n})})\subseteq E[p^n].\]
\begin{proof}
	By adjunction, we have $\X'_{\Ig(p^n)}(\epsilon)^{\perf}(R,R^+)=\X'_{\Ig(p^n)}(\epsilon)(R,R^+)$.
\end{proof}
\end{Lemma}

\begin{Lemma}\label{l: Cartesian diagram Igusa versus perfection}
	For any cusp $x$ of $\X'^{\ast}_{\Ig(p^n)}$, the perfection of the corresponding Tate parameter space $\D'\hookrightarrow \mathcal X'^{\ast}_{\Ig(p^n)}(\epsilon)$ fits into a Cartesian diagram
\begin{equation*}
\begin{tikzcd}
	\D'_\infty \arrow[d]\arrow[r,hook] & \overline{\D}'_\infty \arrow[d] \arrow[r, hook]& \XpaInep \arrow[d] \\
	\D' \arrow[r, hook] & 	\overline{\D}' \arrow[r] & \XpaIne.
\end{tikzcd}
\end{equation*}
Here $\overline{\D}'_\infty:=\overline{\D}'^{\perf}$, and $\D'_\infty:=\D'^{\perf}$ can be canonically identified with the open subspace of the perfectoid unit disc $\Spa(K^{\flat}\langle q^{1/p^\infty}\rangle, \O_{K^{\flat}}\langle q^{1/p^\infty}\rangle)$ defined by $|q|<1$.
\end{Lemma}
\begin{proof}
	This follows in the limit over the Cartesian diagrams from Lemma~\ref{Lemma: Frobenius gives Cartesian diagram on Tate parameter spaces of Igusa curves}.(1).
\end{proof}
\begin{Lemma}\label{l:perfectoid Igusa tower is proetale}
	The following diagram is Cartesian:
	\begin{equation*}
		\begin{tikzcd}
			\mathcal X_{\Ig(p^{n})}'^{\ast}(\epsilon)^{\perf} \arrow[d] \arrow[r] &\Xpaep \arrow[d] \\
			\mathcal X_{\Ig(p^{n})}'^{\ast}(\epsilon) \arrow[r] & \Xpae.
		\end{tikzcd}\end{equation*}
\end{Lemma}
\begin{proof}
	This follows from Lemma~\ref{Lemma: Frobenius gives Cartesian diagram on Tate parameter spaces of Igusa curves}.(2) in the limit over $F_{\rel}$ because perfectoid tilde-limits commute with fibre products.
\end{proof}
\begin{Definition}
	Consider the tower of affinoid perfectoid spaces with finite \'etale maps
	\[\dots \to\Xpa_{\Ig(p^{n+1})}(\epsilon)^\perf\to \XpaInep\to \dots \to \Xpaep. \]
	We denote by $\XpaIiep$ the unique affinoid perfectoid tilde-limit of this system.
\end{Definition}
\begin{Proposition}\label{p:Tate parameter spaces in perfectoid Igusa tower}
	Let $x$ be any cusp of $\Xpa$. Then there are natural Cartesian diagrams
		\begin{equation*}
		(1)	\begin{tikzcd}
				\underline{(\Z/p^n\Z)^\times}\times \D_\infty' \arrow[d,hook] \arrow[r] & \D'_\infty \arrow[d,hook] \\
				\XpaInep \arrow[r] & \Xpaep,
			\end{tikzcd}
	\quad(2)
	\begin{tikzcd}
		\underline{\Z_p^\times}\times \D_\infty' \arrow[d,hook] \arrow[r] & \D'_\infty \arrow[d,hook] \\
		\XpaIiep \arrow[r] & \Xpaep.
	\end{tikzcd}
	\end{equation*}
\end{Proposition}
\begin{proof}
	Part (1) follows from Lemma~\ref{l:Cartesian diagrams for Tate parameter spaces in the Igusa tower}, Lemma~\ref{l: Cartesian diagram Igusa versus perfection} and Lemma~\ref{l:perfectoid Igusa tower is proetale}  using the Cartesian cube that these three squares span.
	Part (2) follows in the inverse limit $n\to \infty$.
\end{proof}

\section{Tilting isomorphisms for modular curves}\label{s:tilting-isomorphism}
\subsection{The tilting isomorphism at level $\Gamma_0(p^\infty)$}
While so far we have studied modular curves in characteristic $0$ and $p$ separately, we now compare the two worlds via tilting. This is possible based on the following result:

\begin{thm}[{\cite[Cor.~III.2.19]{torsion}}]\label{theorem torsion paper: tilting the perfectoid modular curve}
	There is a canonical isomorphism
	\[\mathcal X^{\ast}_{\Gamma_0(p^\infty)}(\epsilon)_a^\flat \isomarrow \mathcal X'^{\ast}(\epsilon)^{\perf}.\]
\end{thm}
Let us recall how this is proved: Via $\O_K/p\cong\O_{K^{\flat}}/\varpi^\flat$
we have an identification of the reductions $\mathfrak X^{\ast}/p=\mathfrak X'^{\ast}/\varpi^{\flat}$ which by explicit inspection extends to a natural isomorphism
\begin{equation}\label{equation: identification of mXae mod p and mXpae mod t} 
	\mathfrak X^{\ast}(\epsilon)/p\cong\mathfrak X'^{\ast}(\epsilon)/\varpi^{\flat}.
\end{equation}
The morphism $\mathcal X^{\ast}_{\Gamma_0(p^{n+1})}(\epsilon)\to \mathcal X^{\ast}_{\Gamma_0(p^n)}(\epsilon)$ gets identified via the Atkin--Lehner isomorphism \eqref{eq:Atkin--Lehner} with a map $\mathcal X^{\ast}(p^{-(n+1)}\epsilon)\to \mathcal X^{\ast}(p^{-n}\epsilon)$ that has a formal model $\phi\colon \mathfrak X^{\ast}(p^{-(n+1)}\epsilon)\to \mathfrak X^{\ast}(p^{-n}\epsilon)$. One can then prove that mod $p^{1-\delta}$ where $\delta:=\frac{p+1}{p}\epsilon$, the map $\phi$ gets identified with $F_{\rel}$ in the sense that the following diagram commutes:
\[\begin{tikzcd}[row sep=0.2cm]
	\mathfrak X^{\ast} (p^{-(n+1)}\epsilon)/p^{1-\delta} \arrow[r,"\phi"] \arrow[d,equal,"\sim"labelrotate] & 	\mathfrak X^{\ast} (p^{-n}\epsilon)/p^{1-\delta}  \arrow[d,equal,"\sim"labelrotate] \\
	\mathfrak X'^{\ast}(p^{-(n+1)}\epsilon)/\varpi^{\flat(1-\delta)} \arrow[r,"F_{\rel}"]                                    &  	\mathfrak X'^{\ast}(p^{-n}\epsilon)/\varpi^{\flat(1-\delta)}.              
\end{tikzcd}
\]
In the inverse limit, this gives the result by \cite[Thm~5.2]{perfectoid-spaces}.

The following lemma says that the isomorphism of Thm.~\ref{theorem torsion paper: tilting the perfectoid modular curve} identifies the cusps:

\begin{Lemma}\label{l:tilting cusps}
	 The cusps of $\mathcal X^{\ast}$ and $\mathcal X'^{\ast}$ correspond via tilting: For each cusp $x$ of $\mathcal X^{\ast}$, considered as a finite \'etale adic space over $K$, its tilt can be canonically identified with a cusp $x^{\flat}$ of  $\mathcal X'^{\ast}$. In particular, $(L_{x})^{\flat} =L_{x^{\flat}}$ for the fields of definition.
\end{Lemma}
\begin{proof}
	The cusp $x$ can be described as a closed immersion $\Spa(L_x)\hookrightarrow \mathcal X'^{\ast}$. It has a canonical formal model $\Spf(\O_{L_x})\to \mathfrak X'^{\ast}$ that reduces mod $p$ to a morphism $\Spec(\O_{L_x}/p)\to X^{\ast}_{\O_K/p}$
	which in turn can be interpreted as cusp of $X^{\ast}_{\O_{K^{\flat}/\varpi^{\flat}}}$. Lifting to $\mathfrak X'^{\ast}$ and taking generic fibre gives a cusp $x^{\flat}$ defined by a closed point
	\[\Spa((L_x)^{\flat})\hookrightarrow \X'^{\ast}.\]
	It is clear from this construction that this can be identified with the tilt of $x$ via the equivalence of \'etale sites.
	Reversing this argument shows that this defines a bijection on cusps.
\end{proof}

\begin{Proposition}\label{p: description of the tilting isomorphism at level Gamma0 on Tate parameter spaces}
	Let $x$ be any cusp of $\mathcal X^{\ast}$. Then the canonical isomorphism of $K^{\flat}$-algebras
		\[L_x\llbracket q^{1/p^\infty}\rrbracket^{\flat} =L_{x^{\flat}}\llbracket q^{1/p^\infty}\rrbracket\]
		defines isomorphisms $\overline{\D}_{\infty,x}^\flat\cong\overline{\D}'_{\infty,x^{\flat}} $ and  $\D_{\infty,x}^\flat\cong\D'_{\infty,x^{\flat}}$ that fit into a commutative diagram
		\[
		\begin{tikzcd}[row sep = 0.15cm]
		\D_{\infty,x}^{\flat} \arrow[d,equal,"\sim"labelrotate] \arrow[r, hook] & \overline{\D}_{\infty,x}^{\flat} \arrow[d,equal,"\sim"labelrotate] \arrow[r, hook] & \XaGea{0}{\infty}^{\flat} \arrow[d,equal,"\sim"labelrotate] \\
		\D'_{\infty,x^{\flat}} \arrow[r, hook] & \overline{\D}'_{\infty,x^{\flat}}\arrow[r, hook] & \Xpaep.
		\end{tikzcd}\]
\end{Proposition}
\begin{proof}
	For the proof we use that $\overline{\D}_\infty$ has a very simple $p$-adic formal model. Here and in the following, let us for simplicity drop the additional $x$ and $x^{\flat}$ in the index.
	
	We can without loss of generality assume $\epsilon=0$. Using the identifications
	\[\D_\infty^{\flat}=\cup_n \overline{\D}_\infty(|q|^n\leq |\varpi|)^{\flat}= \cup_n \overline{\D}'_\infty(|q|^n\leq |\varpi^{\flat}|)=\D'_\infty,\] 
	it is clear that the left square commutes. It therefore suffices to consider the right square.
	
	Recall that the morphism $\overline{\D}\to \Xa(0)$ arises as the adic generic fibre of the morphism $\overline{\mathfrak D}\to \mX^{\ast}(0)$ where $\overline{\mathfrak D}:=\Spf(\O_L\llbracket q\rrbracket)$ is endowed with the $p$-adic topology. Similarly, $\overline{\D}'\to \Xpa(0)$ is the adic generic fibre of $\overline{\mathfrak D}'\to \mXpa(0)$ where $\overline{\mathfrak D}':=\Spf(\O_{L^{\flat}}\llbracket q\rrbracket)$. The reductions mod $\varpi$ and $\varpi^{\flat}$ of these formal models can be canonically identified with the map \[\Spec((\O_L/p)\llbracket q\rrbracket )\to \mX^{\ast}(0)/p=\mXpa(0)/\varpi^{\flat}\]
	associated to the Tate curve for the corresponding cusp of $X^{\ast}_{\O_K/p}$.
	
	In the limit over $\varphi$ and $F_{\rel}$, these identifications therefore fit into a commutative diagram
	\[\begin{tikzcd}[row sep =0.15cm]
	\varprojlim_{q\mapsto q^p} \overline{\mathfrak D}/p\arrow[r]\arrow[d,equal,"\sim"labelrotate]&\varprojlim_{\phi}\mXa(0)/p\arrow[d,equal,"\sim"labelrotate]\\
	\varprojlim_{q\mapsto q^p} \overline{\mathfrak D}'/\varpi^{\flat}\arrow[r]& \mXpa(0)^{\perf}/\varpi^{\flat}.
	\end{tikzcd}
	\]
	As these are perfectoid schemes over $\O_K^{a}/p$ (or $\O_{K^{\flat}}^{a}/\varpi^{\flat}$) with corresponding perfectoid spaces $\overline{\D}_\infty\to \mathcal X^{\ast}_{\Gamma_0(p^\infty)}(0)_a$ and $\overline{\D}'_\infty\to \mathcal X'^{\ast}(0)^{\perf}$ via \cite[Thm~5.2]{perfectoid-spaces}, this gives the desired identification of the tilts.
\end{proof}
\begin{Remark}
	The correspondence of moduli of Tate curves implicit in Prop.~\ref{p: description of the tilting isomorphism at level Gamma0 on Tate parameter spaces} can be made explicit as follows: Let $(R,R^+)$ be a perfectoid $K$-algebra, then Thm~\ref{theorem torsion paper: tilting the perfectoid modular curve} gives a correspondence \[\X_{\Gamma_0(p^\infty)}(\epsilon)_a(R,R^{+})=\Xpep(R^{\flat},R^{\flat+})\]
	of elliptic curves with extra data. If now $E_q$ is a Tate curve with parameter $q\in R$, equipped with $\Gamma_0(p^\infty)$-structure $(q^{1/p^n})_{n\in\N}$ and $\Gamma^p$-structure, corresponding to a point in $\D_\infty(R,R^+)$, then via $\D_\infty(R,R^+)=\D'_\infty(R^{\flat},R^{\flat+})$, this corresponds to the Tate curve $E_{q'}$ with parameter 
	\[q':=(q^{1/p^n})_{n\in\N}\in \varprojlim_{q\mapsto q^p}R^{\times}=R^{\flat\times}.\]
	
	One can moreover identify the $\Gamma^p$-structure of $E_{q'}$ using that $E_{q}[N]^{\flat}=E_{q'}[N]$.
\end{Remark}
\subsection{The tilting isomorphism at level $\Gamma_1(p^\infty)$}
We now extend the tilting isomorphism of Thm.~\ref{theorem torsion paper: tilting the perfectoid modular curve} to level $\Gamma_1(p^\infty)$ by proving the following theorem stated in the introduction:
\begin{Theorem}\label{t: tilting the Gamma_1(p^infty)-tower-2}
	\begin{enumerate}
		\item There is a canonical isomorphism 
		\[\mathcal X^{\ast}_{\Gamma_1(p^\infty)}(\epsilon)_a^{\flat} \isomarrow \mathcal X'^{\ast}_{\Ig(p^\infty)}(\epsilon)^{\perf}\]
		which is $\Z_p^\times$-equivariant and makes the following diagram commute:
		\begin{equation*}
		\begin{tikzcd}[row sep = 0.15cm]
		\mathcal X^{\ast}_{\Gamma_1(p^\infty)}(\epsilon)_a^{\flat} \arrow[d,"\sim"labelrotate,equal] \arrow[r] & \mathcal X^{\ast}_{\Gamma_0(p^\infty)}(\epsilon)_a^{\flat}\arrow[d,"\sim"labelrotate,equal] \\ \mathcal X'^{\ast}_{\Ig(p^\infty)}(\epsilon)^{\perf}  \arrow[r] & \mathcal X'^{\ast}(\epsilon)^{\perf}.
		\end{tikzcd}
		\end{equation*}
		\item The cusps of $\XaGea{1}{\infty}$ and $\XpaIie$ correspond via the isomorphism in (1). Moreover, for any cusp $x$ of $\X^{\ast}$, the following diagram commutes:
		\begin{equation*}
		\begin{tikzcd}[row sep = 0.15cm]
		\underline{\Z_p^{\times}}\times \D_{\infty,x}^{\flat} \arrow[d,"\sim"labelrotate,equal] \arrow[r,hook] & 
		\XaGoiea^{\flat} \arrow[d,"\sim"labelrotate,equal]\\
		\underline{\Z_p^{\times}}\times \D'_{\infty,x^{\flat}}\arrow[r,hook]& \mathcal X'^{\ast}_{\Ig(p^\infty)}(\epsilon)^{\perf},
		\end{tikzcd}
		\end{equation*}
		where the left map is given by the canonical identification $\D_{\infty,x}^{\flat}\cong  \D'_{\infty,x^{\flat}}$. 
	\end{enumerate}
\end{Theorem}
For the proof, we use the univeral anticanonical subgroup at infinite level:
\begin{Definition}
	For any $n\in \Z_{\geq 1}$,
	we denote by $G_n\to \X_{\Gamma_0(p^\infty)}(\epsilon)_a$ the universal anticanonical subgroup of rank $n$. This can be defined via pullback from finite level $\X_{\Gamma_0(p^n)}(\epsilon)_a$, and is a finite \'etale morphism of perfectoid spaces.
	
	Let $\E'\to \X'$  be the analytification of the universal elliptic curve over $X'$, and write $\E'(\epsilon)\to \X'(\epsilon)$ for the pullback.  We denote by $G'_n\to \Xpep $ the finite \'etale morphism of perfectoid spaces given by the perfection of $\ker V^n\subseteq \E'(\epsilon)^{(p^n)}$.
\end{Definition}

\begin{Lemma}\label{Lemma G_n identifies with ker V^n upon tilting}
	There is a natural isomorphism making the following diagram commutative:
	\begin{equation*}
	\begin{tikzcd}[column sep=0.15cm]
	G_n^{\flat} \arrow[r,"\sim"] \arrow[d] & G'_n \arrow[d]\\
	\X_{\Gamma_0(p^\infty)}(\epsilon)_a^{\flat}\arrow[r,"\sim",equal]& \X'(\epsilon)^{\perf},
	\end{tikzcd}
	\end{equation*}
\end{Lemma}
This lemma is a slight extension of \cite[Lemma III.2.26]{torsion},  from the good reduction locus to the whole uncompactified modular curve (we reiterate that \cite{torsion} writes $\X$ for the good reduction locus, whereas we use this symbol to denote the whole open modular curve).
\begin{proof}
	It suffices to see this locally on $\X_{\Gamma_0(p^\infty)}(\epsilon)_a$. The case of good reduction is \cite[Lemma~III.2.26]{torsion}. It therefore suffices to prove the lemma over the ordinary locus $\X_{\Gamma_0(p^\infty)}(0)_a$.
	
	Over $\mathfrak X^{\ast}(0)$, the universal semi-abelian scheme has a canonical subgroup, a finite flat group scheme $C_n\to \mathfrak X^{\ast}(0)$. Via the Atkin--Lehner isomorphism $\mathcal X^{\ast}(0)\isomarrow\X^{\ast}_{\Gamma_0(p^n)}(0)_a$, the generic fibre of its dual $(C_n^\vee)^{\ad}_{\eta}$ can be identified over $\X_{\Gamma_0(p^n)}(0)_a$ with the universal anticanonical subgroup over $\X_{\Gamma_0(p^n)}(0)_a$.
	
	Similarly, over  $\mathfrak X'^{\ast}(0)$, we have a canonical subgroup $C'_n\to\mathfrak X^{\ast}(0)$, and the dual $(C_n'^\vee)^{\ad}_{\eta}$ restricted  to $\mathcal X'(0)$ can be identified with the kernel of Verschiebung of $\E'(0)$ over $\Xp(0)$.
		
	It follows from these descriptions that after pullback we have identifications 
	\begin{alignat*}{2}
	G_n&=&&\big(C_n^{\vee}\times _{\mX^{\ast}(0)} \varprojlim_{\phi}\mX^{\ast}(0)\big)^{\ad}_{\eta} \text{ restricted to }\X_{\Gamma_0(p^\infty)}(0)_a,\\
	G_n'&=&&\big(C'^{\vee}_n\times _{\mXp^{\ast}(0)} \mXp^{\ast}(0)^{\perf}\big)^{\ad}_\eta \text{ restricted to }\mathcal X'(0)^{\perf}.
	\end{alignat*}
	To prove the lemma, it therefore suffices to prove that the formal models on the right hand side can be identified after reduction to $\O_K/p=\O_{K^{\flat}}/\varpi^{\flat}$, for which it suffices to prove that $C_n^{\vee}/p=C'^{\vee}_n/\varpi^{\flat}$ on $\mathfrak X^{\ast}(0)/p=\mathfrak X'^{\ast}(0)/\varpi^{\flat}$. But over the ordinary locus, $C_n/p=C'_n/\varpi^{\flat}$ are both the kernel of Frobenius, and Cartier duals commute with base change.
\end{proof}
We can now complete the proof of Thm.~\ref{t: tilting the Gamma_1(p^infty)-tower-2} stated in the introduction:
\begin{proof}[Proof of Thm.~\ref{t: tilting the Gamma_1(p^infty)-tower-2}]
We start by proving that for any $n\in\Z_{\geq 1}$, there is a natural isomorphism 
\begin{equation}\label{dg:proof-of-thm-1.11-case-of-n<infty}
	\begin{tikzcd}
	\mathcal X^{\ast}_{\Gamma_1(p^n)\cap \Gamma_0(p^\infty)}(\epsilon)_a^{\flat} \arrow[r,"\sim"] \arrow[d] & \mathcal X'^{\ast}_{\Ig(p^n)}(\epsilon)^{\perf} \arrow[d] \\ \mathcal X^{\ast}_{\Gamma_0(p^\infty)}(\epsilon)_a^{\flat} \arrow[r,equal] & \mathcal X'^{\ast}(\epsilon)^{\perf}
	\end{tikzcd}
\end{equation}
making the diagram commute. Part (1) of the theorem then follows in the limit $n\to \infty$. 

Away from the cusps, the desired isomorphism is induced by
the natural isomorphism from Lemma~\ref{Lemma G_n identifies with ker V^n upon tilting}, using the moduli interpretations in Lemma~\ref{l:moduli interpretation of modular curve of level Gamma(p^n) cap Gamma_0(p^infty)} and Lemma~\ref{Lemma moduli interpretation of Ig(p^n)^perf}.

We need to extend this over the cusps. One way of doing this is to give the vertical maps in the diagram a relative moduli interpretation that extends to the cusps. More in the spirit of our arguments so far, we shall instead give a more explicit proof using Tate parameter spaces, which also proves part (2).

To this end, fix a cusp $x$ of $\Xa$. By Prop.~\ref{p: description of the tilting isomorphism at level Gamma0 on Tate parameter spaces}, the isomorphism $\mathcal X^{\ast}_{\Gamma_0(p^\infty)}(\epsilon)_a^{\flat} \to \mathcal X'^{\ast}(\epsilon)^{\perf}$ restricts to the canonical isomorphism $\D_{\infty}^{\flat}=\D'_{\infty}$ over $x$.
Using the description of the Tate curve parameter spaces in $\mathcal X^{\ast}_{\Gamma_1(p^m)\cap \Gamma_0(p^\infty)}(\epsilon)_a\to \XaGea{0}{\infty}$  from Lemma~\ref{l: Tate parameter spaces of X^ast_Gamma_1(p^m)cap Gamma_0(p^infty)} and similarly in $\XpaInep\to \Xpaep$ from Prop.~\ref{p:Tate parameter spaces in perfectoid Igusa tower}.(1), it now suffices to prove that the isomorphism $G_n^\flat=G_n'$ over the Tate parameter spaces becomes the natural map
		\begin{equation}\label{dg:proof-of-thm-1.11-situation-over-cusps}
			\begin{tikzcd}
				(\underline{\Z/p^n\Z}\times \mathring{\D}_\infty)^{\flat} \arrow[r,"\sim"] \arrow[d] & 	\underline{\Z/p^n\Z}\times \mathring \D'_\infty \arrow[d] \\  \mathring \D_\infty^{\flat} \arrow[r,"\sim"] & \mathring \D'_\infty.
			\end{tikzcd}
		\end{equation}
It is then clear that the diagram extends uniquely over the cusps.
		
To see this, we note that the restriction of $G_n$ to $\mathring{\D}_{\infty}$ is indeed canonically isomorphic to $\underline{\Z/p^n\Z}\times \mathring{\D}_{\infty}$ due to the canonical section given by the element $q^{1/p^n}$ of the anticanonical subgroup $\langle q^{1/p^n}\rangle\subseteq \Tate(q)[p^n]$. Similarly,  $G'_n$ is isomorphic to $\underline{\Z/p^n\Z}\times \mathring{\D}'_\infty$ on $\mathring{\D}'_\infty\to \Xpep$. By considering the dual trivialisations of the respective canonical subgroups, it follows from the construction in the proof of Lemma~\ref{Lemma G_n identifies with ker V^n upon tilting} that these isomorphisms are compatible with tilting and make diagram~\eqref{dg:proof-of-thm-1.11-situation-over-cusps} commute, as desired.

Part (2) now follows from diagram~\eqref{dg:proof-of-thm-1.11-situation-over-cusps} in the limit $n\to \infty$.
\end{proof}

\section{$q$-expansion principles}\label{s:q-expansion-principles}
In this section, we prove various $q$-expansion principles for functions on the infinite level spaces $\XGea{0}{\infty}$, $\XGea{1}{\infty}$, $\XGea{}{\infty}$, etc, based on our discussion of cusps in \S\ref{s:adic-cusps-finite}-\S\ref{s:cusps-in-char-p}.

\subsection{Detecting vanishing}
We begin with the proof of $q$-expansion principle I, Prop.~\ref{p: q-expansion principle I} in the introduction, recalled below. On the way, we also prove principles III and IV. We focus on the case of characteristic $0$, the case of characteristic $p$ is completely analogous.

\begin{Proposition}\label{p: q-expansion principle I,second version}
	Let $\mathcal C$ be a collection of cusps of $\mathcal X^{\ast}$ such that each connected component of $\mathcal X^{\ast}$ contains at least one $x\in \mathcal C$. Let $n\in \Z_{\geq 0}\cup\{\infty\}$ and let $\Gamma$ be one of $\Gamma_0(p^n),\Gamma_1(p^n), \Gamma(p^n)$. Define $\D_{\mathcal C,\Gamma}$ as the pullback
	\[
	\begin{tikzcd}
	\D_{\mathcal C,\Gamma}\arrow[r,hook]\arrow[d]&\X^{\ast}_{\Gamma}(\epsilon)_a\arrow[d]\\
	\bigsqcup_{x\in \mathcal C}\D_{x}\arrow[r,hook]& \X^{\ast}(\epsilon).
	\end{tikzcd}
	\]
	Then the map $\mathcal O(\mathcal X^{\ast}_{\Gamma}(\epsilon)_a)\rightarrow \mathcal O(\D_{\mathcal C,\Gamma})$
	is injective. 
\end{Proposition}

This is an analogue of saying that for any affine irreducible integral variety over $K$, completion at any $K$-point gives rise to an injection on function, which is a consequence of Krull's Intersection Theorem. As this requires Noetherianess, we first reduce to the Noetherian situation using that all of the above spaces have natural models over $\Z_p$.

The proof is in two steps: We first consider $\mathcal X^{\ast}_{\Gamma_0(p^\infty)}(0)_a$ where it is easy to reduce to the Noetherian case. In a second step, we then show that restriction of functions from $\mathcal X^{\ast}_{\Gamma_1(p^\infty)}(\epsilon)_a$ to $\mathcal X^{\ast}_{\Gamma_1(p^\infty)}(0)_a$ is injective, which is a straight-forward computation on power series.
We start with the case $\epsilon=0$. On the way we will also see Prop.~\ref{p: q-expansion principle III}.

\begin{proof}[Proof of Prop.~\ref{p: q-expansion principle I} for $\epsilon=0$]
	The case of $\Gamma=\Gamma(p^n)$ reduces to the one of  $\Gamma_1(p^n)$  by Cor.~\ref{c: over ordinary locus, Gamma to Gamma_1 is split}.
	
	We first consider the case of $n<\infty$. Then the case of $\Gamma=\Gamma_0(p^n)$ further reduces to the case of tame level via the Atkin--Lehner isomorphism $\X^{\ast}(0)\cong\X^{\ast}_{\Gamma_0(p^n)}(0)_a$. We are therefore left with the case of $\Gamma_1(p^n)$ for $n\in \Z_{\geq 0 }$ (the case of tame level being $n=0$).
	
	The space $\mathcal X^{\ast}_{\Gamma_1(p^n)}(0)_a$ has an affine formal model $\mathfrak X^{\ast}_{\Gamma_1(p^n)}(0)_a=\Spf(R)$ for some complete $\Z_p$-algebra $R$. Let $\mathcal C'$ be the pullback of $\mathcal C$ to $\mathfrak X^{\ast}_{\Gamma_1(p^n)}(0)_a$ and let 
	\[\sqcup_{x\in \mathcal C'} \Spf \mathcal O_{L_x}\llbracket q\rrbracket\rightarrow \mathfrak X^{\ast}_{\Gamma_1(p^n)}(0)_a\]
	be the completion along $\mathcal C'$. 
	It suffices to show that the map on global sections 
	\[\varphi:R\to \prod_{x\in \mathcal C'} \mathcal O_{L_x}\llbracket q\rrbracket\] is injective.
	As these are flat $\O_K$-algebras, it suffices to see that the reduction 
	\begin{equation}\label{eq:q-exp-III-for-n<infty}
	R/p\to \prod_{c \in \mathcal C'} \mathcal O_{L_x}/p\llbracket q\rrbracket
	\end{equation}
	 is injective.
	But this reduction can be interpreted as the completion of $X^{\ast}_{\mathcal O_K/p,\Ig(p^n),\ord}$ at the divisor of cusps $\mathcal C'$. 
	By base change from $\F_p$ to $\O_K/p$, we can now reduce to showing that for $Y:=X^{\ast}_{\F_p,\Ig(p^n),\ord}$, completion at $\mathcal C'$ defines an injection 
	\[\O(Y)\to \prod_{x\in \mathcal C'}\F_p(x)\llbracket q\rrbracket \]
	where $\F_p(x)\subseteq \F_p[\zeta_N]$ is the coefficient field of definition of the level structure on the Tate curve corresponding to the cusp $x\in \mathcal C'$.
	
	Since $Y$ is a smooth affine curve over $\F_p$,  and by considering each connected component separately, the desired injectivity follows  as for an integral Noetherian ring $A$, completion at any maximal ideal $\mathfrak m\subseteq A$ gives an injection $A\to \hat{A}_{\mathfrak m}$ by Krull's intersection theorem.
	
	The case of $n=\infty$ can be deduced in the limit:
	As the natural restriction map \[\O^+(\overline{\D}_\infty)\hookrightarrow\O^+({\D}_\infty)\]
	is injective (see Def.~\ref{d: OK bb q^1/p^infty _p}), it suffices to prove the statement for $\overline{\D}_\infty\sim \varprojlim\overline{\D}_n$,
	 while conversely it is clear from $\O^+(\overline{\D}_n)=\O^+({\D}_n)$ for $n<\infty$ and the first part that the corresponding result at finite level holds for $\D_n$ replaced by $\overline{\mathcal D}_n$.
	
	For any $m\in\N$, let $\mathfrak Y_m=\mathfrak X^{\ast}_{\Gamma_0(p^m)}(0)_a$ or $\mathfrak Y_m=\mathfrak X^{\ast}_{\Gamma_1(p^m)}(0)_a$. Then $\mathfrak Y=\varprojlim \mathfrak Y_m$ is a formal model of $\mathcal X^{\ast}_{\Gamma}(0)_a$. To see the result it suffices to prove that the natural maps
		\begin{alignat*}{4}
	\O(\mathfrak X^{\ast}_{\Gamma_0(p^m)}(0)_a)\to& \prod_{x\in \mathcal C}\O(\mathfrak D_{\infty,x})\\
		\O(\mathfrak X^{\ast}_{\Gamma_1(p^m)}(0)_a)\to& \prod_{x\in \mathcal C}\Map_{\cts}(\Z_p^\times,\O(\mathfrak D_{\infty,x})).
		\end{alignat*}
	are injective. By flatness, it suffices to prove this on the reduction mod $\varpi$. But here it follows in the direct limit over $m\to \infty$ from the case of finite level.
\end{proof}

The proof of Prop.~\ref{p: q-expansion principle I} is completed by the following two lemmas:
\begin{Lemma}\label{l:restricting functions from (epsilon) to (0) is injective}
	\leavevmode
	Let $n\in\Z_{\geq 0}\cup \{\infty\}$ and 
	let $\mathcal Y\rightarrow \mathcal X^{\ast}$ be one of $\mathcal X^{\ast}_{\Gamma_0(p^n)}$, $\mathcal X^{\ast}_{\Gamma_1(p^n)}$, $\mathcal X^{\ast}_{\Gamma(p^n)}$. Then the open immersion $\mathcal Y(0)\rightarrow \mathcal Y(\epsilon)$ on sections gives an injection $\mathcal O(\mathcal Y(\epsilon))\rightarrow \mathcal O(\mathcal Y(0))$.
\end{Lemma}
\begin{proof}[Proof of Lemma~\ref{l:restricting functions from (epsilon) to (0) is injective}]
	It suffices to prove this locally. Let $\mathcal Y_{\gd}(\epsilon):=\mathcal Y(\epsilon)\times_{\X}\X_{\gd}$. Then since $\mathcal Y(0)\to\mathcal Y(\epsilon)$ is an open immersion, and $\mathcal Y(0)$ and $\mathcal Y_{\gd}(\epsilon)$ cover all of $\mathcal Y(\epsilon)$, it suffices to prove the statement for $\mathcal Y_{\gd}(\epsilon)$.
	Let thus $\mathfrak Y$ be one of $\mathfrak X_{\Gamma_0(p^n)}$ , $\mathfrak X_{\Gamma_1(p^n)}$, $\mathfrak X_{\Gamma(p^n)}$, each for any $n\in\Z_{\geq 0}\cup \{\infty\}$. It suffices to prove that for any affine open $\mathfrak U=\Spf(R)\subseteq \mathfrak Y$ where $\omega$ is trivial, the natural map $\mathfrak Y(0)\rightarrow \mathfrak Y(\epsilon)$ induces an injection $
	\mathcal O(\mathfrak Y(\epsilon)|_{\mathfrak U})\rightarrow \mathcal O(\mathfrak Y(0)|_{\mathfrak U})$. We have
	$\mathfrak Y(\epsilon)|_{\mathfrak U}=\Spf(S)$ where $S=R\langle X\rangle /(X\Ha -p^{\epsilon} )$, and $\mathfrak Y(0)|_{\mathfrak U}=\Spf(R\langle \Ha^{-1}\rangle)$. Since $\Ha$ is a non-zero-divisor on $R/p^n$, Lemma~\ref{l:restricting functions from (epsilon) to (0) on algebras} below now gives the desired statement.
\end{proof}	
\begin{Lemma}\label{l:restricting functions from (epsilon) to (0) on algebras}
	Let $A$ be any ring, let $0\neq \varpi \in A$ be a non-zero-divisor and let $H\in A$ be such that its image in $A/\varpi$ is a non-zero-divisor. Endow $A$ with the $\varpi$-adic topology. Then
	\[\varphi\colon A\langle X\rangle /(XH-\varpi)\xrightarrow{X\mapsto \varpi X } A\langle X \rangle/(XH-1) \]
	is injective.
\end{Lemma}
\begin{proof}
	We first note that the assumption on $H\in A$ implies that $H$ is a non-zero-divisor in any $A/\varpi^n$.
	Suppose $f=\sum a_nX^n$ is in the kernel of $A\langle X\rangle\rightarrow A\langle X\rangle/(XH-\varpi)\xrightarrow{\varphi} A\langle X\rangle /(XH-1)$. Then there is $g=\sum b_nX^n\in A\langle X\rangle$ such that 
	\[f(\varpi X)=\sum a_n\varpi^nX^n=(XH-1)g=(XH-1)\sum b_nX^n.\]
	Reducing mod $\varpi^m$, we see that
	\[a_0+\dots +a_{m-1}\varpi^{m-1} X^{m-1} \equiv (XH-1)\sum b_nX^n \bmod \varpi^m\]
	By comparing degrees as polynomials in $A/\varpi^m[X]$, we conclude from  $H$ being a non-zero-divisor mod $\varpi^m$ that $\deg(\sum b_nX^n \bmod \varpi^m)<m-1$, thus $b_k\equiv 0 \bmod \varpi^m$ for $k\geq m-1$.
	
	Consequently, there are elements $c_m={b_m}/{\varpi^{m+1}}\in A$ for all $m$ and in $A\llbracket X\rrbracket$  we have
	\[f':= (XH-\varpi)\sum \frac{b_m}{\varpi^{m+1}}X^m \stackrel{X\mapsto \varpi X}{\longmapsto} (XH-1)\sum b_mX^m.\]
	Thus $f'(\varpi X)=f(\varpi X)$ in $A\llbracket X\rrbracket$ which implies $f'=f$ since $\varpi$ is a non-zero-divisor.
	
	It remains to prove that $\sum c_mX^m$ converges in $A\langle X\rangle$: Since $f\in A\langle X\rangle$, for every $k\in \N$ there is an $N_k$ such that $v(a_{m})\geq k$ for all $m\geq N_k$, where $v$ is the $\varpi$-adic valuation. In particular, we then have $v(\varpi^ma_{m})\geq k+m$ for all $m\geq N_k$. Consequently, for all $m\geq N_k$
	\[a_0+\dots +a_{m-1}\varpi^{m-1} X^{m-1} \equiv (XH-1)\sum b_mX^m \bmod \varpi^{m+k}.\]
	This shows that $v(b_{m-1})\geq m+k$, and thus $v(c_m)\geq k$ for all $m\geq N_k$. Thus $\sum c_mX^m \in A\langle X\rangle$ as desired.
	We conclude that $f$ is already in $(XH-\varpi)A\langle X\rangle$. Thus $\varphi$ is injective.
\end{proof}

We can extract from this argument a proof of $q$-expansion principle III in the introduction:

\begin{Proposition}[$q$-expansion principle III]
	For $f\in \O(\mathcal X^{\ast}_{\Gamma_0(p^\infty)}(0)_a)$ are equivalent:
	\begin{enumerate}
		\item	$f$ is integral, i.e.\ it is contained in $\O^+(\mathcal X^{\ast}_{\Gamma_0(p^\infty)}(0)_a)$. 	
		\item The $q$-expansion of $f$ at every cusp $x$ is already in $\O_{L_x}\llbracket q^{1/p^\infty}\rrbracket$.
		\item On each connected component of $\XaGea{0}{n}$, there is at least one cusp $x$ at which the $q$-expansion of $f$ is in $\O_{L_x}\llbracket q^{1/p^\infty}\rrbracket$.
	\end{enumerate}
	Equivalently, the natural map
	$\varphi\colon \O^+(\mathcal X^{\ast}_{\Gamma_0(p^\infty)}(0)_a)/p\to \prod_{x}(\O_{L_x}/p)\llbracket q^{1/p^\infty}\rrbracket$
	is injective. 
	The analogous statements for $\mathcal X^{\ast}_{\Gamma_1(p^\infty)}(0)_a$, $\mathcal X^{\ast}_{\Gamma(p^\infty)}(0)_a$, $\mathcal X'^{\ast}(0)^{\perf}$ and $\mathcal X'^{\ast}_{\Ig(p^\infty)}(0)^\perf$ are also true when we replace $\O_{L_x}\llbracket q^{1/p^\infty}\rrbracket$ by the respective algebra from Prop.~\ref{p: q-expansion principle I}.
\end{Proposition}
\begin{proof}
	It is clear from $\O^+(\D_{\infty,x})= \O_{L_x}\llbracket q^{1/p^\infty}\rrbracket$ that (1) implies (2) implies (3). To prove that (3) implies (1), it suffices to see that $\O^+(\mathcal X^{\ast}_{\Gamma_0(p^\infty)}(0)_a)/p\to \prod_{x}(\O_{L_x}/p)\llbracket q^{1/p^\infty}\rrbracket$ is injective.
	We have already seen in \eqref{eq:q-exp-III-for-n<infty} in the proof of Prop.~\ref{p: q-expansion principle I,second version} that
	\[ \O(\mathfrak X^{\ast}_{\Gamma_0(p^n)}(0)_a)/p\hookrightarrow \prod_{x\in \mathcal C}(\O_{L_x}/p)\llbracket q^{1/p^n}\rrbracket \]
	is injective for any $n\in \Z_{\geq 0}\cup\{\infty\}$. Since by 
	by \cite[Prop.~4.1.3]{heuer-thesis}, we have \[\O^+(\mathcal X^{\ast}_{\Gamma_0(p^n)}(0)_a)=\O(\mathfrak X^{\ast}_{\Gamma_0(p^n)}(0)_a),\]
	this gives the desired statement in the case of $\Gamma_0(p^\infty)$.
	
	By the same argument, the cases of $\mathcal X^{\ast}_{\Gamma_1(p^\infty)}(0)_a$, $\mathcal X^{\ast}_{\Gamma(p^\infty)}(0)_a$, $\mathcal X'^{\ast}(0)^{\perf}$ and $\mathcal X'^{\ast}_{\Ig(p^\infty)}(0)^\perf$ also follow from \eqref{eq:q-exp-III-for-n<infty} in the limit $n\to \infty$ using instead \cite[Lemma A.2.2.3]{heuer-thesis}.
\end{proof}

We can also use the lemmas for the proof of $q$-expansion principle IV:
\begin{Proposition}[$q$-expansion principle IV]
	Let $\mathcal C$ be a collection of cusps of $\Xa$ such that each connected component contains at least one $x\in \mathcal C$. Then a function on the good reduction locus $\X_{\gd}(\epsilon)$ extends to all of $\X^{\ast}(\epsilon)$ if and only if its $q$-expansion with respect to $\overline{\D}(|q|\geq 1)\to \X_{\gd}(\epsilon)$ at each $x\in \mathcal C$ is already in $\O_{L_x}\llbracket q\rrbracket[\tfrac{1}{p}]\subseteq \O_{L_x}\lauc{q}[\tfrac{1}{p}]$. In this case, the extension is unique. The analogous statements for $\mathcal X^{\ast}_{\Gamma_0(p^\infty)}(0)_a$ $\mathcal X^{\ast}_{\Gamma_1(p^\infty)}(0)_a$, $\mathcal X^{\ast}_{\Gamma(p^\infty)}(0)_a$, $\mathcal X'^{\ast}(0)$, $\mathcal X'^{\ast}(0)^{\perf}$ and $\mathcal X'^{\ast}_{\Ig(p^\infty)}(0)^\perf$ are also true.
\end{Proposition}
\begin{proof}
	As before, one can reduce to the case of finite level. For simplicity, let us treat $\mathcal X_{\gd}$, the other cases are similar. By Lemma~\ref{l:restricting functions from (epsilon) to (0) is injective} we can reduce to $\epsilon=0$. We then need to prove that the following sequence is left exact:
	\[ 0\to \O(\mathfrak X^{\ast}(0)) \to \O(\mathfrak X(0))\times \textstyle\prod_{x\in \mathcal C} \O_{L_x} \bb{q}\xrightarrow{(f,g)\mapsto f-g} \textstyle\prod_{x\in \mathcal C}\O_{L_x}\lauc{q}.\]
	It suffices to prove that this is true mod $\varpi^n$ for all $n$. By tensoring with the flat $\F_p$-algebra $\O_L/\varpi^n$, the statement then follows from the following sequence being left-exact:
	\[ 0\to \O(X_{\F_p,\ord}^{\ast}) \to \O(X_{\F_p,\ord})\times \textstyle\prod_{x\in \mathcal C} \F_p(x)\bb{q}\xrightarrow{(f,g)\mapsto f-g} \textstyle\prod_{x\in \mathcal C}\F_p(x)(\!(q)\!).\]
	This holds as $X_{\F_q}^{\ast}$ is the normalisation of $j\colon X_{\F_q}\to \A_{\F_q}^1$ in $\P_{\F_q}^1$, and thus a function $f$ extends to the cusp $x$ if and only if it is finite over the completion $\F_p\llbracket q\rrbracket $ of $\P_{\F_p}^1$ at $\infty$.
\end{proof}
\subsection{Tate traces and detecting the level}
While the transition from $\Gamma_0(p^\infty)$ to $\Gamma(p^\infty)$ is controlled by the Galois action, the transition from $\Gamma_0(p)$ to $\Gamma_0(p^\infty)$ is controlled by normalised Tate traces, as discussed in \cite[\S III.2.4]{torsion}:

\begin{Proposition}[{\cite[Cor.~III.2.23]{torsion}}]
	Let $0\leq n\leq k\in\N$. Then the normalised traces
	\[\tr_{k,n}\colon  \O_{\mathfrak X^{\ast}_{\Gamma_0(p^k)}(\epsilon)_a}\to \O_{\mathfrak X^{\ast}_{\Gamma_0(p^n)}(\epsilon)_a}[\tfrac{1}{p}]\]
	of the finite flat forgetful map $\mathfrak X^{\ast}_{\Gamma_0(p^k)}(\epsilon)_a\to \mathfrak X^{\ast}_{\Gamma_0(p^n)}(\epsilon)_a$ in the limit $k\to \infty$ give rise  to compatible continuous $\O_{\mathfrak X^{\ast}_{\Gamma_0(p^n)}(\epsilon)_a}$-linear morphisms with bounded image
	\[\tr_{n}\colon \O_{\mathfrak X^{\ast}_{\Gamma_0(p^\infty)}(\epsilon)_a}\to \O_{\mathfrak X^{\ast}_{\Gamma_0(p^n)}(\epsilon)_a}[\tfrac{1}{p}].\]
\end{Proposition}
\begin{proof}
	Via the Atkin--Lehner isomorphism $\mathfrak X^{\ast}_{\Gamma_0(p^n)}(\epsilon)_a\cong \mathfrak X^{\ast}(p^{-n}\epsilon)$,
	this is the statement of \cite[Cor.~III.2.23]{torsion}, except that we use the compactified $\mathfrak X^{\ast}$ instead of $\mathfrak X$: This is possible since in contrast to the higher dimensional Siegel moduli spaces, the minimal compactification of the modular curve $\mathfrak X^{\ast}$ is a smooth formal scheme, and thus Cor.~III.2.22 applies over all of $\mathfrak X^{\ast}$, not just over $\mathfrak X$, which means that the proof of III.2.23 goes through for $\mathfrak X^{\ast}$.
\end{proof}

\begin{Definition}
	Taking global sections and inverting $p$, the trace $\tr_n$ gives a $K$-linear map
	\[\tr\colon \O(\mathcal X^{\ast}_{\Gamma_0(p^\infty)}(\epsilon)_a)\to \O(\mathcal X^{\ast}_{\Gamma_0(p^n)}(\epsilon)_a).\]
\end{Definition}

\begin{Proposition}\label{Proposition: effect of trace on q-expansions}
	Let $x$ be any cusp of $\mathcal X^{\ast}$, with corresponding Tate curve parameter space $\D_{n,x}\hookrightarrow \mathcal X^{\ast}_{\Gamma_0(p^n)}(\epsilon)_a$. Then the normalised Tate trace fits into a commutative diagram
	\begin{equation*}
		\begin{tikzcd}[row sep={0.7cm,between origins}]
			\O(\XaGea{0}{\infty}) \arrow[r,  "\tr_n"] \arrow[dd, '] & \O(\mathcal X^{\ast}_{\Gamma_0(p^n)}(\epsilon)_a) \arrow[dd, ']&&\\
			&&	{\displaystyle\sum_{m\in\Z[\frac{1}{p}]_{\geq 0}} a_mq^m} \arrow[r, maps to]& {\displaystyle\sum_{m\in \tfrac{1}{p^n}\Z_{\geq 0}} a_mq^m}
			\\
			\O(\D_{\infty,x})  \arrow[r, "\tr_n"] & \O(\D_{n,x}),&&
		\end{tikzcd}
	\end{equation*}
	where the bottom map is given by forgetting all coefficients $a_m$ for $m\not \in \tfrac{1}{p^n}\Z_{\geq 0}$.
\end{Proposition}
\begin{proof}
	Let us treat the case of $n=0$, the other cases are completely analogous.
	By continuity, $\tr_n$ is uniquely determined by the normalised traces $\tr_{k,0}$. By Lemma~\ref{Proposition: Tate parameter spaces in the Gamma_0-tower}, this is on $q$-expansions the trace of the inclusion $\O_L\llbracket q\rrbracket\to \O_L\llbracket q^{1/p^k}\rrbracket$. Since after inverting $q$, this map becomes Galois with automorphisms $q^{1/p^k}\mapsto q^{1/p^k}\zeta^d_{p^k}$ for $d\in \Z/p^k\Z$, we compute
	\[\tr_{k,0}\left(\sum_{i=0}^\infty a_{\frac{i}{p^k}}q^{\frac{i}{p^k}}\right)= \frac{1}{p^k}\sum_{i=0}^\infty a_{\frac{i}{p^k}}(1+\zeta_{p^k}^i+\dots+\zeta_{p^k}^{(p^{k}-1)i})q^{{\frac{i}{p^k}}}=\sum_{i=0}^\infty a_{i}q^i\]
	as $1+\zeta_{p^k}^i+\dots+\zeta_{p^k}^{(p^{k}-1)i}=0$ unless $p^k|i$, when it is $=p^k$, giving the desired description.
\end{proof}
\begin{Proposition}[$q$-expansion principle II]\label{p: q-expansion principle II}
	Let  $f\in \O(\XaGea{0}{\infty})$. Then for any $n\in \Z_{\geq 0}$, the following are equivalent:
	\begin{enumerate}
		\item $f$ comes via pullback from $\XaGea{0}{n}$, i.e.\  $f\in\O(\XaGea{0}{n})\subseteq \O(\XaGea{0}{\infty})$.
		\item The $q$-expansion of $f$ at every cusp $x$ is contained in $\O_{L_x}\llbracket q^{1/p^n}\rrbracket[\frac{1}{p}]\subseteq \O_{L_x}\llbracket q^{1/p^\infty}\rrbracket[\frac{1}{p}]$.
		\item On each connected component of $\XaGea{0}{n}$, there is at least one cusp $x$ at which the $q$-expansion of $f$ is already in ${\O_{L_x}}\llbracket q^{1/p^n}\rrbracket[\frac{1}{p}]\subseteq \O_{L_x}\llbracket q^{1/p^\infty}\rrbracket[\frac{1}{p}]$.
	\end{enumerate}
	The analogous statements for $\Xpaep\to \Xpae$ are also true.
\end{Proposition}
\begin{proof}
	It suffices to prove that (3) implies (1). Clearly $f$ is in $\O(\XaGea{0}{n} )$ if and only if $\tr_n(f)=f$. By Prop.\ \ref{p: q-expansion principle I,second version}, this can be checked on $q$-expansions on each component. By Prop.\ \ref{Proposition: effect of trace on q-expansions}, we have $\tr_n(f)=f$ if and only if the $q$-expansion at each $x$ is in $\O_L\llbracket q^{1/p^n}\rrbracket[\frac{1}{p}]$.
	
	The case of $\Xpaep$ is completely analogous, by replacing the normalised Tate traces of \cite{torsion} with those of \cite[\S 6.3]{AIP}. 
\end{proof}

\bibliography{/amd/nfs-12/export/home/heuer/texmf/tex/bibtex/bib/universal}
\bibliographystyle{alpha}
\end{document}